\documentclass[11pt]{article}

\usepackage{fullpage}
\usepackage{amssymb}

\usepackage{epsfig,amsthm}

\usepackage{amsmath}
\usepackage[utf8]{inputenc}

\usepackage{tikz-cd}

\usepackage{enumitem}

\usepackage{hyperref}

\hypersetup{
    colorlinks=true,
    linkcolor=blue,
    filecolor=magenta,      
    urlcolor=cyan,
    pdftitle={Sharelatex Example},
    bookmarks=true,
    pdfpagemode=FullScreen,
    }

\usepackage{subcaption}

\usepackage{multirow}

\newtheorem{theorem}{Theorem}[section]
\newtheorem{corollary}[theorem]{Corollary}
\newtheorem{lemma}[theorem]{Lemma}
\newtheorem{proposition}[theorem]{Proposition}

\newtheorem{definition}[theorem]{Definition}

\theoremstyle{definition}
\newtheorem*{remark}{Remark}
\theoremstyle{definition}

\newtheorem{thmy}{Theorem}
\newtheorem{thmc}{Corollary}
\def\Hom{\mathrm{Hom}}
\def\cW{\mathcal{W}}
\def\g{\mathfrak{g}}
\def\h{\mathfrak{h}}
\def\ZZ{\mathbb{Z}}
\def\cO{\mathcal{O}}
\def\cS{\mathcal{S}}
\def\ch{\mathrm{ch}}
\def\ow{\overline{w}}
\def\p{\mathfrak{p}}

\begin{document}

\title{Geometric and algebraic parameterizations for Dirac cohomology of simple modules in $\cO^\p$ and their applications}
\author{
Ho-Man Cheung \\
Department of Mathematics \\  
Hong Kong University of Science and Technology \\ 
{\tt hmcheungae@connect.ust.hk}
}

\date{}

\maketitle

\begin{abstract}
In this paper, we show that the Dirac cohomology $H_{D}(L(\lambda))$ of a simple highest weight module $L(\lambda)$ in $\cO^\p$ can be parameterized by a specific set of weights: a subset $\cW_I(\lambda)$ of the orbit of the Weyl group $W$ acting on $\lambda+\rho$. 
As an application, we show that any simple module in $\cO^\p$ is determined up to isomorphism by its Dirac cohomology.
We describe four parameterizations of $H_D(L(\lambda))$ when $\lambda$ is regular. Two of these parameterizations are geometric in terms of a partial ordering on the dual of the Cartan subalgebra and a generalization of strong linkage, respectively. Using these geometric parameterizations, we derive two algebraic parameterizations in terms of the multiplicities of the composition factors of a Verma module and the embeddings between Verma modules, respectively.
As an application, for Verma modules with regular infinitesimal character, we obtain an extended version of the Verma-BGG Theorem.
We also investigate Dirac cohomology of Kostant modules.
Using Dirac cohomology, we give a new proof of the simplicity criterion for Verma modules and describe a new simplicity criterion for parabolic Verma modules with regular infinitesimal character.
\end{abstract}

\setcounter{tocdepth}{2}
\tableofcontents

\section{Introduction}
\label{S:1}

In this paper, we concentrate on algebraic methods
for studying the representations of a finite dimensional semisimple Lie algebra $\g$ over the complex numbers $\mathbb{C}$ with universal enveloping algebra $U(\g)$. 
The category $U(\g)$-Mod of
all $U(\g)$-modules is too large to be understood algebraically. However, many interesting and important representations of Lie groups can be investigated within the framework of the BGG category $\cO$ introduced in the early 1970s by Bernstein, Gelfand, and Gelfand \cite{BGG2}.

The category $\cO$ is the category of all finitely generated, locally $\mathfrak{b}$-finite and $\h$-semisimple
$\g$-modules, where $\g$ is a finite dimensional complex semisimple Lie algebra with Cartan subalgebra $\h$ and Borel subalgebra $\mathfrak{b}$ containing $\h$. The Verma module corresponding to $\lambda\in\h^*$ is 
\[
M(\lambda) := U(\g)\otimes_{U(\mathfrak{b})} \mathbb{C}_\lambda,
\]
where $\mathbb{C}_\lambda$ is a simple $\mathfrak{b}$-module with weight $\lambda$. Denote by $L(\lambda)$ the unique
simple quotient of $M(\lambda)$. Kazhdan-Lusztig theory guarantees that the formal character of $L(\lambda)$ can be expressed in terms of formal characters of Verma modules and Kazhdan-Lusztig polynomials.

A relative version $\cO^\p$ of category $\cO$ is often needed in the study of representations of Lie groups. This is a subcategory of category $\cO$ defined by replacing $\h$
with a Levi subalgebra $\mathfrak{l}$ and $\mathfrak{b}$ with a parabolic subalgebra $\p$ containing $\mathfrak{b}$.

The parabolic Verma module corresponding to a subset of simple roots $I$ and an element $\lambda\in \Lambda_I^+$ (cf. Section \ref{S2.2}) is defined to be
\[
M_I(\lambda) := U(\g)\otimes_{U(\p)} F(\lambda),
\]
where $F(\lambda)$ is the finite dimensional simple $\mathfrak{l}$-module with highest weight $\lambda$. Deodhar \cite{VD} and Casian and Collingwood \cite{CLCD} developed a relative version of Kazhdan-Lusztig theory for the category $\cO^\p$.

Dirac cohomology is a new tool in representation theory that turns out to be an intrinsic invariant of irreducible unitary representations and more general admissible representations. 
Here is some relevant history of this construction.
In \cite{HP1}, Huang and Pandžić proved Vogan's conjecture which reveals a deep relationship between the infinitesimal character of a Harish-Chandra module $V$ and the infinitesimal characters appearing in its Dirac cohomology $H_D(V)$. 
%Kostant \cite{BK1} proved an analogous result in
%  $\cO^\p$. %(cf. Theorem \ref{Kostantthm}). 
%%
In \cite{BK1}, Kostant proved an analogous result in
$\cO^\p$, introduced the cubic Dirac operator and calculated the Dirac cohomology of finite dimensional modules in the equal rank case. 
In \cite{HX}, Huang and Xiao determined the Dirac cohomology of simple highest weight modules in terms of the sums of coefficients of the relative Kazhdan-Lusztig-Vogan polynomials.

Let $G$ be a connected real reductive group with maximal compact subgroup $K$ of the same rank as $G$. Let $\g$ and $\mathfrak{k}$ be the complexifications of the corresponding Lie algebras. 
In \cite{HPV}, Huang, Pandžić and Vogan identified the Dirac cohomology of certain unitary $(\g,K)$-modules $A_{\mathfrak{q}}(\lambda)$ with a geometric object, namely, the $\mathfrak{k}$-dominant part of
a face of the convex hull of the Weyl group orbit of the parameter $\lambda+\rho$.
In this paper, we give two similar geometric parameterizations of the Dirac cohomology of simple highest weight modules $L(\lambda)$ with regular infinitesimal character. We will also discuss two algebraic parameterizations of the Dirac cohomology of $L(\lambda)$ that are related to the Verma-BGG Theorem (cf. Theorem \ref{BGGVerma} and see \cite{BGG,BGG1}).

We use the rest of this introduction to sketch some of our main results.
Continue to let $\g$ be a finite dimensional complex semisimple Lie algebra with Cartan subalgebra $\h$.
Following \cite{HJ}, let $W=W_{\g}$ be the associated Weyl group and let $\Phi$ be its root system.
We write $\Phi^+$ for the set of positive roots in $\Phi$ and let $\Lambda_r = \ZZ \Phi$.
For $\eta \in \h^*$, 
define
\[
\Phi_{[\eta]}:=\{\alpha\in \Phi: \langle\eta,\alpha^\lor \rangle\in\mathbb{Z}\}
\quad\text{and}\quad
W_{[\eta]}:=\{w\in W: w\cdot\eta-\eta\in \Lambda_r\}.
\]
Fix a subset of simple roots $I$ and let $W_I$ be the corresponding standard parabolic subgroup of $W$, with longest element $w_I$ and root system $\Phi_I\subseteq \Phi$. Let $\Phi_I^+:=\Phi_I\cap\Phi^+$.
Define
\[
\Lambda^+_I := \{\nu \in \h^*  :  \langle\nu,\alpha^\lor\rangle \in \mathbb{Z}^{\ge 0} \ \text{for all }\alpha \in \Phi^+_I\}.
\]
Any simple module $V\in\cO^\p$ is isomorphic to $L(\lambda)$ for some $\lambda\in\Lambda_I^+$ and this implies $H_D(V)\cong H_D(L(\lambda))$ as an $\mathfrak{l}$-module (cf. Theorem \ref{determine}). We can therefore concentrate on $H_D(L(\lambda))$.
%Consider $L(\lambda)\in\cO^\p$,  we get $\lambda \in \Lambda^+_I$. 
Since $\lambda \in \Lambda^+_I$, the subgroup $W_I$ is then contained in $W_{[\lambda]}$ (cf. Remark following Corollary \ref{category}), and we define 
\[
{}^IW_{[\lambda]} := \{w\in W_{[\lambda]}: w<s_\alpha w \ \text{for all }\alpha\in I\},
\]
where $<$ is the Bruhat ordering on $W$. 

Denote by $\Delta_{[\lambda]}$ the simple system corresponding to the positive system $\Phi_{[\lambda]}\cap\Phi^+$ in $\Phi_{[\lambda]}$. The orbit $W_{[\lambda]}\cdot\lambda$
contains a unique $\mu\in\mathfrak{h}^*$ that is \emph{antidominant} in the sense that 
$\langle \mu+\rho,\alpha^{\lor}\rangle\not\in\mathbb{Z}^{>0}$ for all $\alpha\in \Phi^+$, where $\rho = \frac{1}{2} \sum_{\alpha \in \Phi^+} \alpha$.
The set of \emph{singular simple roots associated to $\mu$} in $\Delta_{[\lambda]}$ is defined by 
\[
{\Sigma_\mu} : = \{\alpha\in \Delta_{[\lambda]}: \langle\mu+\rho,\alpha^\lor\rangle=0\}.
\]
The subgroup $W_{\Sigma_\mu} := \{w\in W: w(\mu+\rho)=\mu+\rho\}\subseteq W_{[\lambda]}$ is then the isotropy group of $\mu$. Let
\[
{}^IW_{[\lambda]}^{\Sigma_\mu} : = \{w\in {}^IW_{[\lambda]}: w<ws_\alpha\in {}^IW_{[\lambda]}\ \text{for all }\alpha\in {\Sigma_\mu}\},
\]
where $<$ is the Bruhat ordering on $W$. 

Following \cite{BBMH,HX}, we define
the \emph{relative Kazhdan-Lusztig-Vogan polynomial} associated to $\lambda$ of $x,w\in {}^IW_{[\lambda]}^{\Sigma_\mu}$
to be
\[
{}^IP_{x,w}^{\Sigma_\mu}(q) : = \sum_{i\ge 0}q^{\frac{\ell_{[\lambda]}(x,w)-i}{2}}\dim \mathrm{Ext}^i_{\cO^\p}\left(M_I(w_Ix\cdot\mu),L(w_Iw\cdot\mu)\right),\]
where $\ell_{[\lambda]}$ is the length function on $W_{[\lambda]}$ and $\ell_{[\lambda]}(x,w):=\ell_{[\lambda]}(w)-\ell_{[\lambda]}(x)$.
Note that this is always a polynomial in $q$ (cf. Theorem \ref{KLVparaKL}).
%Results in \cite{BBMH,HX} show that this is always a polynomial in $q$. 
%Then, the properties of Kazhdan-Lusztig-Vogan polynomials (cf. Theorem \ref{KLpoly}) implies the key Corollary \ref{keyequiv}.

Definitions of $\mathfrak{l}$, $\mathfrak{u}$, $\rho(\mathfrak{u})$ and $\rho_\mathfrak{l}$ can be found in Section~\ref{2.1-sect}, and the definition of $\cO^\p_\mu$ can be found in the remark following Proposition \ref{antidominant}.
Our main results rely on the following theorem:

%To find out $\cW_I(\lambda)$ explicitly, we need to make use of the following theorem:

\begin{theorem}[{See \cite[Theorem 6.16]{HX}}]
\label{maindecom}
Let $L(\lambda)$ be a simple highest weight module in $\cO^\p_\mu$ of weight $\lambda$.
Let $\ow \in  {}^IW_{[\lambda]}^{\Sigma_\mu}$ be the unique element such that 
$\lambda=w_I\ow\cdot \mu$ (cf. Remark following Lemma \ref{simplemod}). Then, one has an $\mathfrak{l}$-module decomposition
\[
H_D(L(\lambda))\cong\bigoplus_{x\in {}^I W_{[\lambda]}^{\Sigma_\mu}} {}^IP_{x,\ow}^{\Sigma_\mu}(1)F(w_Ix\cdot \mu+\rho(\mathfrak{u})).
\]
\end{theorem}

\begin{definition}\label{WIlambda}
	For $I\subseteq \Delta$ and $\lambda\in\Lambda_I^+$. Let \[
	\cW_I(\lambda):=\left\{w_Ix\cdot\mu+\rho(\mathfrak{u}): x\in {}^IW_{[\lambda]}^{\Sigma_\mu}, {}^IP_{x,\ow}^{\Sigma_\mu}(1)\neq 0\right\}+\rho_\mathfrak{l}.
	\]
\end{definition}

\begin{remark}
	Let $\cW(\lambda):=\cW_\emptyset(\lambda)$ as a convention.
\end{remark}
%\begin{remark}
%By Theorem \ref{general1to1'}, we say $\cW_I(\lambda)$ is the parameterization of $H_D(L(\lambda))$. Viewing $\cW_I(\lambda)$ in four different ways gives us four parameterizations of $H_D(L(\lambda))$ when $\lambda$ is regular.
%\end{remark}

We can now state our main results.
First, it turns out that $\cW_I(\lambda)$ is a subset of four sets that are defined in terms of a generalization of strong linkage, the embeddings between Verma modules, the multiplicities of the composition factors of a Verma module, and a partial ordering on the dual of the Cartan subalgebra, respectively.
More precisely, we will prove the following:
\begin{thmy} [Theorem \ref{generalcase}]\label{generalcase'}
	Let $\lambda\in\Lambda_I^+$, $\mathcal{S}_{[\lambda]}(\lambda):=\{\nu\in\h^*: \nu\uparrow_{[\lambda]} \lambda\}$, $C_\mathfrak{l}:=\{\nu\in\h^*:\langle\nu,\alpha^\lor\rangle\ge 0, \ \forall \alpha\in I\}$ and $\mathcal{L}_{\eta}
	:=\{\nu\in\h^*: \nu\le \eta\}$. Then 
	\begin{align*}
	\cW_I(\lambda)
	&\subseteq \left(\mathcal{S}_{[\lambda]}(\lambda)+\rho\right)\cap C_\mathfrak{l}\\
	&\subseteq \{\nu\in\Lambda_I^+-\rho: M(\nu)\hookrightarrow M(\lambda)\}+\rho\\
	&=\{\nu\in\Lambda_I^+-\rho: [M(\lambda),L(\nu)]\neq 0\}+\rho\\
	&\subseteq W_{[\lambda]}(\lambda+\rho)
	\cap \mathcal{L}_{\lambda+\rho}\cap C_\mathfrak{l},
	\end{align*}
\end{thmy}

We will show that the Dirac cohomology of simple highest weight modules $L(\lambda)$ is parameterized by $\cW_I(\lambda)$, in the sense of the following theorem.

\begin{thmy}[Theorem \ref{general1to1}]\label{general1to1'}
	Let $\lambda,\eta\in \Lambda_I^+$. The following statements are then equivalent:
	\begin{enumerate}
		\item $\lambda=\eta$.
		%\item $H_D(L(\lambda))=H_D(L(\eta))$.
		\item $H_D(L(\lambda))\cong H_D(L(\eta))$ as an $\mathfrak{l}$-module.
		\item $\cW_I(\lambda)=\cW_I(\eta)$. 
	\end{enumerate}
\end{thmy}

Using Theorem \ref{general1to1'}, we show that a simple module in $\cO^\p$ is determined up to isomorphism by its Dirac cohomology.
\begin{thmy}[Theorem \ref{determine'}]\label{determine}
	Suppose $V$ and $W$ are simple modules in the category $\cO^\p$. Then $V\cong W$ as an $\mathfrak{g}$-module 
	if and only if $H_D(V)\cong H_D(W)$ as an $\mathfrak{l}$-module.
\end{thmy}

When $\lambda\in\Lambda_I^+$ is regular, it turns out that $H_D(L(\lambda))$ has two geometric parameterizations.
\begin{thmy} [Theorem \ref{geompara2}]\label{geompara2'}
	Let $\mathcal{R}$ be the set of regular weights in $\h^*$. Let $\lambda\in\Lambda_I^+\cap \mathcal{R}$.
	Then \begin{align*}
	\cW_I(\lambda)
	=W_{[\lambda]}(\lambda+\rho)\cap\mathcal{L}_{\lambda+\rho}\cap C_\mathfrak{l}
	=\left(\mathcal{S}_{[\lambda]}(\lambda)+\rho\right)\cap C_\mathfrak{l}.
	\end{align*}
\end{thmy}

When $\lambda\in\Lambda_I^+$ is regular, there are two algebraic parameterizations of $H_D(L(\lambda))$.
\begin{thmy}[Theorem \ref{parameterization'}]\label{parameterization}
	Let $\lambda\in\Lambda_I^+\cap\mathcal{R}$. Then
	\begin{align*}
	\cW_I(\lambda)
	=\{\nu\in\Lambda_I^+: [M(\lambda),L(\nu)]\neq 0\}+\rho
	=\{\nu\in\Lambda_I^+: M(\nu)\hookrightarrow M(\lambda)\}+\rho.
	\end{align*}
\end{thmy}

As an application, for Verma modules with regular infinitesimal character, we use Theorems~\ref{generalcase'}, \ref{geompara2'} and \ref{parameterization} to obtain an extended version of the Verma-BGG Theorem; see Theorem~\ref{5equiv}. 
%Motivated by Theorem~\ref{5equiv}, we get a theorem concerning Dirac cohomology of Kostant modules; see Theorem~\ref{kostant-prop}.
%
%\begin{theorem}\label{5equiv}
%Given $\lambda\in \Lambda_I^+\cap\mathcal{R}$ and $\eta\in\Lambda_I^+$, the following statements are equivalent:
%\begin{enumerate}
%    \item $[M(\lambda),L(\eta)]\neq 0$.
%    \item $M(\eta)\hookrightarrow M(\lambda)$.
%    \item $\eta$ is strongly linked to $\lambda$.
%    \item $\eta$ is $[\lambda]$-strongly linked to $\lambda$.
%    \item $\eta\le \lambda$ and $\eta=w\cdot\lambda$ for some $w\in W_{[\lambda]}$.
%    \item $\cW_I(\eta)
%    \subseteq\cW_I(\lambda)$.
%\end{enumerate}
%Note that $\mu$ is the antidominant weight in $W_{[\lambda]}\cdot\lambda$.
%\end{theorem}
%
%
%
%We then investigate the Dirac cohomology of Kostant modules. 
%\begin{proposition}
%	Given $\lambda,\lambda' \in \Lambda_I^+$. Suppose $\lambda$ or $\lambda'$ is regular and $L(\lambda)$ is a Kostant module.
%	Then
%	the following statements are equivalent:
%	\begin{enumerate}
%		\item $H_D(L(\lambda))\le H_D(L(\lambda'))$ as $\mathfrak{l}$-modules.
%		\item $\lambda\uparrow_{[\lambda']}\lambda'$.
%		\item $\lambda\uparrow \lambda'$.
%		\item $M(\lambda)\hookrightarrow M(\lambda')$.
%		\item $[M(\lambda'), L(\lambda)]\neq 0$.
%		\item $\lambda\le \lambda'$ and $\lambda=\tilde{w}\cdot\lambda'$ for some $\tilde{w}\in W_{[\lambda']}$.
%	\end{enumerate}
%	Note that $(-2\rho_{[\lambda]})$ is the antidominant weight in $W_{[\lambda']}\cdot\lambda'$.
%\end{proposition}
%
Using Dirac cohomology, we are also able to give a new proof of the following simplicity criterion for Verma modules, which also appears as  \cite[Theorem 4.8]{HJ}.

\begin{thmy} [Theorem \ref{Verma'}] \label{Verma}
Let $\lambda\in\h^*$. Then 
$M(\lambda)\cong L(\lambda)$ as an $\mathfrak{g}$-module if and only if $\lambda $ is an antidominant weight.
\end{thmy}

By similar methods, we are able to derive a new simplicity criterion for parabolic Verma modules with regular infinitesimal character.
\begin{thmy} [Theorem \ref{ParaVerma'}]\label{ParaVerma}
Let $\lambda\in\Lambda_I^+\cap \mathcal{R}$.
Then $M_I(\lambda)\cong L(\lambda)$ as an $\mathfrak{g}$-module if and only if $\lambda=w_I\cdot\nu$ for some antidominant weight $\nu$.
\end{thmy}

Comparing this with Jantzen’s simplicity criterion for parabolic Verma modules with regular infinitesimal character, we derive the following non-trivial corollary.

\setcounter{thmc}{7}
\begin{thmc} [Corollary \ref{nontrivial'}]\label{nontrivial}
Let $\lambda\in\Lambda_I^+\cap \mathcal{R}$. Then 
$\Psi^+_\lambda = \emptyset$ if and only if $\lambda=w_I\cdot\nu$ for some antidominant weight $\nu$,
where 
$\Psi^+_\lambda := \{\beta \in \Psi^+ : \langle \lambda + \rho, \beta^\lor\rangle \in \mathbb{Z}^{>0}\}$ and $\Psi := \Phi \backslash \Phi_I$.
\end{thmc}

This paper is organized as follows. 

In Section~\ref{S:2}, we recall the definitions of the category $\cO^\p$ and Dirac cohomology. 

In Section~\ref{S:3}, we turn to Kazhdan-Lusztig-Vogan theory.
% including Kazhdan-Lusztig polynomials, parabolic Kazhdan-Lusztig polynomials and relative Kazhdan-Lusztig-Vogan polynomials. 
%We also present some of their basic properties. 
We derive two general identities relating relative Kazhdan-Lusztig-Vogan polynomials and parabolic Kazhdan-Lusztig polynomials. As an application, we 
determine when the sum of coefficients of a relative Kazhdan-Lusztig-Vogan polynomial associated to $\lambda$ is nonzero, in the case when $\lambda\in\Lambda^+_I$ is regular.

In Section~\ref{S:4}, we prove Theorems~\ref{generalcase'}, \ref{geompara2'} and \ref{parameterization}.
%give four parameterizations of the Dirac cohomology of simple highest weight modules with regular infinitesimal character in category $\cO^\p$.
%Two of these parameterizations are geometric in terms of a generalization of strong linkage and a partial ordering on weights respectively. 
%Then we use the two geometric parameterizations to give two algebraic parameterizations in terms of the multiplicities of composition factors of a Verma module and the embeddings between Verma modules respectively. 
As an application, for Verma modules with regular infinitesimal character, we obtain an extended version of the Verma-BGG Theorem.

Section~\ref{S:5} contains our results on the Dirac cohomology of Kostant modules.

In Section~\ref{S:6}, finally, we use Dirac cohomology to prove Theorems~\ref{Verma} and \ref{ParaVerma}.
%give a new proof of 
%the simplicity criterion for Verma modules, and we derive a new simplicity criterion for parabolic Verma modules with
%regular infinitesimal character. 
%Combined with Jantzen’s simplicity criterion, this implies a nontrivial result about any weight $\nu$ such that the parabolic Verma module
%$M_I(\nu)$ is simple.

\subsection*{Acknowledgements}

I thank my thesis advisor Jing-Song Huang for guidance and Eric Marberg for useful comments.

\section{Notation and preliminaries}
\label{S:2}

\subsection{Notation}\label{2.1-sect}

Throughout this paper, we adopt the following notation.
Denote by $\g$ a finite dimensional complex semisimple Lie algebra and let $\h$ be a Cartan subalgebra of $\g$. 
Let $\Phi$
be the root system of $(\g,\h)$, write $W$ for the corresponding Weyl group of $\Phi$, and denote by $\g_\alpha$ the root subspace of $\g$ corresponding
to a root $\alpha$.
We fix a choice of positive roots $\Phi^+$, and let $\Delta$ be the corresponding subset of simple roots in $\Phi^+$. Note that each subset $I\subseteq\Delta$ generates a root system $\Phi_I\subseteq\Phi$, with positive roots $\Phi_I^+:=\Phi_I\cap \Phi^+$.
There are a number of subalgebras of $\g$ associated with the root system $\Phi_I$. By \cite[\S 9.1]{HJ}, the Lie algebra 
\[
\mathfrak{p}_I:=\h\oplus\sum_{\alpha\in\Phi_I\cup \left(\Phi^+\backslash\Phi_I^+\right)}\g_\alpha
\]
is a standard parabolic subalgebra of $\g$, the Lie algebra
\[
\mathfrak{l}_I:=\h\oplus\sum_{\alpha\in\Phi_I}\g_\alpha
\]
is the Levi subalgebra of $\p_I$ and the Lie algebras \[
\mathfrak{u}_I:=\sum_{\alpha\in\Phi^+\backslash\Phi_I^+}\g_\alpha
\qquad\text{and}\qquad \overline{\mathfrak{u}_I}:=\sum_{\alpha\in\Phi^+\backslash\Phi_I^+}\g_{-\alpha}
\]
are the nilradical of $\p_I$ and its dual space with respect to the Killing form $B$ of $\g$ such that $\p_I=\mathfrak{l}_I\oplus \mathfrak{u}_I$ and $\g=\overline{\mathfrak{u}_I}\oplus \p_I$.
We note that once $I$ is fixed, there is little use for other subsets of $\Delta$. Therefore, we omit the subscript if a subalgebra is obviously associated to $I$.
Let
\[
\rho:=\dfrac{1}{2}\sum_{\alpha\in\Phi^+} \alpha, 
\qquad 
\rho_\mathfrak{l}:=\dfrac{1}{2}\sum_{\alpha\in\Phi^+_I} \alpha, 
\qquad\text{and}\qquad
\rho(\mathfrak{u}):=\dfrac{1}{2}\sum_{\alpha\in\Phi^+\backslash \Phi^+_{I}} \alpha.
\]
The set of integral weights in $\h^*$ is
\[
\Lambda := \{\nu \in \h^*  :  \langle\nu,\alpha^\lor\rangle \in \mathbb{Z} \ \text{for all }\alpha \in \Phi\},
\]
%The set of $\Phi^+_I$-dominant integral weights in $\h^*$ is
%\[
%\Lambda^+_I := \{\lambda \in \h^*  :  \langle\lambda,\alpha^\lor\rangle \in \mathbb{Z}^{\ge 0} \ \text{for all }\alpha \in \Phi^+_I\},
%\]
and the set of dominant integral weights in $\h^*$ is
\[
\Lambda^+ := \{\nu \in \h^* : \langle \nu, \alpha^\lor\rangle \in \mathbb{Z}^{\ge 0} \ \text{for all }\alpha \in \Phi^+\},
\]
where $\langle\cdot, \cdot \rangle$ is the bilinear form on $\h^*$ induced from the Killing form $B$ of $\g$ and $\alpha^\lor:=\frac{2\alpha}{\langle\alpha,\alpha\rangle}$.
%Let $F(\lambda)$ be the finite dimensional simple $\mathfrak{l}$-module with highest weight $\lambda$.
Denoted by $U(\g)$ the universal enveloping algebra of $\g$ with centre
$Z(\g)$.
%Let $\mathfrak{s} = \mathfrak{u} \oplus \overline{\mathfrak{u}}$. 
%Denote by $C(\mathfrak{s})$ the Clifford algebra of $\mathfrak{s}$ with
%\[
%uu' + u'u = -2B(u, u')
%\]
%for all $u, u' \in \mathfrak{s}$. 
%Finally, let $\cS$ be a spin module of $C(\mathfrak{s})$.

%\begin{example}
%TODO:
%Maybe put in an example of all this notation: e.g., type C roots with system $I$ corresponding to type A subsystem.
%\end{example}

\subsection{Preliminaries on Category $\cO^\p$}\label{S2.2}

In this section, we recall the definition and basic properties of category $\cO^\p$.
Continue to let $\g$ be a finite dimensional complex semisimple Lie algebra with Cartan subalgebra $\h$. Fix a Borel subalgebra $\mathfrak{b}$ containing $\h$ and a parabolic subalgebra $\p$ containing $\mathfrak{b}$. Let $I \subseteq\Delta$ be the subset of simple roots corresponding to $\p$.
Denote by $\Phi_I$ the subsystem generated by $I$, i.e.,
$\Phi_I:=\Phi\cap\sum_{\alpha\in I}\mathbb{Z}\alpha$, and let $\Phi_I^+:=\Phi_I\cap\Phi^+$.
Let
$\mathfrak{l} := \h\oplus\sum_{\alpha\in \Phi_I}\g_\alpha$ be the associated Levi subalgebra. Denote by $\mathfrak{u}$ the nilradical of $\p$ and let $\overline{\mathfrak{u}}$ be the dual space of $\mathfrak{u}$ with respect to the Killing form $B$ of $\g$.

\begin{definition}[{See \cite[Definition 2.1]{HX}}]
The category $\cO^\p$ is the full subcategory of $U(\g)$-Mod whose objects $M$ satisfy the following conditions:
\begin{enumerate}
    \item $M$ is a finitely generated $U(\g)$-module.
    \item $M$ is a direct sum of finite dimensional simple $U(\mathfrak{l})$-modules.
    \item $M$ is locally finite as a $U(\p)$-module.
\end{enumerate}
\end{definition}
The set of \emph{$\Phi^+_I$-dominant integral weights} in $\h^*$ is
\[
\Lambda^+_I := \{\nu \in \h^* : \langle \nu, \alpha^\lor\rangle \in \mathbb{Z}^{\ge 0} \ \text{for all }\alpha \in \Phi^+_I\},
\]
where $\langle \cdot, \cdot\rangle$ is again the bilinear form on $\h^*$ induced from the Killing form of $\g$. % and $\alpha^\lor = \dfrac{2\alpha}{\langle \alpha,\alpha\rangle}$.
Let $F(\lambda)$ be the finite dimensional simple $\mathfrak{l}$-module with highest weight $\lambda$. We have $\lambda\in\Lambda^+_I$ by a result in \cite[\S 9.2]{HJ}.
Note that $F(\lambda)$ is a $\p$-module on which $\mathfrak{u}$ acts trivially.
The \emph{parabolic Verma module} with highest weight $\lambda$ is the induced module
\[
M_I(\lambda) := U(\g)\otimes _{U(\p)} F(\lambda).
\]
When $\p = \mathfrak{b}$, we obtain the ordinary Verma module $M(\lambda)$. By results in \cite[\S 9.4]{HJ}, $M_I(\lambda)$ is a quotient of $M(\lambda)$ and $L(\lambda)$ is the unique simple quotient of both $M_I(\lambda)$ and $M(\lambda)$. 
Furthermore, since every nonzero module in $\cO^\p$ has at least one nonzero maximal vector, Proposition \ref{9.3} implies that every simple module in $\cO^\p$ is isomorphic to $L(\lambda)$ for some $\lambda \in \Lambda^+_I$.
Let $Z(\g)$ be the centre of $U(\g)$ and let $\chi_\lambda$ be an algebra homomorphism
$Z(\g) \to \mathbb{C}$ such that
\[
z \cdot v = \chi_\lambda(z)v
\]
for all $z \in Z(\g)$, $v \in M(\lambda)$. Then $M_I(\lambda)$ and its subquotients (including $L(\lambda)$) have
the same infinitesimal character $\chi_\lambda$.
As is shown in \cite[\S1.2]{HJ}, every nonzero module $M \in \cO^\p$ has a finite filtration with nonzero quotients, each of which is a highest weight module in $\cO^\p$. 
Thus the action of $Z(\g)$ on $M$ is finite. Let
\[
M^\chi := \{v \in M  :  \text{for each $z \in Z(\g)$,}\ (z - \chi(z))^n\cdot v = 0  \ \text{for some $n \in \mathbb{Z}_{>0}$ depending on $z$}\}.
\]
Then $z - \chi(z)$ acts locally nilpotently on $M^\chi$ for all $z \in Z(\g)$ and $M^\chi$ is a $U(\g)$-submodule of $M$.
Denote by $\cO^\p_\chi$ the full subcategory of $\cO^\p$ whose objects are of the form $M^\chi$. We then have the following direct sum decomposition
\[
\cO^\p=\bigoplus_{\chi}\cO^\p_\chi,
\]
where $\chi = \chi_\lambda$ for some $\lambda \in\h^*$.
%
%Let $W$ be the Weyl group associated with the root system $\Phi$. 
The dot action
of the Weyl group $W$ on $\h^*$ is given by
\[
w \cdot \lambda := w(\lambda + \rho) - \rho
\]
for all $\lambda \in \h^*$. Then $\chi_\lambda = \chi_\mu$ if and only if $\lambda \in W\cdot \mu$ by the Harish-Chandra
isomorphism $Z(\g) \to S(\h)^W$.
For any $U(\g)$-module $M$ and for any $\lambda \in\h^*$, let
\[
M_\lambda := \{v \in M  :  h\cdot v = \lambda(h)v \ \text{for all $h \in \h$}\}
\]
be a weight space relative to the action of $\h$. 
If $M_\lambda \neq 0$, then $\lambda$ is called a \emph{weight} of $M$. Since $M(\lambda)_\lambda\neq 0$ for all $\lambda\in\h^*$, any element of $\h^*$ is called a \emph{weight}. The \emph{multiplicity} of $\lambda$ in $M$ is defined to be $\dim M_\lambda$. 
In general, any module in the category of weight modules having finite dimensional weight spaces can be assigned a formal character 
\[
\ch(M) := \sum_{\lambda\in \h^*}
(\dim M_\lambda)e(\lambda).
\]
Let $\Gamma$ be the set of all $\mathbb{Z}^{\ge 0}$-linear combinations of simple roots in $\Delta$. 
Denote by $\mathcal{X}$ the additive group of functions
$f : \h^* \to Z$
whose support lies in a finite union of sets of the form
$\lambda - \Gamma$ for $\lambda \in \h^*$. 
The convolution product on $\mathcal{X}$ is given by
$(f * g)(\lambda) := \sum_{\mu+\nu=\lambda}
f(\mu)g(\nu)$.
Denote by $e(\lambda)$ the function in $\mathcal{X}$ which takes value $1$ at $\lambda$ and value $0$ at $\mu \neq \lambda$,
so that $e(\lambda) * e(\mu) = e(\lambda + \mu)$. 
It is easy to check that $\mathcal{X}$ is a commutative ring
under convolution whose multiplicative identity is $e(0)$.
All modules in $\cO^\p$ and all finite dimensional semisimple $\h$-modules have characters in $\mathcal{X}$.
\begin{proposition}[{\cite[Proposition 9.3]{HJ}}]
	\label{9.3}
Fix $\p = \p_I$ as above.
\begin{enumerate}
    \item $M \in \cO$ lies in $\cO^\p$ if and only if $M$ satisfies the equivalent conditions in \cite[Lemma 9.3]{HJ}.
    
    \item $\cO^\p$ is closed under duality in $\cO$.
    \item $\cO^\p$ is closed under direct sums, submodules, quotients, and extensions in $\cO$, as well as tensoring with finite dimensional $U(\g)$-modules.
    \item If $M \in \cO^\p$ decomposes as $M =\bigoplus_{\chi} M^\chi$ with $M^\chi$ in $\cO_\chi$, then each $M^\chi$ lies in $\cO^\p$; this gives a decomposition $\cO^\p = \bigoplus_\chi \cO_\chi^\p$. As a result, translation functors preserve $\cO^\p$.
    \item If the simple module $L(\lambda)$ lies in $\cO^\p$, then $\lambda \in \Lambda^+_I$.
\end{enumerate} 
\end{proposition}

\begin{theorem} [{See \cite[Theorem 9.4]{HJ}}] 
	\label{9.4}
Let $\lambda \in \Lambda_I^+$.
\begin{enumerate}
    \item The module $M_I(\lambda)$ and its quotient $L(\lambda)$ both belong to $\cO^\p$.
    \item There is an exact sequence 
    $
    \bigoplus_{\alpha\in I}M(s_\alpha \cdot \lambda) \to M(\lambda) \to M_I(\lambda) \to 0.
    $
\end{enumerate}

\end{theorem}

\subsection{Preliminaries on Dirac cohomology}

We recall the definitions and basic properties of Dirac cohomology associated to the Kostant cubic Dirac operator. Let $\mathfrak{r}$ be a reductive Lie subalgebra of the finite dimensional complex semisimple Lie algebra $\g$ and let $B$ be the Killing form of $\g$.
Suppose that the restriction $B|_\mathfrak{r}$ of B on $\mathfrak{r}$ is non-degenerate. Let $\g = \mathfrak{r} \oplus\mathfrak{s}$ be
the orthogonal decomposition with respect to $B$. It is easy to check that the
restriction $B|_\mathfrak{s}$ of $B$ on $\mathfrak{s}$ is also non-degenerate.
Denote by $C(\mathfrak{s})$ the Clifford algebra of $\mathfrak{s}$ with \[
uu' + u'u = -2B(u, u')
\]
for all $u, u' \in \mathfrak{s}$. 
Let $\{Z_1, Z_2, \cdots, Z_m\}$ be an orthonormal basis of $\mathfrak{s}$. In \cite{BK}, Kostant introduced the \emph{cubic Dirac operator} $D$ defined by 
\[
D := \sum_{1\le i\le m} Z_i \otimes Z_i + 1 \otimes v \in U(\g) \otimes C(\mathfrak{s}), 
\]
where $v \in C(\mathfrak{s})$ is the image of the fundamental $3$-form $\omega \in \bigwedge^3(\mathfrak{s}^*)$ such that \[
\omega(X, Y, Z) := \frac{1}{2} B(X, [Y, Z]),
\]
under the Chevalley map $\bigwedge(\mathfrak{s}^*) \to C(\mathfrak{s})$ and the identification of $\mathfrak{s}^*$ with $\mathfrak{s}$ via the Killing form $B$.
Explicitly, \[
v= \frac{1}{2}\sum_{1\le i,j,k\le m} B([Z_i, Z_j],Z_k)Z_iZ_jZ_k.
\]
%
%Let $\alpha : \mathfrak{r} \to C(\mathfrak{s})$ be the composition of the adjoint map 
%$\text{ad} : \mathfrak{r} \to \mathfrak{so}(\mathfrak{s})$ and the embedding of $\mathfrak{so}(\mathfrak{s})$ into $C(\mathfrak{s})$ using the identification $\mathfrak{so}(\mathfrak{s}) \cong \bigwedge^2\mathfrak{s}$, then $\alpha$ is given by \[
%\alpha(X) = \dfrac{1}{2} \sum_{1\le i<j\le m} B(X, [Z_i, Z_j])Z_iZ_j
%\]
%for all $X \in \mathfrak{r}$. 
%%
%The diagonal embedding of $\mathfrak{r}$ into $U(\g) \otimes C(\mathfrak{s})$ is given by \[
%\text{Diag}: X \to X_\Delta = X \otimes 1 + 1 \otimes\alpha(X)
%\]
%which extends to $U(\mathfrak{r})$. Let $\mathfrak{r}_\Delta$ be the image of $\mathfrak{r}$ under Diag. The image of $U(\mathfrak{r})$ is then the universal enveloping algebra $U(\mathfrak{r}_\Delta)$ of $\mathfrak{r}_\Delta$. Denote by $\Omega_\g$ and $\Omega_\mathfrak{r}$ the Casimir elements of $\g$ and $\mathfrak{r}$, respectively. Let $\Omega_{\mathfrak{r}_\Delta}$ be the image of $\Omega_\mathfrak{r}$ under Diag.
%%
%Fix a Cartan subalgebra $\h_\mathfrak{r}$ of $\mathfrak{r}$ which is contained in $\h$. 
%%
%In \cite{BK}, Kostant proved that
%\[
%D^2 = \Omega_{\g}\otimes 1 -\Omega_{\mathfrak{r}_\Delta}+ (\lVert\rho \rVert^2 - \lVert\rho_\mathfrak{r} \rVert^2)1 \otimes 1,
%\]
%where $\rho_\mathfrak{r}$ is the half sum of positive roots for $(\mathfrak{r}, \h_\mathfrak{r})$.
\begin{definition}[{See \cite[Definition 3.2]{HX}}]
Let $\cS$ be a spin module of $C(\mathfrak{s})$. Consider the action of $D$ on
$V \otimes \cS$ 
\[
D: V \otimes \cS\to V \otimes \cS
\]
with $\g$ acting on $V$ and $C(\mathfrak{s})$ acting on $\cS$. 
The
\emph{Dirac cohomology} of $V$ is defined to be the $\mathfrak{r}$-module
\[
H_D(V) :=\dfrac{\ker(D)}{\ker(D) \cap \mathrm{Im}(D)}.
\] 
\end{definition}
Following \cite{HX}, we will only consider the case $\mathfrak{r} = \mathfrak{l}$ and $\mathfrak{s} = \mathfrak{u} + \overline{\mathfrak{u}}$ for the rest of the paper. In particular, we have the following proposition:

\begin{proposition}[{\cite[Proposition 3.7]{HX}}]
 \label{subset} 
Suppose that $V$ is in $\cO^\p_{\chi_\lambda}$. Then the Dirac cohomology $H_D(V)$ is a completely reducible finite dimensional $\mathfrak{l}$-module. Moreover, if the finite dimensional simple $\mathfrak{l}$-module $F(\eta)$ is contained in $H_D(V)$, then
$\eta + \rho_\mathfrak{l} = w(\lambda + \rho)$ for some $w \in W$.
\end{proposition}

\begin{remark} By Proposition \ref{subset}, it follows that $\cW_I(\lambda)$ is a subset of $W(\lambda+\rho)$ since $L(\lambda)\in\cO_{\chi_\lambda}^\p$.
\end{remark}

\section{Kazhdan-Lusztig-Vogan Theory}
\label{S:3}

\subsection{Kazhdan-Lusztig polynomials}\label{kl-sect}

In \cite{KL}, Kazhdan and Lusztig define a family of polynomials $\{P_{x,w}: x,w\in W\}$ in a variable $q$. 
We briefly recall some of their well-known properties.
Let $S=\{s_\alpha:\alpha\in\Delta\}$ 
so that $(W,S)$ is a Coxeter system. Let $\ell$ be the associated length function on $W$ and let $\le$ denote the Bruhat ordering on $W$. 
The following then holds:
\begin{enumerate}[label=(\alph*)]
    \item $P_{x,w}(q)=0$ unless $x\le w$.
    \item $P_{w,w}(q)=1$.
    \item If $x<w$ then deg$(P_{x,w})\le \frac{1}{2}\left(\ell(w)-\ell(x)-1\right)$. %(Here $\ell$ is the length function on $W$.)

    \item Let $\prec$ be the relation on $W$ with
    $x\prec w$ if and only if $x<w$ and deg$(P_{x,w})=\tfrac{1}{2}\left(\ell(w)-\ell(x)-1\right)$.
    Define $\mu(x,w)$ to be the coefficient of $P_{x,w}$ of degree $\tfrac{1}{2}\left(\ell(w)-\ell(x)-1\right)$. %The $P_{x,w}$ also have the following property:
    Suppose $x\le w$, $s\in S$, and $ws<w$. Then 
     \[
    P_{x,w}(q)=q^{1-a}P_{xs,ws}(q)+q^aP_{x,ws}(q)-\sum_{\substack{ z \\
                  x\le z\prec ws\\
                 zs<z
                 }} \mu(z,ws)q^{\frac{1}{2}(\ell(w)-\ell(z))}P_{x,z}(q),
    \]
    where $a=1$ if $xs<x$ and $a=0$ if $xs>x$.
%    
%    $\begin{cases}
%    a=1& \ \text{if } xs<x\\
%    a=0& \ \text{if } xs>x\\
%    \end{cases}$.
    These identities allow one to compute $P_{x,w}$ by induction on $\ell(w)$.

\item For each $x\le w$ in $W$, $P_{x,w}$ is a polynomial in $q$ with constant term $1$. 
%(A proof for property (e) can be found in \cite{KL}.)

\item It holds that $P_{x,w}=P_{x^{-1},w^{-1}}$. %(A proof for property (f) can be found in !!!!!!)

\item Let $M_w$ be the Verma module with highest weight $-w(\rho)-\rho$ and let $L_w$ be its unique simple quotient.
The \emph{Kazhdan-Lusztig Conjecture} (which is now a theorem \cite{BB,BJLKM}) asserts that 
\[
\ch(L_{w})=\sum_{x\le w}\varepsilon_x\varepsilon_w P_{x,w}(1) \ch(M_x), \ \text{where } \varepsilon_w:=(-1)^{\ell(w)}.
\]

\item Vogan proves in \cite{DV} that the Kazhdan-Lusztig Conjecture is equivalent to the formula
    \[
P_{x,w}(q)=\sum_{i\ge 0} q^i\dim\mathrm{Ext}_{\cO}^{\ell(w)-\ell(x)-2i}(M_x,L_w) \quad \text{for all $x\le w$}.
\]
\end{enumerate}
\begin{remark}
	We call $P_{x,w}$ the \emph{Kazhdan-Lusztig polynomial} of $x,w\in W$. Denote by $P_{x,w}^{[\lambda]}$ the Kazhdan-Lusztig polynomial of $x,w\in W_{[\lambda]}$ for $\lambda\in\mathfrak{h}^*$.
	%By property (d), Lemma \ref{bruhatorders} and induction on $\ell(w)$, we have the Kazhdan-Lusztig polynomial of $x,w\in W$ restricted to $W_{[\lambda]}$ coincides with the Kazhdan-Lusztig polynomial of $x,w\in W_{[\lambda]}$ for $\lambda\in \mathfrak{h}^*$ since only group elements involved in the recursions are those which are $\le w$ or equivalently, $\le_{[\lambda]} w$ for $w\in W_{[\lambda]}$.
\end{remark}

\subsection{Parabolic Kazhdan-Lusztig polynomials}
In this section, we introduce the \emph{parabolic Kazhdan-Lusztig polynomial} $P_{u,v}^{J,y}$ of $u,v\in {}^JW$ of type $y$ for $y\in\{-1,q\}$.
Given ${w} \in W$,  we let $
D({w})
:= \{s \in S : {w} s <
{w}\}$.
Given $K\subseteq \Phi$, let $W_K$ be the subgroup of $W$ generated by $K$.
For $J\subseteq \Delta$, $W_J$ is the standard parabolic subgroup of $W$ generated by $J$. In the literature, the set ${}^JW$ of minimal length right coset representatives of $W_J$ in $ W$ is often characterized in different equivalent ways.

\begin{lemma}
	[{See \cite[Remark 3.6]{EHP}}]
	\label{JW} 
	Let $w\in W$. The following statements are then equivalent:
	\begin{enumerate}
		\item $w^{-1}\Phi_J^+\subseteq \Phi^+$.
		\item $\ell(s_\alpha w)=\ell(w)+1$ for all $\alpha\in J$.
		\item $\ell(s_\alpha w)>\ell(w)$ for all $\alpha\in J$.
		\item $s_\alpha w>w$ for all $\alpha\in J$.
		\item $w$ is the unique minimal length element in its right $W_J$-coset $W_Jw$.
	\end{enumerate}
	%where the Bruhat order, length function are defined on $W$.
\end{lemma}

\begin{remark}
	 We make some observations.
	 \begin{itemize}
	 	\item Although $W_{\Sigma_\mu}$ is defined as the isotropy group of $\mu$,
	 	$W_{\Sigma_\mu}$ is the subgroup of $W$ generated by ${\Sigma_\mu}$ (cf. Lemma \ref{parabolicgenerate}). This justifies the notation $W_{\Sigma_\mu}$.
	 	\item Define $W^J:=\{w^{-1}\in W : w\in {}^JW\}$.	Note that $W^J$ is the set of minimal length left coset representatives of $W_J$ in $ W$.
	 \end{itemize}
\end{remark}

In this section, the polynomial $R_{u,v}^{J,y}$ in our notation is the same as $R_{u,v}^{\{s_\alpha: \alpha\in J\},y}$ in the notation of \cite{Fran0}.
Similarly, the polynomial $P_{u,v}^{J,y}$ in our notation is the same as $P_{u,v}^{\{s_\alpha: \alpha\in J\},y}$ in the notation of \cite{Fran0}.
Since $(W,S)$ is a Coxeter system and $\{s_\alpha: \alpha\in J\}\subseteq S$, the following four results are special cases of results of Deodhar, and 
we refer the reader to \cite[\S2]{Fran0} for the statements in full generality and \cite[\S2 and \S3]{VD} for their proofs.

\begin{theorem} [{See \cite[\S2 and \S 3]{VD}}]
For each $y\in \{-1,q\}$, there is a unique family of polynomials $\{R_{u,v}^{J,y}(q)\}_{u,v\in {}^JW}\subseteq \mathbb{Z}[q]$ such that, for all $u,v\in {}^JW$:
\begin{enumerate}
    \item $R_{u,v}^{J,y}(q)=0$ if $u\not\le v$.
    \item $R_{u,u}^{J,y}(q)=1$.
    \item If $u<v$ and $s\in D(v)$ then \[
    R_{u,v}^{J,y}(q)=
    \begin{cases}
    R^{J,y}_{us,vs}(q), & \text{if $s\in D(u)$,} \\
    (q-1)R^{J,y}_{u,vs}(q)+q R^{J,y}_{us,vs}(q), & \text{if $s\not\in D(u)$ and $us\in {}^JW$,} \\
    (q-1-y)R^{J,y}_{u,vs}(q), & \text{if $s\not\in D(u)$ and $us\not\in {}^JW$.} \\
    \end{cases}
    \]
\end{enumerate}
\end{theorem}

\begin{theorem} [{See \cite[\S 2 and \S 3]{VD}}]
	\label{paraKL}
For each $y\in \{-1,q\}$, there is a unique family of polynomials $\{P_{u,v}^{J,y}(q)\}_{u,v\in {}^JW}\subseteq \mathbb{Z}[q]$ such that, for all $u,v\in {}^JW$:
\begin{enumerate}
    \item $P_{u,v}^{J,y}(q)=0$ if $u\not\le v$.
    \item $P_{u,u}^{J,y}(q)=1$.
    \item $\deg (P_{u,v}^{J,y}(q))\le \frac{1}{2}(\ell(v)-\ell(u)-1)$, if $u<v$.
    \item $
    q^{\ell(v)-\ell(u)}P_{u,v}^{J,y}\left(\frac{1}{q}\right)=\sum_{u\le z\le v} R_{u,z}^{J,y}(q)P_{z,v}^{J,y}(q),
    $
    if $u\le v$.
\end{enumerate}
\end{theorem}

It is well-known that Kazhdan-Lusztig polynomials and parabolic Kazhdan-Lusztig polynomials are closely related. In fact, we have the following proposition.

\begin{proposition} [{See \cite[Proposition 3.4 and Remark 3.8]{VD}}]
	\label{paraKL-KL}
Let $u,v\in {}^JW$. Then
\[
P^{J,-1}_{u,v}(q)= P_{w_Ju,w_Jv}(q)
\]
where $w_J$ is the longest element in $W_J$, and
%\[
%R^{J,y}_{u,v}(q)=\sum_{w\in W_J} (-x)^{\ell(w)} R_{wu,v}(q)
%\]
%for all $y\in \{-1,q\}$, and
\[
P^{J,q}_{u,v}(q)=\sum_{w\in W_J} (-1)^{\ell(w)} P_{wu,v}(q).
\]
%Furthermore, it holds that .
\end{proposition}

\begin{remark}
	By Proposition \ref{paraKL-KL} and property (e) in Section \ref{kl-sect}, the constant term of $P_{u,v}^{J,-1}(q)$ is $1$. 
%	Proposition \ref{paraKL-KL} and the remark in Section \ref{kl-sect} implies that for $\lambda\in \mathfrak{h}^*$ so that $J\subseteq\Delta_{[\lambda]}$, we have the parabolic Kazhdan-Lusztig polynomial of $x,w\in {}^JW$ of type $-1$ restricted to $W_{[\lambda]}$ coincides with the parabolic Kazhdan-Lusztig polynomial of $x,w\in {}^J W_{[\lambda]}$ of type $-1$. 	
\end{remark}

%\begin{proposition} [{See \cite[Corollary 2.2]{VD1}}]
%It holds that
%\[
%q^{\ell(v)-\ell(u)}R^{J,y}_{u,v}\left(\frac{1}{q}\right)
%=(-1)^{\ell(v)-\ell(u)}R^{J,q-1-x}_{u,v}(q)
%\]
%for all $u,v\in {}^JW$, and $y\in \{-1,q\}$.
%\end{proposition}

For $u,v\in {}^JW$, let 
\[
\tilde{\mu}(u,v):=\left[q^{\frac{1}{2}(\ell(v)-\ell(u)-1)}\right](P^{J,q}_{u,v}(q)) \in \ZZ
\]
denote the coefficient of degree of $\frac{1}{2}(\ell(v)-\ell(u)-1)$ in $P^{J,q}_{u,v}(q)$. Then we have the following proposition.

\begin{proposition} [{See \cite[Proposition 3.10]{VD}}]
	\label{recursion}
Let $u,v\in {}^JW$ and $u\le v$. Then for each $s\in D(v)$ we have 
\[
P^{J,q}_{u,v}(q)=\tilde{P}-\sum_{\{u\le w\le vs: ws<w\}} \tilde{\mu}(w,vs)q^{\frac{\ell(v)-\ell(w)}{2}}P^{J,q}_{u,w}(q)
\]
where
\[
\tilde{P}
=\begin{cases}
P^{J,q}_{us,vs}(q)+qP^{J,q}_{u,vs}(q), &\text{if $us<u$,}\\
qP^{J,q}_{us,vs}(q)+P^{J,q}_{u,vs}(q), &\text{if $u<us\in {}^JW$,}\\
0, &\text{if $u<us\not\in {}^JW$.}\\
\end{cases}
\]
\end{proposition}

\begin{remark}
	Denote by $P^{[\lambda],J,y}_{u,v}(q)$ the parabolic Kazhdan-Lusztig polynomial of $u,v\in {}^JW_{[\lambda]}$ of type $y$ for $\lambda\in\mathfrak{h}^*$ and $J\subseteq\Delta_{[\lambda]}$.
%     Similar to the remark in Section \ref{kl-sect}, Proposition \ref{recursion} and Lemma \ref{bruhatorders} implies that for $\lambda\in \mathfrak{h}^*$ so that $J\subseteq\Delta_{[\lambda]}$, we have the parabolic Kazhdan-Lusztig polynomial of $x,w\in {}^JW$ of type $q$ restricted to $W_{[\lambda]}$ coincides with the parabolic Kazhdan-Lusztig polynomial of $x,w\in {}^J W_{[\lambda]}$ of type $q$. 
\end{remark}

\subsection{Relative Kazhdan-Lusztig-Vogan polynomials}
In this section, for each $\lambda\in\Lambda^+_I$ and for each $x,w\in {}^IW_{[\lambda]}^{\Sigma_\mu}$, we will define a polynomial called
the \emph{relative Kazhdan-Lusztig-Vogan polynomial} associated to $\lambda$ of $x,w\in {}^IW_{[\lambda]}^{\Sigma_\mu}$. We will prove that it is related to a parabolic Kazhdan-Lusztig polynomial of ${}^KW_{[\lambda]}$ of type $y$ for $(K,y)\in\{({\Sigma_\mu},q),(I,-1)\}$. As an application, we determine when the sum of coefficients of a relative Kazhdan-Lusztig-Vogan polynomial associated to $\lambda$ is nonzero for regular $\lambda\in \Lambda^+_I$.
All definitions in this section (except Definition \ref{KLV}) follow the conventions in \cite{HJ}.

\begin{definition}[{See \cite[\S 1.8]{HJ}}]
	\label{dot}
	For $w \in W$ and $\eta\in \h^*$, define a shifted action of $W$ (called
	the \emph{dot action}) by $w \cdot \eta := w(\eta + \rho) - \rho$. If $\eta, \nu \in \h^*$, then we say that $\eta$ and $\nu$
	are \emph{linked} (or \emph{$W$-linked}) if for some $w \in W$, we have $\nu = w \cdot \eta$. 
\end{definition}

Linkage is clearly an equivalence relation on $\h^*$. The orbit $\{w \cdot \eta : w \in W\}$ of $\eta$ under
the dot action is called the \emph{linkage class} (or \emph{$W$-linkage class}) of $\eta$.

The weight
$\eta \in \h^*$ is \emph{regular} if $|W \cdot \eta| = |W|$ or, equivalently, if $\langle \eta + \rho,\alpha^\lor\rangle \neq 0$ for
all $\alpha\in\Phi$ (see {\cite[\S 1.8]{HJ}}). Weights which are not regular are called \emph{singular}.
We say the infinitesimal character $\chi_\eta$ is \emph{regular} if $\eta$ is regular in this sense.

Following \cite{HJ}, we denote by $E$ the Euclidean space spanned by $\Phi$. The $\mathbb{Z}$-span $\Lambda_r$ of $\Phi$ is called the \emph{root lattice}. 
For $\eta \in \h^*$, let
\[
\Phi_{[\eta]}:=\{\alpha\in \Phi  :  \langle \eta,\alpha^{\lor}\rangle\in\mathbb{Z} \}
\qquad\text{and}\qquad
W_{[\eta]}:=\{w\in W  :  w\cdot\eta-\eta\in \Lambda_r\}.
\]
Fix $\nu\in\h^*$, if $\eta-\nu$ is an integral weight, then 
$
\Phi_{[\eta]}=\Phi_{[\nu]} $ and $W_{[\eta]}=W_{[\nu]}.
$
If $\nu\in W_{[\eta]}\cdot\eta$, then 
$
W_{[\eta]}=W_{[\nu]} $ and $W_{[\eta]}\cdot\eta=W_{[\nu]}\cdot\nu.
$

\begin{theorem} [{\cite[Theorem 3.4]{HJ}}]
	\label{intweylgp}
	For $\eta\in \h^*$, the following holds:
	\begin{enumerate}
		\item $\Phi_{[\eta]}$ is an abstract root system in its $\mathbb{R}$-span $E(\eta)\subseteq E$.
		\item $W_{[\eta]}$ is the Weyl group of the root system $\Phi_{[\eta]}$. In particular, it is generated by the reflections $s_\alpha$ with $\alpha\in \Phi_{[\eta]}$.
	\end{enumerate}
\end{theorem}

Recall that a weight $\eta\in \h^*$ is \emph{antidominant} if $\langle \eta+\rho,\alpha^{\lor}\rangle\not\in\mathbb{Z}^{>0}$ for all $\alpha\in \Phi^+$.

\begin{proposition} [{\cite[Proposition 3.5]{HJ}}] 
	\label{antidominant}
	
	Let $\Phi_{[\eta]}$ and $W_{[\eta]}$ be the corresponding root system and Weyl group of $\eta\in\h^*$.
	Denote by $\Delta_{[\eta]}$ the simple system corresponding to the positive system $\Phi_{[\eta]}\cap \Phi^+$ in $\Phi_{[\eta]}$. Then, $\eta$ is antidominant if and only if the following equivalent conditions hold:
	\begin{enumerate}
		\item $\langle \eta+\rho,\alpha^{\lor}\rangle\le 0$ for all $\alpha\in\Delta_{[\eta]}$.
		\item $\eta\le s_\alpha\cdot \eta$ for all $\alpha\in \Delta_{[\eta]}$.
		\item $\eta\le w \cdot \eta$ for all $w \in W_{[\eta]}$.
	\end{enumerate}
	Therefore, there is a unique antidominant weight in the orbit $W_{[\eta]}\cdot \eta$.
\end{proposition}

\begin{remark} [{\cite[Remark 6.3]{HX}}]
	It holds that $M\in\cO^\p_{\chi_\lambda}$ has a direct sum decomposition $M=\bigoplus M_i$ such that all weights of each $M_i$ are contained in a single coset of the root lattice $\Lambda_r$ in $\h^*$. Therefore, the category $\cO^\p_{\chi_\lambda}$ decomposes as a direct sum of full subcategories, which can be indexed by the nonempty intersection of the orbit $W\cdot\lambda$ with the cosets $\h^*/\Lambda_r$. We use the antidominant weight $\mu$ in the intersection to parameterize the corresponding subcategory of $\cO^\p_{\chi_\lambda}$. 
	Following \cite{HX}, we denote this subcategory by $\cO^\p_{\mu}$. %If $I=\emptyset$, then the parabolic subalgebra $\p$ collapses to a Borel subalgebra $\mathfrak{b}$ of $\g$ and $\cO^\p_{\mu}$ turns into the ordinary block $\cO_{\mu}$ of the BGG category $\cO$.
\end{remark}

Now we consider the simple highest weight module $L(\lambda)$ for $\lambda\in \Lambda^+_I$. For the rest of the paper, we will denote by $\mu$ the unique antidominant weight in  $W_{[\lambda]}\cdot\lambda$. Let $W_I$ be the Weyl group attached to the root system $\Phi_{I}$ with the longest element $w_I$. Then $W_{I}\subseteq W_{[\lambda]}$. 
Define
\[
{}^IW_{[\lambda]}:=\{w\in W_{[\lambda]}:w<s_\alpha w \ \text{for all }\alpha\in I\}, 
\]
where $<$ is the Bruhat ordering on $W$. 
We denote the set of singular
simple roots associated to $\mu$ in $\Delta_{[\lambda]}$ by
\[
{\Sigma_\mu} := \{\alpha \in \Delta_{[\lambda]} : \langle \mu + \rho, \alpha^\lor\rangle= 0\}.
\]
Then $W_{\Sigma_\mu}:= \{w \in W : w \cdot \mu = \mu\} \subseteq W_{[\lambda]}$ is the isotropy group of $\mu$. Let
\[
{}^IW_{[\lambda]}^{\Sigma_\mu} : = \{w \in {}^IW_{[\lambda]} : w < ws_\alpha\in {}^IW_{[\lambda]} \ \text{for all }\alpha\in{\Sigma_\mu}\},
\]
where $<$ is again the Bruhat ordering on $W$.

Let $S_{[\eta]}:=\{s_\alpha  :  \alpha\in\Delta_{[\eta]}\}$ be a generating set of $W_{[\eta]}$. %The following theorem is proven in \cite{BGS} and generalized in \cite{WS}.

\begin{theorem} [{\cite[Theorem 11]{WS}}] \label{Soergel} 
	Let $\g$, $\g'$ be semisimple Lie algberas, with respective Weyl group $W$ and $W'$. Fix antidominant weights $\nu$ for a Cartan subalgebra $\h\leftarrow \mathfrak{b}\subseteq\g$ and $\nu'$ for $\h'\leftarrow \mathfrak{b}'\subseteq\g'$, where $\mathfrak{b}$ and $\mathfrak{b}'$ are the Borel subalgebras compatible with $\h$ and $\h'$, respectively. Let $W_{[\nu]}$ and $W_{[\nu']}'$ be the corresponding reflection subgroups. Suppose there is an isomorphism of Coxeter system
	\[
	\begin{array}{rcl} (W_{[\nu]},S_{[\nu]}) &\simeq& (W'_{[\nu']},S'_{[\nu']})
	\\ x & \mapsto& x'
	\end{array}
	\]
	that takes the isotropy group of $\nu$ to the isotropy group of $\nu'$. Then, the corresponding subcategory $\cO_{\nu}$ is equivalent to $\cO'_{\nu'}$, with $L(x\cdot \nu)$ sent to $L(x'\cdot \nu')$ and $M(x\cdot\nu)$ sent to $M(x'\cdot\nu')$.
\end{theorem}

Given an arbitrary $\eta \in\h^*$, let $\eta^\natural$ denote the integral weight (relative to $\Phi_{[\eta]}$) in $E(\eta)$ characterized uniquely by the requirement that
$\langle \eta^\natural, \alpha^\lor\rangle = \langle \eta, \alpha^\lor\rangle$ for all $\alpha \in \Phi_{[\eta]}$; see \cite[\S 7.4]{HJ}. 
%When $\eta\in\Lambda$, we have $\eta^\natural=\eta$.

Consider $\mu\in\h^*$, the antidominant weight in $W_{[\lambda]}\cdot\lambda$.
% There is a unique $\Phi_{[\mu]}$-integral weight $\mu^{\natural}$ (\cite[\S 7.4]{HJ}) in the subspace $E(\mu)$ ($\mathbb{R}$-span of $\Phi_{[\mu]}$) characterized by the condition that 
%\[
%\langle \mu^{\natural},\alpha^\lor\rangle=\langle \mu,\alpha^\lor\rangle \ \text{for all $\alpha\in\Phi_{[\mu]}$.} 
%\]
Then, $\mu$ and $\mu^\natural$ have the same attached Coxeter systems and isotropy groups (cf. Lemma \ref{sameisotropy}). 
%Hence, the isomorphism between the corresponding Coxeter system is
%\[
%(W_{[\mu]},S_{[\mu]})\rightarrow ((W_{[\mu]})_{[\mu^{\natural}]},S'_{[\mu^{\natural}]})=(W_{[\mu]},S_{[\mu]}), x\mapsto x,
%\]
As noted in \cite[Remark 6.5]{HX},
there is a finite dimensional complex semisimple Lie algebra $\g^\natural$ compatible with the abstract reduced root system $\Phi_{[\mu]}=\Phi_{[\lambda]}$ and hence the Coxeter system $(W_{[\mu]},S_{[\mu]})=(W_{[\lambda]},S_{[\lambda]})$. Let $\h^\natural$ be a Cartan subalgebra of $\g^\natural$. Its dual $(\h^\natural)^*$ is a subspace in $\h^*$. Denote by $\cO^\natural$ the BGG category of $\g^\natural$. We denote by $\cO_{\mu^\natural}^\natural$ the full subcategory of $\cO^\natural$ corresponding to $\mu^\natural$ (which is antidominant for $\Phi_{[\lambda]}$) and write $\p^\natural$ for the standard parabolic subalgebra of $\g^\natural$ corresponding to $I$. Denote by $\g^\natural_\alpha$ the root subspace of $\g^\natural$ corresponding to a root $\alpha$ and let $\mathfrak{l}^\natural:=\h^\natural\oplus\sum_{\alpha\in \Phi_I} \g^\natural_\alpha$. Write $(\cO^\natural)^{\p^\natural}$ for the category of all finitely generated, locally $\mathfrak{p}^\natural$-finite and $\mathfrak{l}^\natural$-semisimple
$\g^\natural$-modules.
%relative version of $\cO^\natural$ corresponding to $\p^\natural$.
Finally, let $\rho_{[\lambda]}$ be the half sum of positive roots in $\Phi_{[\lambda]}^+:=\Phi_{[\lambda]}\cap\Phi^+$.

To summarize, we have the following corollary.

\begin{corollary} [{\cite[Corollary 6.6]{HX}}]  \label{category} 
	With the setting as above, there is an equivalence of categories $\mathcal{F}$ between $\cO_\mu$ and $\cO_{\mu^\natural}^\natural$ satisfying the following conditions:
	\begin{enumerate}[label=(\roman*)]
		\item $\mathcal{F}(L(x\cdot \mu))\cong L(x\cdot \mu^\natural)$ and $\mathcal{F}(M(x\cdot \mu))\cong M(x\cdot \mu^\natural)$ for $x\in W_{[\lambda]}$.
		
		\item If $x\cdot\mu$ is in $\Lambda_I^+$ then $\mathcal{F}(M_I(x\cdot \mu))\cong M_I(x\cdot \mu^\natural)$.
		\item For any $V\in\cO_\mu$, one has 
		$
		\mathrm{Ext}^i_\cO(M_I(x\cdot \mu),V)\cong  \mathrm{Ext}^i_{\cO^\natural}\left(M_I(x\cdot \mu^\natural),\mathcal{F}(V)\right).
		$
	\end{enumerate}
\end{corollary}

\begin{remark} We make some observations.
	\begin{itemize}
		\item  Note that $\Delta_{[\lambda]}=\Delta_{[\mu]}$, $\Phi_{[\lambda]}=\Phi_{[\mu]}$, $E(\lambda)=E(\mu)$, $W_{[\lambda]}=W_{[\mu]}$, ${}^I W_{[\lambda]}={}^I W_{[\mu]}$, ${}^I W_{[\lambda]}^{\Sigma_\mu}={}^I W_{[\mu]}^{\Sigma_\mu}$.
		\item $W_{\Sigma_\mu} = \{w \in W : w \cdot \mu = \mu\}=\{w \in W_{[\lambda]}: w \cdot \mu = \mu\}$ since $w \cdot \mu-\mu=0\in\Lambda_r$ for all $w\in W_{\Sigma_\mu}$.
		\item Since $\lambda\in \Lambda_I^+$, we have $\langle\lambda,\alpha^\lor\rangle\in\mathbb{Z}^{\ge 0}$ for all $\alpha\in I$. Therefore $\alpha\in\Phi_{[\lambda]}$ for all $\alpha\in I$. If $\alpha\in I$ then $\alpha\in  \Phi^+$ so $\alpha\in \Phi^+\cap \Phi_{[\lambda]}=\Phi_{[\lambda]}^+$. Suppose $\alpha\in I$ can be written as sum of two roots in $\Phi_{[\lambda]}^+$ on contrary. Then $\alpha\in I$ can be written as sum of two roots in $\Phi^+$, a contradiction to the fact that $\alpha\in I\subseteq \Delta$. Hence $\alpha\in I$ cannot be written as sum of two roots in $\Phi_{[\lambda]}^+$.
		This implies that $\alpha\in\Delta_{[\lambda]}$ for all $\alpha\in I$. Therefore $I\subseteq \Delta_{[\lambda]}$ and hence $W_I\subseteq W_{[\lambda]}$ by Theorem \ref{intweylgp}.
	\end{itemize}
\end{remark}

\begin{proposition} [{\cite[Proposition 2.72 (Chevalley's Lemma)]{AK}}]  \label{chevalley}
	Fix $v\in E$, and let $W_0=\{w\in W: wv=v\}$. Then $W_0$ is generated by the root reflection $s_\alpha$ such that $\langle v,\alpha^\lor\rangle=0$.
\end{proposition}

\begin{lemma}\label{sameisotropy}
	It holds that
	\begin{enumerate}
		\item $\mu$ and $\mu^\natural$ have the same attached Coxeter systems.
		\item the isotropy group of $\mu$ is the same as the isotropy group of $\mu^\natural$.
	\end{enumerate}
\end{lemma}

\begin{proof}	
	We prove each part in turn.
	\begin{enumerate}
		\item 
		The attached Coxeter system of $\mu$ is $(W_{[\mu]},S_{[\mu]})$. 
		Since $\mu^\natural\in E(\mu)$ is $\Phi_{[\mu]}$-integral, it holds that the attached Coxeter system of $\mu^\natural$ is \[
		\left(\left(W_{[\mu]}\right)_{[\mu^\natural]},\left(S_{[\mu]}\right)_{[\mu^\natural]}\right)=(W_{[\mu]},S_{[\mu]}).
		\] 
		The claim follows.
		\item 
		For all $\beta\in\Phi_{[\lambda]}$, 
		\begin{align*}
		\langle
		\rho-\rho_{[\lambda]},\beta
		\rangle
		&=\langle
		s_\beta(\rho-\rho_{[\lambda]}),s_\beta(\beta)
		\rangle\\
		&=\begin{cases}
		\left\langle
		s_\beta\left(\frac{1}{2}\sum_{\alpha\in \Phi^+\backslash \Phi_{[\lambda]}^+}\alpha\right),s_\beta(\beta)
		\right\rangle,&\text{if $\beta\in \Phi_{[\lambda]}^+$}\\
		\left\langle
		s_{-\beta}\left(\frac{1}{2}\sum_{\alpha\in \Phi^+\backslash \Phi_{[\lambda]}^+}\alpha\right),s_\beta(\beta)
		\right\rangle,&\text{if $\beta\in -\Phi_{[\lambda]}^+$}\\
		\end{cases}\\
		&=\left\langle
		\frac{1}{2}\sum_{\alpha\in \Phi^+\backslash \Phi_{[\lambda]}^+}\alpha,-\beta
		\right\rangle\\
		&=-\langle
		\rho-\rho_{[\lambda]},\beta
		\rangle.
		\end{align*}
		This implies $\langle
		\rho,\beta^\lor
		\rangle=\langle
		\rho_{[\lambda]},\beta^\lor
		\rangle$ for all $\beta\in \Phi_{[\lambda]}$.
		%In particular, since $\rho_{[\lambda]}$ is $\Phi_{[\lambda]}$-integral, we get $\rho_{[\lambda]}=\rho^\natural$.
		Recall that $\langle
		\mu,\beta^\lor
		\rangle=\langle
		\mu^\natural,\beta^\lor
		\rangle$ for all $\beta\in \Phi_{[\lambda]}$.
		Hence $\langle
		\mu+\rho,\beta^\lor
		\rangle=\langle
		\mu^\natural+\rho_{[\lambda]},\beta^\lor
		\rangle$ for all $\beta\in \Phi_{[\lambda]}$.
		
		Note that $W_{[\lambda]}$ is the Weyl group of $\Phi_{[\lambda]}$ (cf. Theorem \ref{intweylgp}). 
		Let $w=s_{\alpha_1}\cdots s_{\alpha_k}\in W_{[\lambda]}$ be an arbitrary expression for some $\alpha_1,\cdots, \alpha_k\in\Phi_{[\lambda]}$. Then
		\allowdisplaybreaks
		\begin{align*}
		\mu-w\cdot\mu
		&= \mu-(s_{\alpha_1}\cdots s_{\alpha_k})\cdot\mu\\
		&=\mu-s_{\alpha_1}\cdot\left((s_{\alpha_2}\cdots s_{\alpha_k})\cdot\mu\right)\\
		&=\mu-\left((s_{\alpha_2}\cdots s_{\alpha_k})\cdot\mu-\left\langle (s_{\alpha_2}\cdots s_{\alpha_k})\cdot\mu+\rho, \alpha_1^\lor\right\rangle\alpha_1\right)\\
		&=\mu-(s_{\alpha_2}\cdots s_{\alpha_k})\cdot\mu+\left\langle (s_{\alpha_2}\cdots s_{\alpha_k})\cdot\mu+\rho, \alpha_1^\lor\right\rangle\alpha_1\\
		%&=\mu-(s_{\alpha_3}\cdots s_{\alpha_k})\cdot\mu+\left\langle (s_{\alpha_3}\cdots s_{\alpha_k})\cdot\mu+\rho, \alpha_2^\lor\right\rangle\alpha_2+\left\langle (s_{\alpha_2}\cdots s_{\alpha_k})\cdot\mu+\rho, \alpha_1^\lor\right\rangle\alpha_1\\
		&\quad \vdots \\
		&=\mu-s_{\alpha_k}\cdot\mu+\sum_{i=1}^{k-1}\left\langle (s_{\alpha_{i+1}}\cdots s_{\alpha_k})\cdot\mu+\rho, \alpha_i^\lor\right\rangle\alpha_i\\
		&=\sum_{i=1}^{k-1}\left\langle (s_{\alpha_{i+1}}\cdots s_{\alpha_k})\cdot\mu+\rho, \alpha_i^\lor\right\rangle\alpha_i+\left\langle \mu+\rho, \alpha_k^\lor\right\rangle\alpha_k\\
		&=\sum_{i=1}^{k-1}\left\langle (s_{\alpha_{i+1}}\cdots s_{\alpha_k})(\mu+\rho), \alpha_i^\lor\right\rangle\alpha_i+\left\langle \mu+\rho, \alpha_k^\lor\right\rangle\alpha_k\\
		&=\sum_{i=1}^{k-1}\left\langle \mu+\rho, \left((s_{\alpha_{i+1}}\cdots s_{\alpha_k})^{-1}\alpha_i\right)^\lor\right\rangle\alpha_i+\left\langle \mu+\rho, \alpha_k^\lor\right\rangle\alpha_k.
		\end{align*}
		Similarly,
		\begin{align*}
		\mu^\natural-w\cdot\mu^\natural
		&=\sum_{i=1}^{k-1}\left\langle \mu^\natural+\rho_{[\lambda]}, \left((s_{\alpha_{i+1}}\cdots s_{\alpha_k})^{-1}\alpha_i\right)^\lor\right\rangle\alpha_i+\left\langle \mu^\natural+\rho_{[\lambda]}, \alpha_k^\lor\right\rangle\alpha_k.
		\end{align*}
		Since $(s_{\alpha_{i+1}}\cdots s_{\alpha_k})^{-1}\alpha_i,\alpha_k\in\Phi_{[\lambda]}=\Phi_{[\mu]}$ for all $1\le i\le k-1$, we have $\mu-w\cdot\mu=\mu^\natural-w\cdot\mu^\natural$ for all $w\in W_{[\lambda]}$.
		
		Let ${\Sigma}^\natural_{\mu^\natural}$ be the set of singular simple roots associated to $\mu^\natural$ in $\left(\Delta_{[\lambda]}\right)_{[\mu^\natural]}$, where $\left(\Delta_{[\lambda]}\right)_{[\mu^\natural]}$ is the set of simple roots in $\left(\Phi_{[\lambda]}\right)_{[\mu^\natural]}\cap \Phi_{[\lambda]}^+$, i.e., ${\Sigma}^\natural_{\mu^\natural}
		:=\{\alpha\in \left(\Delta_{[\lambda]}\right)_{[\mu^\natural]}: \langle \mu^\natural+\rho_{[\lambda]},\alpha^\lor\rangle=0\}$.
		
		Since $\mu^\natural\in E(\lambda)$ is $\Phi_{[\lambda]}$-integral, we have $\left(\Phi_{[\lambda]}\right)_{[\mu^\natural]}=\Phi_{[\lambda]}$ and then $\left(\Phi_{[\lambda]}\right)_{[\mu^\natural]}\cap \Phi_{[\lambda]}^+=\Phi_{[\lambda]}\cap \Phi_{[\lambda]}^+=\Phi_{[\lambda]}^+$. This implies $\left(\Delta_{[\lambda]}\right)_{[\mu^\natural]}=\Delta_{[\lambda]}$.
		Then \[
		{\Sigma}^\natural_{\mu^\natural}
		=\{\alpha\in \left(\Delta_{[\lambda]}\right)_{[\mu^\natural]}: \langle \mu^\natural+\rho_{[\lambda]},\alpha^\lor\rangle=0\}
		=\{\alpha\in \Delta_{[\lambda]}: \langle \mu^\natural+\rho_{[\lambda]},\alpha^\lor\rangle=0\}.
		\] 	
		The isotropy group of $\mu^\natural$ is defined to be $\left(W_{[\lambda]}\right)_{{\Sigma}^\natural_{\mu^\natural}}
		:=\{w\in W_{[\lambda]} : w\cdot\mu^\natural=\mu^\natural\}$.
		This implies
		\[
		\left(W_{[\lambda]}\right)_{{\Sigma}^\natural_{\mu^\natural}}
		=\{w\in W_{[\lambda]} : w\cdot\mu^\natural=\mu^\natural\}
		=\{w\in W_{[\lambda]} : w\cdot\mu=\mu\}=W_{\Sigma_\mu}.
		\]
		That is, the isotropy group of $\mu$ is the same as the isotropy group of $\mu^\natural$.
	\end{enumerate}
\end{proof}

\begin{lemma}\label{parabolicgenerate}
	It holds that
	$W_{\Sigma_\mu}=\langle s_\alpha\in W : \alpha\in {\Sigma_\mu}\rangle$. Hence $\left(W_{[\lambda]}\right)_{{\Sigma}^\natural_{\mu^\natural}}
	=\langle s_\alpha\in W_{[\lambda]}: \alpha\in {\Sigma}^\natural_{\mu^\natural}\rangle$.
	%		
	%		$W_{\Sigma_\mu}
	%		=\langle s_\alpha\in W_{[\lambda]}: \alpha\in {\Sigma_\mu}\rangle
	%		=\langle s_\alpha\in W : \alpha\in {\Sigma_\mu}\rangle$.
\end{lemma}

This result justifies the notations $W_{\Sigma_\mu}$ and $\left(W_{[\lambda]}\right)_{{\Sigma}^\natural_{\mu^\natural}}$. 

\begin{proof}
	%Since $\mu^\natural,\rho_{[\lambda]}\in E(\lambda)$, we have $\mu^\natural+\rho_{[\lambda]}\in E(\lambda)$. 
	Recall that $\langle\mu^\natural+\rho_{[\lambda]},\alpha^\lor\rangle=\langle\mu+\rho,\alpha^\lor\rangle$ for all $\alpha\in\Phi_{[\lambda]}$. Then
	%\allowdisplaybreaks
	\begin{align*}
	&\left(W_{[\lambda]}\right)_{{\Sigma}^\natural_{\mu^\natural}}&\\
	&=\{w\in W_{[\lambda]} : w\cdot\mu^\natural=\mu^\natural\}&\text{(by Definition of $\left(W_{[\lambda]}\right)_{{\Sigma}^\natural_{\mu^\natural}}$)}\\
	&=\{w\in W_{[\lambda]} : w(\mu^\natural+\rho_{[\lambda]})=\mu^\natural+\rho_{[\lambda]}\}&\text{(by Definition \ref{dot})}\\
	&=\langle s_\alpha\in W_{[\lambda]}: \langle\mu^\natural+\rho_{[\lambda]},\alpha^\lor\rangle=0\rangle & \text{(by Proposition \ref{chevalley} and $\mu^\natural+\rho_{[\lambda]}\in E(\lambda)$)}\\
	&=\langle s_\alpha\in W_{[\lambda]}: \langle\mu+\rho,\alpha^\lor\rangle=0\rangle&\text{(as $\langle\mu^\natural+\rho_{[\lambda]},\alpha^\lor\rangle=\langle\mu+\rho,\alpha^\lor\rangle$ for all $\alpha\in\Phi_{[\lambda]}$)}\\
	&=\langle s_\alpha\in W_{[\lambda]}: s_\alpha\cdot\mu=\mu\rangle&\text{(as $\mu-s_\alpha\cdot\mu=\langle\mu+\rho,\alpha^\lor\rangle\alpha$ for all $\alpha\in\Phi_{[\lambda]}$)}\\
	&\subseteq \{w\in W_{[\lambda]} : w\cdot\mu=\mu\}&\text{(by Definition of a generating set of a group)}\\
	&=\{w\in W : w\cdot\mu=\mu\}&\text{(by Remark following Corollary \ref{category})}\\
	&=W_{\Sigma_\mu}.&\text{(by Definition of $W_{{\Sigma}_{\mu}}$)}
	\end{align*}
	%	Note that we have $	\left(W_{[\lambda]}\right)_{{\Sigma}^\natural_{\mu^\natural}}=\langle s_\alpha\in W_{[\lambda]}: \alpha\in {\Sigma}^\natural_{\mu^\natural}\rangle$.
	By Lemma \ref{sameisotropy}, it holds that $\left(W_{[\lambda]}\right)_{{\Sigma}^\natural_{\mu^\natural}}=W_{\Sigma_\mu}$ and hence $
	W_{\Sigma_\mu}
	=\langle s_\alpha\in W_{[\lambda]}: \langle\mu+\rho,\alpha^\lor\rangle=0\rangle.
	$
	Let $\left(\Phi_{[\lambda]}\right)_{\Sigma_\mu}:=\Phi_{[\lambda]}\cap\sum_{\alpha\in\Sigma_\mu}\mathbb{Z}\alpha$ and $\Phi({\Sigma_\mu}):=\{\alpha\in\Phi_{[\lambda]}: \langle\mu+\rho,\alpha^\lor\rangle=0\}$. It is clear that $\left(\Phi_{[\lambda]}\right)_{\Sigma_\mu}$ and $\Phi({\Sigma_\mu})$ are root systems. Note that ${\Sigma_\mu}$ is the set of simple roots in the positive root system $\left(\Phi_{[\lambda]}\right)_{\Sigma_\mu}\cap \Phi^+$. 
	Let $\Delta({\Sigma_\mu})$ be the set of simple roots in the positive root system $\Phi({\Sigma_\mu})\cap \Phi^+$.
	Clearly, $\left(\Phi_{[\lambda]}\right)_{\Sigma_\mu}\subseteq \Phi({\Sigma_\mu})$. Note that $\Delta({\Sigma_\mu})\subseteq \Phi({\Sigma_\mu})\cap \Phi^+\subseteq\Phi_{[\lambda]}^+$ and $\langle\mu+\rho,\alpha^\lor\rangle=0$ for all $\alpha\in \Delta({\Sigma_\mu})$. Suppose $\alpha\in \Delta({\Sigma_\mu})$ can be written as sum of two roots in $\Phi_{[\lambda]}^+$ on contrary, then $\alpha=\beta+\gamma$ for some $\beta,\gamma\in \Phi_{[\lambda]}^+$. This implies that
	\[
	0
	=\dfrac{\langle\mu+\rho,\alpha^\lor\rangle\langle \alpha,\alpha\rangle}{2}
	=\langle\mu+\rho,\alpha\rangle
	=\langle\mu+\rho,\beta\rangle
	+\langle\mu+\rho,\gamma\rangle.
	\]
	By Proposition \ref{antidominant} and the positive definiteness of the inner product, we have 
	\[
	\langle\mu+\rho,\beta\rangle\le 0 
	\qquad\text{and}\qquad
	\langle\mu+\rho,\gamma\rangle\le 0,
	\]
	and hence \[
	\langle\mu+\rho,\beta\rangle
	=\langle\mu+\rho,\gamma\rangle=0.
	\]
	This implies that
	\[
	\langle\mu+\rho,\beta^\lor\rangle
	=\langle\mu+\rho,\gamma^\lor\rangle=0,
	\]
	i.e., $\beta,\gamma\in \Phi({\Sigma_\mu})\cap \Phi^+$.
	Then $\alpha\in \Delta({\Sigma_\mu})$ can be written as sum of two roots in $\Phi({\Sigma_\mu})\cap \Phi^+$, a contradiction to the fact that $\alpha\in \Delta({\Sigma_\mu})$. Therefore $\alpha\in \Delta({\Sigma_\mu})$ cannot be written as sum of two roots in $\Phi_{[\lambda]}^+$ and hence $\alpha\in\Delta_{[\lambda]}$. In particular, $\alpha \in\Sigma_\mu$ for all $\alpha\in\Delta(\Sigma_\mu)$, i.e., $\Delta(\Sigma_\mu)\subseteq \Sigma_\mu$. This implies that $\Phi({\Sigma_\mu})\subseteq \left(\Phi_{[\lambda]}\right)_{\Sigma_\mu}$. Therefore $\left(\Phi_{[\lambda]}\right)_{\Sigma_\mu}=\Phi({\Sigma_\mu})$ and hence $\Sigma_\mu=\Delta(\Sigma_\mu)$.
	Since $W_{\Sigma_\mu}$ is the Weyl group of $\Phi({\Sigma_\mu})$, we have $W_{\Sigma_\mu}
	=\langle s_\alpha\in W_{[\lambda]}: \alpha\in \Delta({\Sigma_\mu})\rangle
	=\langle s_\alpha\in W_{[\lambda]}: \alpha\in {\Sigma_\mu}\rangle$.
	Then the claim follows from the fact that $\Sigma_\mu\subseteq \Delta_{[\lambda]}$.	
\end{proof}

%\begin{remark}
%	 Apply Lemma to , we have $	\left(W_{[\lambda]}\right)_{{\Sigma}^\natural_{\mu^\natural}}=\langle s_\alpha\in W_{[\lambda]}: \alpha\in {\Sigma}^\natural_{\mu^\natural}\rangle$.
%%	$\left(W_{[\lambda]}\right)_{{\Sigma}^\natural_{\mu^\natural}}$ is the standard parabolic subgroup of $W_{[\lambda]}$ generated by ${\Sigma}^\natural_{\mu^\natural}$.
%\end{remark}

Following \cite{BBMH,HX}, we adopt the following terminology.

\begin{definition} \label{KLV}
	For $\lambda\in\Lambda_I^+$ define the \emph{relative Kazhdan-Lusztig-Vogan polynomial} 
	associated to $\lambda$ of $x,w\in {}^IW_{[\lambda]}^{\Sigma_\mu}$ to be
	\[
	{}^IP_{x,w}^{\Sigma_\mu}(q)
	:=\sum_{i\ge 0} q^{\frac{\ell_{[\lambda]}(x,w)-i}{2}}\dim\mathrm{Ext}_{\cO^\p}^i(M_I(w_Ix\cdot\mu), L(w_Iw\cdot\mu)),
	\]
	where $\ell_{[\lambda]}$ is the length function on $W_{[\lambda]}$, $\ell_{[\lambda]}(x,w):=\ell_{[\lambda]}(w)-\ell_{[\lambda]}(x)$ and $\mu$ is the unique antidominant weight in $W_{[\lambda]}\cdot\lambda$.
\end{definition}

\begin{remark}
	We also call ${}^IP_{x,w}^{\Sigma_\mu}(q)$ the relative Kazhdan-Lusztig-Vogan polynomial on $\cO^\p_\mu$ of $x,w\in {}^IW_{[\lambda]}^{\Sigma_\mu}$. 
	%since $\mu$ is the unique antidominant weight in $W_{[\lambda]}\cdot\lambda$. 
	Let ${}^IP_{x,w}^{\mu}(q):={}^IP_{x,w}^\emptyset(q)$, $P_{x,w}^{\Sigma_\mu}(q):={}^\emptyset P_{x,w}^{\Sigma_\mu}(q)$ and $P_{x,w}^{\mu}(q):={}^\emptyset P_{x,w}^\emptyset(q)$ as conventions. Note that we have ${}^IW_{[\lambda]}^\emptyset={}^IW_{[\lambda]}$, ${}^\emptyset W_{[\lambda]}^{\Sigma_\mu}=W_{[\lambda]}^{\Sigma_\mu}$ and ${}^\emptyset W_{[\lambda]}^\emptyset=W_{[\lambda]}$.	
\end{remark}

%\begin{remark} From now on, all Kazhdan-Lusztig polynomial, parabolic Kazhdan-Lusztig polynomial, relative Kazhdan-Lusztig-Vogan polynomial are polynomial associated to $(W_{[\lambda]},S_{[\lambda]})=(W_{[\lambda]},S_{[\lambda]})$.
%\end{remark} 

%\textbf{Remark:} When $\lambda$ is regular and then $\mu$ is regular. The relative Kazhdan-Lusztig-Vogan polynomial ${}^IP_{x,w}^{\Sigma_\mu}$ defined in \cite{BBMH} and \cite{HX} becomes ${}^IP_{x,w}$ defined here.

We have some results about ${}^IW_{[\lambda]}^{\Sigma_\mu}$. Before that, we need the following lemma. 

\begin{lemma}\label{bruhatorders}
	Suppose $\eta\in\mathfrak{h}^*$, for all $x,w\in W_{[\eta]}$, it holds that $x\le w$ if and only if $x\le_{[\eta]} w$, where $\le_{[\eta]}$ is the Bruhat ordering on $W_{[\eta]}$.
\end{lemma}

\begin{proof}
	If $w \in W_{[\eta]}$, then $w$ has a reduced expression in $W$ involving only factors in $S_{[\eta]}$. Then 
	$x \le w$ if and only if $x$ has a reduced expression which occurs as a subexpression of this reduced
	expression, which holds if and only if $x\le_{[\eta]} w$.
\end{proof}

\begin{remark}
	Since $\lambda\in\Lambda_I^+$, we have $I\subseteq \Delta_{[\lambda]}$ by the remark following Corollary \ref{category}, and hence ${}^IW_{[\lambda]}$ can be defined by using the Bruhat ordering on $W_{[\lambda]}$ instead of $W$, i.e., ${}^IW_{[\lambda]}={}^I\left(W_{[\lambda]}\right)$, where \[{}^I\left(W_{[\lambda]}\right)
	:=\left\{w\in W_{[\lambda]}: w<_{[\lambda]}s_\alpha w \ \text{for all }\alpha\in I\right\}.
	\]
	By Lemma \ref{JW}, ${}^IW_{[\lambda]}$ is the set of minimal length right coset representatives of $\left(W_{[\lambda]}\right)_I$ in $ W_{[\lambda]}$.
	
	Similarly, since ${\Sigma_\mu}\subseteq \Delta_{[\lambda]}$, ${}^IW_{[\lambda]}^{\Sigma_\mu}$ can be defined by using the Bruhat ordering on $W_{[\lambda]}$ instead of $W$, i.e., ${}^IW_{[\lambda]}^{\Sigma_\mu}
	={}^I\left(W_{[\lambda]}\right)^{\Sigma_\mu}$, where \[
	{}^I\left(W_{[\lambda]}\right)^{\Sigma_\mu}
	:=\left\{w\in {}^I\left(W_{[\lambda]}\right): w<_{[\lambda]}ws_\alpha\in {}^I\left(W_{[\lambda]}\right)\ \text{for all }\alpha\in {\Sigma_\mu}\right\}.
	\]
	In particular, it holds that $W_{[\lambda]}^{\Sigma_\mu}
	=\left(W_{[\lambda]}\right)^{\Sigma_\mu}$. 
	%is the set of minimal length left coset representatives of $\left(W_{[\lambda]}\right)_{\Sigma_\mu}$ in $ W_{[\lambda]}$. 
\end{remark}

\begin{lemma}
	[{See \cite[Proposition 5.4]{EHP}}]
	\label{simplemod} 
	There is a bijection
	\begin{align*}
	{}^IW_{[\lambda]}^{\Sigma_\mu} &\simeq \{\text{simple modules in } \cO^\p_\mu\} / \cong\\
	x&\mapsto [L(w_Ix\cdot\mu)]
	\end{align*}
	where $[L(w_Ix\cdot\mu)]$ is the isomorphism class of $L(w_Ix\cdot\mu)$.
\end{lemma}

\begin{proof}
	Recall that ${}^I W_{[\lambda]}^{\Sigma_\mu}={}^I W_{[\mu]}^{\Sigma_\mu}$. Then by \cite[Proposition 2.2]{BBDN}, the assertion holds for the case when $\mu$ is integral.
	By Corollary \ref{category}, there is a bijection between the set of simple modules in $\cO_{\mu^\natural}^\natural$ up to isomorphism and the set of simple modules in $\cO_\mu$ up to isomorphism. More precisely, there is a bijection
	\begin{align*}
	\left\{[L(x\cdot\mu^{\natural})] : x\in \left(W_{[\mu]}\right)_{[\mu^\natural]}=W_{[\mu]}\right\}
	&\simeq \left\{[L(x\cdot\mu)] : x\in W_{[\mu]}\right\}  \\
	[L(x\cdot\mu^\natural)]&\mapsto [L(x\cdot\mu)].
	\end{align*}
	Since $\lambda\in\Lambda_I^+$, we have $I\subseteq\Delta_{[\lambda]}$ by the remark following Corollary \ref{category}. Then for all $x\in W_{[\lambda]}=W_{[\mu]}$, we have $x^{-1}\Phi_I\subseteq \Phi_{[\lambda]}$ by Theorem \ref{intweylgp}. For all $\alpha\in I$ and $x\in W_{[\mu]}$, we get \[
	\langle x\cdot\mu,\alpha^\lor\rangle
	=\langle \mu,(x^{-1}\alpha)^\lor\rangle
	=\langle\mu^\natural,(x^{-1}\alpha)^\lor\rangle
	=\langle x\cdot\mu^\natural,\alpha^\lor\rangle
	\]
	since $x^{-1}\alpha\in \Phi_{[\lambda]}=\Phi_{[\mu]}$.
	Hence for all $x\in W_{[\mu]}$, we have $x\cdot\mu\in \Lambda_I^+\iff x\cdot\mu^\natural\in \Lambda_I^+$. Then there is a bijection
	\begin{align*}
	\left\{[L(x\cdot\mu^{\natural})] : x\in W_{[\mu]}, x\cdot\mu^\natural\in \Lambda_I^+\right\}
	&\simeq 
	\left\{[L(x\cdot\mu)] : x\in W_{[\mu]}, x\cdot\mu\in \Lambda_I^+\right\} 
	\\
	[L(x\cdot\mu^\natural)]&\mapsto [L(x\cdot\mu)].
	\end{align*}
	By Proposition \ref{9.3} and Theorem \ref{9.4}, it holds that
	\[
	\left\{[L(x\cdot\mu^{\natural})] : x\in W_{[\mu]}, x\cdot\mu^\natural\in \Lambda_I^+\right\}
	=
	\left\{\text{simple modules in } (\cO^\natural)^{\p^\natural}_{\mu^\natural}\right\}/\cong
	\]
	and \[
	\left\{[L(x\cdot\mu)] : x\in W_{[\mu]}, x\cdot\mu\in \Lambda_I^+\right\} 
	=\left\{\text{simple modules in } \cO^\p_{\mu}\right\}/\cong.
	\]
	Since $\mu^\natural\in E(\lambda)$ is $\Phi_{[\lambda]}$-integral, by the integral case, we get a bijection
	\begin{align*}
	{}^I\left(W_{[\lambda]}\right)^{{\Sigma}^\natural_{\mu^\natural}}
	&\simeq
	\left\{\text{simple modules in } (\cO^\natural)^\mathfrak{p^\natural}_{\mu^\natural}\right\}/\cong  \\
	x&\mapsto [L(w_Ix\cdot\mu^\natural)].
	\end{align*}
	Note that
	\[{\Sigma_\mu}
	=\{\alpha\in\Delta_{[\lambda]} :\langle\mu+\rho,\alpha^\lor\rangle=0\}
	=\{\alpha\in\Delta_{[\lambda]} :\langle\mu^\natural+\rho_{[\lambda]},\alpha^\lor\rangle=0\}
	={\Sigma}^\natural_{\mu^\natural}.
	\]
	By the remark following Lemma \ref{bruhatorders}, this implies ${}^IW_{[\lambda]}^{\Sigma_\mu}
	={}^I\left(W_{[\lambda]}\right)^{\Sigma_\mu}
	={}^I\left(W_{[\lambda]}\right)^{{\Sigma}^\natural_{\mu^\natural}}$. The claim follows.
\end{proof}

\begin{remark}
	Suppose $\lambda\in\Lambda_I^+$, then by Theorem \ref{9.4}, it holds that $L(\lambda)$ is a simple module in $\mathcal{O}_\mu^\p$. By Lemma \ref{simplemod}, there exists a unique element $\ow \in  {}^IW_{[\lambda]}^{\Sigma_\mu}$ such that $\lambda=w_I\ow\cdot \mu$.
\end{remark}

%Denote $C_\mathfrak{l}:=\{\nu\in\h^*:\langle\nu,\alpha^\lor\rangle\ge 0, \ \forall \alpha\in I\}$ the $\mathfrak{l}$-dominant chamber. 
%\textbf{Remark:}$\Lambda^+_I=\overline{C_\mathfrak{l}}$.

We will express Kazhdan-Lusztig polynomials in terms of Ext groups in order to relate relative Kazhdan-Lusztig-Vogan polynomials and parabolic Kazhdan-Lusztig polynomials. 
Before that, we need a result which relates four different partial orderings. Two of these are the Bruhat orderings on $W$ and $W_{[\lambda]}$, and the remaining two are partial orderings on $\mathfrak{h}^*$ as defined below.

\begin{definition}[{See \cite[\S0.6 and \S0.7]{HJ}}]
	Let $\le$ denote
	the partial ordering
	on $\h^*$ with $\nu\le \eta$ if and only if $\eta-\nu\in\Gamma$, where $\Gamma$ is defined to be the set of all $\mathbb{Z}^{\ge 0}$-linear combinations of simple roots.
\end{definition}

\def\up{\mathbin{\uparrow_{[\lambda]}}}

\begin{definition}[{See \cite[\S5.1]{HJ}}]
	Given $\eta, \nu \in \h^*$, write $\nu \up \eta$ if $\nu = \eta$ or there is a root $\alpha\in \Phi_{[\lambda]}^+$ such
	that $\nu = s_\alpha \cdot \eta < \eta$ or equivalently $\langle\eta + \rho, \alpha^\lor\rangle \in \mathbb{Z}^{>0}$. 
	If
	$\nu = \eta$ or there exist $\alpha_1, \cdots, \alpha_r \in \Phi_{[\lambda]}^+$ such that
	\[
	\nu = (s_{\alpha_1} \cdots s_{\alpha_r}) \cdot \eta \up (s_{\alpha_2} \cdots s_{\alpha_r}) \cdot \eta \up \cdots \uparrow s_{\alpha_r} \cdot \eta \up \eta,
	\]
	we say that $\nu$ is \emph{$[\lambda]$-strongly linked} to $\eta$ and write $\nu \up \eta$.
	When $\lambda$ is integral, we say that $\nu$ is \emph{strongly linked} to $\eta$ and write $\nu \uparrow \eta$.
\end{definition}

\begin{remark} 
	Note that it is clear that $\nu \uparrow_{[\lambda]} \eta$ implies $\nu \uparrow \eta$.
\end{remark}

We need the following lemma to relate the partial orderings.

\begin{lemma}[{See \cite[Lemma 1.3.1]{AF}}]
	\label{distinct} 
	Each $w\in W_{[\lambda]}$ can be expressed as $w=s_{\beta_l}\cdots s_{\beta_{2}} s_{\beta_1}$ for some distinct positive roots $\{\beta_1, \beta_2, \cdots, \beta_l\}\subseteq \Phi_{[\lambda]}^+$.
\end{lemma}

\begin{proof}
	Apply \cite[Lemma 1.3.1]{AF} to $W_{[\lambda]}$. Note that $W_{[\lambda]}$ is the Weyl group of $\Phi_{[\lambda]}$ (cf. Theorem \ref{intweylgp}) and $s_{\alpha}=s_{-\alpha}$ for all $\alpha\in \Phi_{[\lambda]}^+$.
\end{proof}

Now we are able to relate the partial orderings.
\begin{lemma}\label{HJtypo}
	For all $x,w\in W_{[\lambda]}$, it holds that
	$x\le w\iff x\le_{[\lambda]} w\implies x\cdot\mu\uparrow_{[\lambda]} w\cdot\mu\implies x\cdot\mu\le w\cdot\mu$. 
	If further assume $\lambda\in\Lambda_I^+$ is regular, then the following statements are equivalent:
	\begin{enumerate}
		\item $x\le w$.
		\item $x\le_{[\lambda]} w$.
		\item $x\cdot\mu\uparrow_{[\lambda]} w\cdot \mu$.
		\item $x\cdot\mu\le w\cdot \mu$.
	\end{enumerate}
	Recall that $\mu$ is the unique antidominant weight in $W_{[\lambda]}\cdot\lambda$.
\end{lemma}

\begin{proof}	
	Note that the equivalence $x\le w\iff x\le_{[\lambda]} w$ is true for all $x,w\in W_{[\lambda]}$ by Lemma \ref{bruhatorders}.
	Recall that $W_{[\lambda]}$ is the Weyl group of $\Phi_{[\lambda]}$ (cf. Theorem \ref{intweylgp}). 
	For all $\alpha\in\Phi_{[\lambda]}^+$ and $w\in W_{[\lambda]}$, we have
	\begin{align*}
	w\cdot \mu-s_\alpha\cdot(w\cdot \mu)
	&=w\cdot \mu-(s_\alpha(w\cdot \mu+\rho)-\rho)\\
	&=w\cdot \mu-(w\cdot \mu+\rho-\langle w\cdot \mu+\rho,\alpha^\lor\rangle\alpha -\rho)\\
	&=\langle w\cdot \mu+\rho,\alpha^\lor\rangle\alpha \tag{$**$}
	\end{align*}
	so
	\begin{align*}
		s_\alpha w<_{[\lambda]}w
		&\iff \ell_{[\lambda]}(s_\alpha w)<\ell_{[\lambda]}(w)&\text{(by Definition of the Bruhat ordering on $W_{[\lambda]}$)}\\		
		&\iff w^{-1}\alpha\in-\Phi_{[\lambda]}^+ & \text{(by \cite[\S 0.3, Standard fact (4)]{HJ})}\\
		&\implies \langle\mu+\rho,(w^{-1}\alpha)^\lor
		\rangle\in\mathbb{Z}^{\ge 0} & \text{(by Proposition \ref{antidominant} and $\Phi_{[\lambda]}=\Phi_{[\mu]}$)}\\
		&\iff \langle w(\mu+\rho),\alpha^\lor
		\rangle\in\mathbb{Z}^{\ge 0}&\\
		&\iff \langle w\cdot\mu+\rho,\alpha^\lor
		\rangle\in\mathbb{Z}^{\ge 0}&\\
		&\iff s_\alpha\cdot(w\cdot \mu)\le w\cdot \mu& \text{(by ($**$))}\\
		&\iff (s_\alpha w)\cdot \mu\le w\cdot \mu.
	\end{align*}

	%	Write $v \xrightarrow[{[\lambda]}]{s_{\alpha}} w$ if
	%	$\alpha \in \Phi^+_{[\lambda]}$ and 
	%	$w = s_\alpha v$ and $\ell_{[\lambda]}(w) = \ell_{[\lambda]}(v)+1$.
	
	Write $u \xrightarrow[{[\lambda]}]{s_{\alpha}} v$ if
	$\alpha \in \Phi^+_{[\lambda]}$ and 
	$v = s_\alpha u$ and $\ell_{[\lambda]}(u)<\ell_{[\lambda]}(v)$. The definition of the Bruhat ordering on $W_{[\lambda]}$ implies that for all $\alpha\in\Phi_{[\lambda]}^+$ and $w\in W_{[\lambda]}$, we have $s_\alpha w\xrightarrow[{[\lambda]}]{s_{\alpha}} w\iff s_\alpha w<_{[\lambda]}w$. We deduce from above that
	$
	x\le_{[\lambda]} w
	$
	if and only if $x=w$ or
	\[
	%x=w \ \text{or } 
	x=s_{\alpha_1}\cdots s_{\alpha_r}w
	\xrightarrow[{[\lambda]}]{s_{\alpha_1}}
	s_{\alpha_2}\cdots s_{\alpha_r}w
	\xrightarrow[{[\lambda]}]{s_{\alpha_2}}
	\cdots 
	\xrightarrow[{[\lambda]}]{s_{\alpha_{r-1}}}
	s_{\alpha_r}w
	\xrightarrow[{[\lambda]}]{s_{\alpha_r}}
	w
	\]
	for some $\alpha_1,\cdots, \alpha_r\in\Phi_{[\lambda]}^+$,
	which holds
	if and only if $x=w$ or
	\[
	%x=w \ \text{or } 
	x=s_{\alpha_1}\cdots s_{\alpha_r}w
	<_{[\lambda]}s_{\alpha_2}\cdots s_{\alpha_r}w
	<_{[\lambda]}\cdots <_{[\lambda]}s_{\alpha_r}w
	<_{[\lambda]}w
	\]
	for some $\alpha_1,\cdots, \alpha_r\in\Phi_{[\lambda]}^+$.
	This condition  
	implies that $x\cdot \mu = w\cdot \mu$ or \[
	%x\cdot\mu=w\cdot\mu \ \text{or } 
	x\cdot\mu=(s_{\alpha_1}\cdots s_{\alpha_r}w)\cdot\mu
	\le(s_{\alpha_2}\cdots s_{\alpha_r}w)\cdot\mu
	\le\cdots\le (s_{\alpha_r} w)\cdot\mu
	\le w\cdot\mu
	\]
	for some $\alpha_1,\cdots, \alpha_r\in\Phi_{[\lambda]}^+$,
	which holds
	if and only if
	$x\cdot\mu=w\cdot\mu$ or 
	\[
	%x\cdot\mu=w\cdot\mu \ \text{or } 
	x\cdot\mu=(s_{\alpha_1}\cdots s_{\alpha_r})\cdot(w\cdot\mu)\uparrow_{[\lambda]} (s_{\alpha_2}\cdots s_{\alpha_r})\cdot(w\cdot\mu)\uparrow_{[\lambda]} \cdots \uparrow_{[\lambda]}  s_{\alpha_r}\cdot(w\cdot\mu)\uparrow_{[\lambda]}  w\cdot\mu
	\]
	for some $\alpha_1,\cdots, \alpha_r\in\Phi_{[\lambda]}^+$,
	which finally is equivalent to
	\[
	x\cdot\mu\uparrow_{[\lambda]} w\cdot\mu.
	\]
	By the definition of $[\lambda]$-strong linkage, we have $x\cdot\mu\uparrow_{[\lambda]} w\cdot\mu\implies x\cdot\mu\le w\cdot\mu$.
	We conclude the first statement as needed.

	Now suppose $\lambda\in\Lambda_I^+$ is regular. 
	By the first statement, it suffices to show the implication (4)$\implies$(2) to obtain the second statement. For all $\alpha\in\Phi_{[\lambda]}^+$ and $w\in W_{[\lambda]}$,
	%	recall that we have
	%	\begin{align*}
	%	w\cdot \mu-s_\alpha\cdot(w\cdot \mu)
	%	%	&=w\cdot \mu-(s_\alpha(w\cdot \mu+\rho)-\rho)\\
	%	%	&=w\cdot \mu-(w\cdot \mu+\rho-\langle w\cdot \mu+\rho,\alpha^\lor\rangle\alpha -\rho)\\
	%	&=\langle w\cdot \mu+\rho,\alpha^\lor\rangle\alpha. \tag{$**$}
	%	\end{align*}
	%	Then 
	we have the following equivalent statements:
	\begin{align*}
		s_\alpha w<_{[\lambda]}w
		&\iff \ell_{[\lambda]}(s_\alpha w)<\ell_{[\lambda]}(w) &\text{(by Definition of the Bruhat ordering on $W_{[\lambda]}$)}\\
		&\iff w^{-1}\alpha\in-\Phi_{[\lambda]}^+ & \text{(by \cite[\S 0.3, Standard fact (4)]{HJ})}\\
		&\iff \langle\mu+\rho,(w^{-1}\alpha)^\lor
		\rangle\in\mathbb{Z}^{> 0} & \text{(by Proposition \ref{antidominant}, $\mu$ is regular and $\Phi_{[\lambda]}=\Phi_{[\mu]}$)}\\
		&\iff \langle w(\mu+\rho),\alpha^\lor
		\rangle\in\mathbb{Z}^{> 0}&\\
		&\iff \langle w\cdot\mu+\rho,\alpha^\lor
		\rangle\in\mathbb{Z}^{> 0}&\\
		&\iff s_\alpha\cdot(w\cdot \mu)< w\cdot \mu& \text{by ($**$)}\\
		&\iff (s_\alpha w)\cdot \mu< w\cdot \mu.
	\end{align*}
	
	%	Recall that for all $\alpha\in\Phi_{[\lambda]}^+$ and $w\in W_{[\lambda]}$, we have $s_\alpha w\xrightarrow[{[\lambda]}]{s_{\alpha}} w\iff s_\alpha w<_{[\lambda]}w$. Then for all $\alpha\in\Phi_{[\lambda]}^+$ and $w,w'\in W_{[\lambda]}$, we have $w'\xrightarrow[{[\lambda]}]{s_{\alpha}} w\implies w'\cdot \mu< w\cdot \mu$. 
	Now we can show that (4)$\implies$(2).
	%	Suppose $x<_{[\lambda]}w$. By the definition of the Bruhat ordering on $W_{[\lambda]}$, we get
	%	\[
	%	x
	%	\xrightarrow[{[\lambda]}]{s_{\beta_1}} w_1
	%	\xrightarrow[{[\lambda]}]{s_{\beta_2}}
	%	\cdots 
	%	\xrightarrow[{[\lambda]}]{s_{\beta_k}}
	%	w_k 
	%	\xrightarrow[{[\lambda]}]{s_{\beta_{k+1}}}
	%	w \ \text{for some }\beta_1,\beta_2,\cdots,\beta_{k+1}\in\Phi_{[\lambda]}^+.
	%	\]
	%	Therefore $x\cdot\mu< w_1\cdot\mu< \cdots < w_k\cdot\mu< w\cdot \mu$. Conversely, 
	Suppose $x=s_{\alpha_{1}}\cdots s_{\alpha_k}w$ where $\alpha_{1},\cdots,\alpha_{k}$ are distinct positive roots in $\Phi_{[\lambda]}^+$, which exist by Lemma \ref{distinct}. We then have
	\allowdisplaybreaks
	\begin{align*}
	w\cdot\mu-x\cdot\mu
	&= w\cdot\mu-(s_{\alpha_1}\cdots s_{\alpha_k}w)\cdot\mu\\
	&=w\cdot\mu-s_{\alpha_1}\cdot\left((s_{\alpha_2}\cdots s_{\alpha_k}w)\cdot\mu\right)\\
	&=w\cdot\mu-\left((s_{\alpha_2}\cdots s_{\alpha_k}w)\cdot\mu-\left\langle (s_{\alpha_2}\cdots s_{\alpha_k}w)\cdot\mu+\rho, \alpha_1^\lor\right\rangle\alpha_1\right)\\
	&=w\cdot\mu-(s_{\alpha_2}\cdots s_{\alpha_k}w)\cdot\mu+\left\langle (s_{\alpha_2}\cdots s_{\alpha_k}w)\cdot\mu+\rho, \alpha_1^\lor\right\rangle\alpha_1\\
	%    &=
	%    w\cdot\mu-(s_{\alpha_3}\cdots s_{\alpha_k}w)\cdot\mu
	%    +
	%    \left\langle (s_{\alpha_3}\cdots s_{\alpha_k}w)\cdot\mu+\rho, \alpha_2^\lor\right\rangle\alpha_2
	%    \\&\qquad\qquad+
	%    \left\langle (s_{\alpha_2}\cdots s_{\alpha_k}w)\cdot\mu+\rho, \alpha_1^\lor\right\rangle\alpha_1\\
	&\quad  \vdots\\
	&=w\cdot\mu-(s_{\alpha_k}w)\cdot\mu+\sum_{i=1}^{k-1}\left\langle (s_{\alpha_{i+1}}\cdots s_{\alpha_k}w)\cdot\mu+\rho, \alpha_i^\lor\right\rangle\alpha_i\\
	&=\sum_{i=1}^{k-1}\left\langle (s_{\alpha_{i+1}}\cdots s_{\alpha_k}w)\cdot\mu+\rho, \alpha_i^\lor\right\rangle\alpha_i+\left\langle w\cdot\mu+\rho, \alpha_k^\lor\right\rangle\alpha_k\\
	&=\sum_{i=1}^{k-1}\left\langle (s_{\alpha_{i+1}}\cdots s_{\alpha_k}w)(\mu+\rho), \alpha_i^\lor\right\rangle\alpha_i+\left\langle w(\mu+\rho), \alpha_k^\lor\right\rangle\alpha_k\\
	&=\sum_{i=1}^{k-1}\left\langle \mu+\rho, \left((s_{\alpha_{i+1}}\cdots s_{\alpha_k}w)^{-1}\alpha_i\right)^\lor\right\rangle\alpha_i+\left\langle \mu+\rho, (w^{-1}\alpha_k)^\lor\right\rangle\alpha_k.
	\end{align*}
	Since
	$\mu$ is regular, antidominant and since, by Lemma \ref{distinct}, $\alpha_1,\cdots,\alpha_k$ are distinct positive roots, it follows that $x\cdot\mu<w\cdot\mu$ holds if and only if
	the following equivalent conditions hold:
	\begin{itemize}
		
		\item $ \left\langle \mu+\rho, \left((s_{\alpha_{i+1}}\cdots s_{\alpha_k}w)^{-1}\alpha_i\right)^\lor\right\rangle\in\mathbb{Z}^{>0}\ \text{ for all $1\le i\le k-1$ and } \\ 
		\left\langle \mu+\rho, (w^{-1}\alpha_k)^\lor\right\rangle\in\mathbb{Z}^{>0}$.
		\item  $(s_{\alpha_{i+1}}\cdots s_{\alpha_k}w)^{-1}\alpha_i\in -\Phi_{[\lambda]}^+ \ \text{ for all $1\le i\le k-1$ and } w^{-1}\alpha_k\in -\Phi_{[\lambda]}^+$.
		\item $\ell_{[\lambda]}(s_{\alpha_{i}}s_{\alpha_{i+1}}\cdots s_{\alpha_k}w)<\ell_{[\lambda]}(s_{\alpha_{i+1}}\cdots s_{\alpha_k}w) \ \text{ for all $1\le i\le k-1$ and } \\ \ell_{[\lambda]}(s_{\alpha_{k}}w)<\ell_{[\lambda]}(w)$.
		\item $s_{\alpha_{i}}s_{\alpha_{i+1}}\cdots s_{\alpha_k}w<_{[\lambda]}s_{\alpha_{i+1}}\cdots s_{\alpha_k}w \ \text{ for all $1\le i\le k-1$ and } s_{\alpha_{k}}w<_{[\lambda]}w$.
	\end{itemize}
	The equivalence of these conditions follows from Proposition~\ref{antidominant}, the fact that $\Phi_{[\lambda]}=\Phi_{[\mu]}$, \cite[\S 0.3, Standard fact (4)]{HJ} and the definition of the Bruhat ordering on $W_{[\lambda]}$.
	The last property implies that 
	$x=s_{\alpha_{1}}\cdots s_{\alpha_k}w <_{[\lambda]}\cdots <_{[\lambda]}s_{\alpha_{k-1}}s_{\alpha_k}w<_{[\lambda]}s_{\alpha_{k}}w<_{[\lambda]}w$.
	Therefore, $x\cdot\mu<w\cdot \mu\implies x<_{[\lambda]} w$. 
	Because $\mu$ is regular,
	we have $x\cdot \mu = w\cdot \mu$ if and only if $x=w$,
	so this prove that (4)$\implies$(2).	
	This completes the proof.
\end{proof}

We have a result which relates Ext groups and the Bruhat ordering on $W$.

\begin{lemma}[{See \cite[Theorem 6.11]{HJ}}]
	\label{Extnonzero} 
	Let $u,v\in W$. If $\mathrm{Ext}_{\cO}^i(M(u\cdot(-2\rho)), L(v\cdot(-2\rho)))\neq \{0\}$ for some $i\ge 0$ then $u\le v$.
\end{lemma}

\begin{proof}
	Suppose $\mathrm{Ext}_{\cO}^i(M(u\cdot(-2\rho)), L(v\cdot(-2\rho)))\neq \{0\}$ for some $i\ge 0$. 
	
	For $i=0$, we have $\mathrm{Hom}_{\cO}(M(u\cdot(-2\rho)), L(v\cdot(-2\rho)))\neq \{0\}$. Let $\varphi:M(u\cdot(-2\rho))\to L(v\cdot(-2\rho))$ be a nonzero $\mathfrak{g}$-module homomorphism and $v^+$ be a maximal vector of weight $u\cdot(-2\rho)$ in $M(u\cdot(-2\rho))$. Then $\varphi(v^+)\neq 0$. Let $\mathfrak{n}:=\bigoplus_{\alpha>0}\mathfrak{g}_\alpha$. Then for all $n\in\mathfrak{n}$, $n\cdot\varphi(v^+)=\varphi(n\cdot v^+)=\varphi(0)=0$. 
	Let $\eta=u\cdot(-2\rho)$ and $\nu=v\cdot(-2\rho)$. Then for all $h\in\mathfrak{h}$, $h\cdot\varphi(v^+)=\varphi(h\cdot v^+)=\varphi(\eta(h) v^+)=\eta(h)\varphi(v^+)$. Hence $\varphi(v^+)$ is a maximal vector of weight $u\cdot(-2\rho)$ in $L(v\cdot(-2\rho))$. Then by \cite[Theorem 1.2]{HJ}, $\varphi(v^+)$ is a maximal vector of weight $v\cdot(-2\rho)$ in $L(v\cdot(-2\rho))$ since $L(v\cdot(-2\rho))$ is simple. Hence for all $h\in\mathfrak{h}$, $\eta(h)\varphi(v^+)=h\cdot\varphi(v^+)=\nu(h)\varphi(v^+)$, i.e., $u\cdot(-2\rho)=\eta=\nu=v\cdot(-2\rho)$. Since $-2\rho\in E$ is regular, we get $u=v$.
	
	For $i\ge 1$, we have $\mathrm{Ext}_{\cO}^i(M(uw_0\cdot\kappa), L(vw_0\cdot\kappa))\neq \{0\}$ for some $i\ge 0$, where
	$w_0$ is the longest element in $W$ and $\kappa=w_0\cdot(-2\rho)$. Note that $w_0=w_0^{-1}$ and $w_0\Phi^+=-\Phi^+$.
	Then for all $\alpha\in\Phi^+$, we get $\langle \kappa, \alpha^\lor\rangle=\langle -2\rho, (w_0\alpha)^\lor\rangle=-2\langle \rho, (w_0\alpha)^\lor\rangle\in \mathbb{Z}^{\ge 0}$. Hence $\kappa\in\Lambda^+$. Then by \cite[Theorem 6.11]{HJ}, we get $uw_0\cdot\kappa\uparrow vw_0\cdot\kappa$ and hence $u\cdot(-2\rho)\uparrow_{[-2\rho]} v\cdot(-2\rho)$. 
	Note that $-2\rho\in E$ is the unique integral, regular, antidominant weight in $W_{[-2\rho]}\cdot(-2\rho)$ and this implies that $W=W_{[-2\rho]}$. 
	Then by Lemma \ref{HJtypo}, we have $u\le v$. 
\end{proof}

Now we can express Kazhdan-Lusztig polynomials in terms of Ext groups. 

\begin{proposition}[{See \cite[Theorem 8.11]{HJ}}]
	\label{KLKLV} 
	
	For all $u,v\in W_{[\lambda]}$, it holds that
	\[
	P_{u,v}^{[\lambda]}(q)
	=\sum_{i\ge 0} q^{\frac{\ell_{[\lambda]}(u,v)-i}{2}}\dim\mathrm{Ext}_{\cO^\natural}^i(M(u\cdot(-2\rho_{[\lambda]})), L(v\cdot(-2\rho_{[\lambda]}))).
	\]
\end{proposition}

\begin{proof}
	By Kazhdan-Lusztig Conjecture (cf. Property (g) in Section~\ref{kl-sect}) and \cite[Theorem 8.11]{HJ}, for all $u\le_{[\lambda]}v\in W_{[\lambda]}$, we have 
	\[
	P_{u,v}^{[\lambda]}(q)
	=\sum_{i\ge 0} (-1)^{\ell_{[\lambda]}(u,v)-i}q^{\frac{\ell_{[\lambda]}(u,v)-i}{2}}\dim\mathrm{Ext}_{\cO^\natural}^i(M(u\cdot(-2\rho_{[\lambda]})), L(v\cdot(-2\rho_{[\lambda]})))
	\]
	and $\ell_{[\lambda]}(u,v)-i\equiv 1 \ (\text{mod }2) \implies \mathrm{Ext}_{\cO^\natural}^i(M(u\cdot(-2\rho_{[\lambda]})), L(v\cdot(-2\rho_{[\lambda]})))=\{0\}$.
	Then we have \[
	P_{u,v}^{[\lambda]}(q)
	=\sum_{i\ge 0} q^{\frac{\ell_{[\lambda]}(u,v)-i}{2}}\dim\mathrm{Ext}_{\cO^\natural}^i(M(u\cdot(-2\rho_{[\lambda]})), L(v\cdot(-2\rho_{[\lambda]})))
	\]
	for all $u\le_{[\lambda]}v\in W_{[\lambda]}$.
	By Theorem \ref{paraKL}, $P_{u,v}^{[\lambda]}(q)=0$ if $u\not\le_{[\lambda]}v$.
	It suffices to show the RHS is also zero if $u\not\le_{[\lambda]}v$.
	Applying Lemma \ref{Extnonzero} to $\mathfrak{g}^\natural$, it holds that $\mathrm{Ext}_{\cO^\natural}^i(M(u\cdot(-2\rho_{[\lambda]})), L(v\cdot(-2\rho_{[\lambda]})))\neq \{0\}$ for some $i\ge 0$ implies $u\le_{[\lambda]}v$.
	Taking the contrapositive, it holds that $u\not \le_{[\lambda]}v$ implies $\mathrm{Ext}_{\cO^\natural}^i(M(u\cdot(-2\rho_{[\lambda]})), L(v\cdot(-2\rho_{[\lambda]})))=\{0\}$ for all $i\ge 0$. Hence the RHS is zero if $u\not \le_{[\lambda]}v$. The claim follows. 
\end{proof}

For arbitrary $\lambda\in\Lambda_I^+$, we can relate relative Kazhdan-Lusztig-Vogan polynomials associated to $\lambda$ and parabolic Kazhdan-Lusztig polynomials of ${}^{\Sigma_\mu}W_{[\lambda]}$ of type $q$ by the results due to Soergel \cite{WS1} and Irving \cite{RSI}:

%In fact, we can generalize Theorem \ref{KLpoly} to a different form;
%for all $\lambda\in\Lambda_I^+$, we can relate the relative Kazhdan-Lusztig-Vogan polynomial associated to $\lambda$ of $x,w\in {}^IW_{[\lambda]}^{\Sigma_\mu}$ to a parabolic Kazhdan-Lusztig polynomial of ${}^{\Sigma_\mu}W_{[\lambda]}$ of type $q$ as follows. 

\begin{theorem}\label{KLVparaKL}
	For all $x,w\in {}^IW_{[\lambda]}^{\Sigma_\mu}$, it holds that \[
	{}^IP_{x,w}^{\Sigma_\mu}(q)
	=\sum_{t\in \left(W_{[\lambda]}\right)_{\Sigma_{\mu}}}
	(-1)^{\ell_{[\lambda]}(t)} P^{[\lambda]}_{w_Ixt,w_Iw}(q)
	=P^{[\lambda],{\Sigma_\mu},q}_{(w_Ix)^{-1},(w_Iw)^{-1}}(q).
	\]
	%    where the polynomial on the RHS is the parabolic Kazhdan-Lusztig polynomial of $(w_Ix)^{-1},(w_Iw)^{-1}\in {}^{\Sigma_\mu}W_{[\lambda]}$ of type $q$.
\end{theorem}

%This result enables us to express ${}^IP^{\Sigma_\mu}_{x,w}$ in terms of a parabolic Kazhdan-Lusztig polynomial and therefore ${}^IP^{\Sigma_\mu}_{x,w}$ is indeed a polynomial. This result also implies that ${}^IP_{x,w}^{\Sigma_\mu}(q)=P_{w_Ix,w_Iw}^{\Sigma_\mu}(q)$ and hence $\cW_I(\lambda)\subseteq \cW(\lambda)$ for all $I\subseteq\Delta$.

%Lemma \ref{beatifulformula} also generalizes Theorem \ref{KLpoly}. 
%We refer the reader to \cite[\S 6]{HX} for an expression of ${}^IP^{\Sigma_\mu}_{x,w}$ in terms of regular parabolic Kazhdan-Lusztig-Vogan polynomials when $\lambda$ is integral.	

\begin{proof}
	It holds that  
	\begin{equation}\label{Ext}
	\mathrm{Ext}_{(\cO^\natural)^{\p^\natural}}^i(M_I(w_Ix\cdot\mu^\natural), L(w_Iw\cdot\mu^\natural))\cong \mathrm{Ext}_{\cO^\natural}^i(M(w_Ix\cdot\mu^\natural), L(w_Iw\cdot\mu^\natural)) 	
	\end{equation}
	for all $x,w\in {}^I\left(W_{[\lambda]}\right)^{\Sigma^\natural_{\mu^\natural}}$.
	This isomorphism is well-known; see \cite[\S 9.2]{BBMH} and \cite[\S 6]{HX}. This isomorphism can be proved by using the argument of the Lyndon-Hochschild-Serre spectral sequence as in \cite[Chapter 15]{TEBS}.
	Recall that ${}^IW_{[\lambda]}^{\Sigma_{\mu}}={}^I\left(W_{[\lambda]}\right)^{\Sigma^\natural_{\mu^\natural}}$.
	It holds that ${}^I\left(W_{[\lambda]}\right)^{{\Sigma}^\natural_{\mu^\natural}}
	={}^I\left(W_{[\lambda]}\right)\cap w_I\left(W_{[\lambda]}\right)^{{\Sigma}^\natural_{\mu^\natural}}$ by applying \cite[Corollary 2.2]{BBDN} to $\g^\mathfrak{\natural}$.
	Then for all $x,w\in {}^IW_{[\lambda]}^{\Sigma_{\mu}}$, we have 
	$w_Ix, w_Iw\in w_I{}^I\left(W_{[\lambda]}\right)\cap \left(W_{[\lambda]}\right)^{{\Sigma}^\natural_{\mu^\natural}}\subseteq \left(W_{[\lambda]}\right)^{{\Sigma}^\natural_{\mu^\natural}}$,
	i.e., $w_Ix$ and $w_Iw$ are both the minimal length left coset representatives of $\left(W_{[\lambda]}\right)_{\Sigma^\natural_{\mu^\natural}}$ in $W_{[\lambda]}$. 
	Let $M(u):=M(u\cdot(-2\rho_{[\lambda]}))$ and $L(u):=L(u\cdot(-2\rho_{[\lambda]}))$.	
	Then by the Nil-cohomology Theorem due to Soergel (see \cite[page 566]{WS1}) or the result due to Irving (see \cite[Theorem 1.3.1 and Lemma 1.3.2]{RSI}), we have 
	\begin{align*}
	&\dim\mathrm{Ext}_{\cO^\natural}^i(M(w_Ix\cdot\mu^\natural), L(w_Iw\cdot\mu^\natural))\\
	&=\sum_{t\in \left(W_{[\lambda]}\right)_{\Sigma^\natural_{\mu^\natural}}}(-1)^{\ell_{[\lambda]}(t)}
	\dim\mathrm{Ext}_{\cO^\natural}^{i-\ell_{[\lambda]}(t)}(M(w_Ixt), L(w_Iw))
	\end{align*}
	for all $x,w\in {}^IW_{[\lambda]}^{\Sigma_{\mu}}$. Recall that ${\Sigma^\natural_{\mu^\natural}}
	=\Sigma_{\mu}$. This implies that $\left(W_{[\lambda]}\right)_{\Sigma^\natural_{\mu^\natural}}
	=\left(W_{[\lambda]}\right)_{\Sigma_{\mu}}$ and ${}^{\Sigma^\natural_{\mu^\natural}}\left(W_{[\lambda]}\right)
	={}^{\Sigma_{\mu}}\left(W_{[\lambda]}\right)$.
	For all $x,w\in {}^IW_{[\lambda]}^{\Sigma_{\mu}}$ and $t\in \left(W_{[\lambda]}\right)_{\Sigma_{\mu}}$, it holds that $w_I\in \left(W_{[\lambda]}\right)_{I}$, $x,w\in {}^I\left(W_{[\lambda]}\right)$, $t^{-1}\in \left(W_{[\lambda]}\right)_{\Sigma_{\mu}}$ and $(w_Ix)^{-1}\in {}^{\Sigma^\natural_{\mu^\natural}}\left(W_{[\lambda]}\right)
	={}^{\Sigma_{\mu}}\left(W_{[\lambda]}\right)$. Then by applying \cite[Proposition 3.4 and Remark 3.6]{EHP} to $W_{[\lambda]}$ and the fact that $\ell_{[\lambda]}(v)=\ell_{[\lambda]}(v^{-1})$ for all $v\in W_{[\lambda]}$, we get $\ell_{[\lambda]}(w_Iw)=\ell_{[\lambda]}(w_I)+\ell_{[\lambda]}(w)$ and $\ell_{[\lambda]}(w_Ixt)
	=\ell_{[\lambda]}((w_Ixt)^{-1})
	=\ell_{[\lambda]}(t^{-1}(w_Ix)^{-1})
	=\ell_{[\lambda]}(t^{-1})+\ell_{[\lambda]}((w_Ix)^{-1})
	=\ell_{[\lambda]}(t)+\ell_{[\lambda]}(w_Ix)
	=\ell_{[\lambda]}(t)+\ell_{[\lambda]}(w_I)+\ell_{[\lambda]}(x)
	=\ell_{[\lambda]}(w_I)+\ell_{[\lambda]}(x)+\ell_{[\lambda]}(t)$.
	Hence for all $x,w\in {}^IW_{[\lambda]}^{\Sigma_{\mu}}$, we have
	\allowdisplaybreaks
	\begin{align*}
	&{}^IP^{\Sigma_\mu}_{x,w}(q)&\\
	&=\sum_{i\ge 0} q^{\frac{\ell_{[\lambda]}(x,w)-i}{2}}\dim\mathrm{Ext}_{\cO^\p}^i(M_I(w_Ix\cdot\mu), L(w_Iw\cdot\mu))&\\
	&=\sum_{i\ge 0} q^{\frac{\ell_{[\lambda]}(x,w)-i}{2}}\dim\mathrm{Ext}_{(\cO^\natural)^{\p^\natural}}^i(M_I(w_Ix\cdot\mu^\natural), L(w_Iw\cdot\mu^\natural))&\\
	&=\sum_{i\ge 0} q^{\frac{\ell_{[\lambda]}(x,w)-i}{2}}\dim\mathrm{Ext}_{\cO^\natural}^i(M(w_Ix\cdot\mu^\natural), L(w_Iw\cdot\mu^\natural))&\\
	&=\sum_{i\ge 0} q^{\frac{\ell_{[\lambda]}(x,w)-i}{2}} \sum_{t\in \left(W_{[\lambda]}\right)_{\Sigma_{\mu}}}(-1)^{\ell_{[\lambda]}(t)}
	\dim\mathrm{Ext}_{\cO^\natural}^{i-\ell_{[\lambda]}(t)}(M(w_Ixt), L(w_Iw))&\\
	&=\sum_{t\in \left(W_{[\lambda]}\right)_{\Sigma_{\mu}}}(-1)^{\ell_{[\lambda]}(t)}
	\sum_{i\ge 0} q^{\frac{\ell_{[\lambda]}(w_Iw)-\ell_{[\lambda]}(w_Ixt)-\left(i-\ell_{[\lambda]}(t)\right)}{2}} \dim\mathrm{Ext}_{\cO^\natural}^{i-\ell_{[\lambda]}(t)}(M(w_Ixt), L(w_Iw))&\\
	&=\sum_{t\in \left(W_{[\lambda]}\right)_{\Sigma_{\mu}}}(-1)^{\ell_{[\lambda]}(t)}
	\sum_{i-\ell_{[\lambda]}(t)\ge 0} q^{\frac{\ell_{[\lambda]}(w_Ixt,w_Iw)-\left(i-\ell_{[\lambda]}(t)\right)}{2}} \dim\mathrm{Ext}_{\cO^\natural}^{i-\ell_{[\lambda]}(t)}(M(w_Ixt), L(w_Iw))&\\
	&=\sum_{t\in \left(W_{[\lambda]}\right)_{\Sigma_{\mu}}}(-1)^{\ell_{[\lambda]}(t)}
	\sum_{i\ge 0} q^{\frac{\ell_{[\lambda]}(w_Ixt,w_Iw)-i}{2}} \dim\mathrm{Ext}_{\cO^\natural}^{i}(M(w_Ixt), L(w_Iw))&\\
	&=\sum_{t\in \left(W_{[\lambda]}\right)_{\Sigma_{\mu}}}(-1)^{\ell_{[\lambda]}(t)} P_{w_Ixt,w_Iw}^{[\lambda]}(q) &\\ 
	&=\sum_{t\in \left(W_{[\lambda]}\right)_{\Sigma_{\mu}}}(-1)^{\ell_{[\lambda]}(t)}P_{(w_Ixt)^{-1},(w_Iw)^{-1}}^{[\lambda]}(q) & \\
	&=\sum_{t^{-1}\in \left(W_{[\lambda]}\right)_{\Sigma_{\mu}}}(-1)^{\ell_{[\lambda]}(t^{-1})}P_{t^{-1}(w_Ix)^{-1},(w_Iw)^{-1}}^{[\lambda]}(q)&\\
	&=\sum_{t\in \left(W_{[\lambda]}\right)_{\Sigma_{\mu}}}(-1)^{\ell_{[\lambda]}(t)}P_{t(w_Ix)^{-1},(w_Iw)^{-1}}^{[\lambda]}(q)&\\
	&=P^{[\lambda],{\Sigma_\mu},q}_{(w_Ix)^{-1},(w_Iw)^{-1}}(q). & 
	\end{align*}
	The first equality follows from Definition \ref{KLV}.
	The second equality follows from Corollary \ref{category} and the fact that $\cO^\p$ and $(\cO^\natural)^{\p^\natural}$ are full subcategories of $\cO$ and $\cO^\natural$, respectively.
	The third equality follows from the isomorphism (\ref{Ext}). 
	The fourth equality follows from the results due to Soergel and Irving.
	The fifth equality follows from the fact that $\ell_{[\lambda]}(w_Iw)=\ell_{[\lambda]}(w_I)+\ell_{[\lambda]}(w)$, $\ell_{[\lambda]}(w_Ixt)=\ell_{[\lambda]}(w_I)+\ell_{[\lambda]}(x)+\ell_{[\lambda]}(t)$ for all $x,w\in {}^IW_{[\lambda]}^{\Sigma_\mu}$ and $t\in \left(W_{[\lambda]}\right)_{\Sigma_{\mu}}$, and the definition of $\ell_{[\lambda]}(u,v)$.
	The sixth equality follows from the definition of $\ell_{[\lambda]}(u,v)$ and the fact that $\mathrm{Ext}_{\cO^\natural}^{k}(M(w_Ixt), L(w_Iw)):=\{0\}$ for all $k\in\mathbb{Z}^{<0}$.
	The seventh equality follows from the replacement of $i-\ell_{[\lambda]}(t)$ by $i$. 
	The eighth equality follows from Proposition \ref{KLKLV} with $u=w_Ixt$ and $v=w_Iw$.
	The ninth equality follows from that fact that $P_{u,v}^{[\lambda]}=P_{u^{-1},v^{-1}}^{[\lambda]}$ (cf. Property (f) in Section~\ref{kl-sect}). The tenth equality follows from that fact that $\ell_{[\lambda]}(t)=\ell_{[\lambda]}(t^{-1})$ for all $t\in W_{[\lambda]}$ and $t\in \left(W_{[\lambda]}\right)_{\Sigma_{\mu}}\iff t^{-1}\in \left(W_{[\lambda]}\right)_{\Sigma_{\mu}}$.
	The eleventh equality follows from the replacement of $t^{-1}$ by $t$.
	The last equality follows from Proposition \ref{paraKL-KL} and $(w_Ix)^{-1}, (w_Iw)^{-1}\in {}^{{\Sigma}_{\mu}}\left(W_{[\lambda]}\right)={}^{{\Sigma}_{\mu}}W_{[\lambda]}$ (cf. Remark following Lemma \ref{bruhatorders}). 
\end{proof}

Now assume $\lambda\in\Lambda_I^+$ is regular. 
Note that $\lambda$ is regular iff $\mu$ is regular.
We have the following result about ${}^IW_{[\lambda]}$. %for regular $\lambda\in\Lambda_I^+$.

%, so $W_I\subseteq W_{[\lambda]}$.

\begin{lemma}\label{IW}
	%The following statements are equivalent for $w \in W_{[\lambda]}$:\begin{enumerate}
	%	\item $s_\alpha w>w$ for all $\alpha\in I$.
	%	\item $w^{-1}\Phi_I^+\subseteq \Phi^+$.
	%	\item $s_\alpha w>_{[\lambda]}w$ for all $\alpha\in I$.    
	%    \item $w^{-1}\Phi_I^+\subseteq \Phi_{[\lambda]}^+$.
	%\item $w\cdot\mu\in -C_{\mathfrak{l}}-\rho$. 
	%\end{enumerate}
	%Here the Bruhat order and length function are defined on $W_{[\lambda]}$.
	It holds that ${}^IW_{[\lambda]}=\{w\in W_{[\lambda]} : w\cdot\mu\in -C_{\mathfrak{l}}-\rho\}$, where $C_\mathfrak{l}:=\{\nu\in\h^*:\langle\nu,\alpha^\lor\rangle\ge 0, \ \forall \alpha\in I\}$. 
\end{lemma}

\begin{proof}
	By the remark following Lemma \ref{bruhatorders}, we get ${}^IW_{[\lambda]}={}^I\left(W_{[\lambda]}\right)$.
	Let $\left(\Phi_{[\lambda]}\right)_I:=\Phi_{[\lambda]}\cap\sum_{\alpha\in I}\mathbb{Z}\alpha$ and $\left(\Phi_{[\lambda]}\right)_I^+:=\left(\Phi_{[\lambda]}\right)_I\cap \Phi_{[\lambda]}^+$.
	Then by Lemma \ref{JW}, ${}^I\left(W_{[\lambda]}\right)=\{w\in W_{[\lambda]} : w^{-1}\left(\Phi_{[\lambda]}\right)_I^+\subseteq \Phi_{[\lambda]}^+\}$.
	It suffices to show the equivalence $w^{-1}\left(\Phi_{[\lambda]}\right)_I^+\subseteq \Phi_{[\lambda]}^+ \iff w\cdot\mu\in -C_{\mathfrak{l}}-\rho$ is true for all $w\in W_{[\lambda]}$.
	%By Lemma \ref{JW}, we have (1)$\iff$(2) and (3)$\iff$(4).
	%For (2)$\iff$(4), suppose $w^{-1}\Phi_I^+\subseteq \Phi^+$. Since $\lambda\in\Lambda_I^+$, we have $I\subseteq\Delta_{[\lambda]}$ and hence $\Phi_I\subseteq \Phi_{[\lambda]}$.Since $w\in W_{[\lambda]}$, we have $w^{-1}\Phi_I^+\subseteq \Phi_{[\lambda]}$ and hence $w^{-1}\Phi_I^+\subseteq \Phi_{[\lambda]}\cap \Phi^+=\Phi_{[\lambda]}^+$.  Conversely, suppose $w^{-1}\Phi_I^+\subseteq \Phi_{[\lambda]}^+$. Since $\Phi_{[\lambda]}^+=\Phi_{[\lambda]}\cap \Phi^+\subseteq \Phi^+$. Then $w^{-1}\Phi_I^+\subseteq \Phi^+$. 
	%For (4)$\iff$(5). %The other equivalences follow from Remark 3.6 in \cite{RMJ} and the fact that $I\subseteq\Delta_{[\lambda]}$.
	Since $\mu$ is regular and antidominant, we have $\langle \mu+\rho,\alpha^\lor\rangle<0$ for all $\alpha\in\Delta_{[\mu]}=\Delta_{[\lambda]}$ by Proposition \ref{antidominant}.
	Consider $w\in W_{[\lambda]}$. Suppose $w^{-1}\left(\Phi_{[\lambda]}\right)_I^+\subseteq \Phi_{[\lambda]}^+$. 
	Then for all $\alpha\in I$, 
	$\langle -w(\mu+\rho),\alpha^\lor\rangle
	=-\langle \mu+\rho,(w^{-1}\alpha)^\lor\rangle>0$.
	Hence $-w(\mu+\rho)\in C_\mathfrak{l}$ or equivalently, $w\cdot\mu\in -C_\mathfrak{l}-\rho$.
	
	Conversely, suppose $w\cdot\mu\in -C_\mathfrak{l}-\rho$. Then $-w(\mu+\rho)\in C_\mathfrak{l}$ 
	and hence 
	$\langle \mu+\rho,(w^{-1}\alpha)^\lor\rangle=-\langle -w(\mu+\rho),\alpha^\lor\rangle\le 0$ for all $\alpha\in I$. 
	Since $\lambda\in\Lambda_I^+$, we have $I\subseteq\Delta_{[\lambda]}$ by the remark following Corollary \ref{category}. Then by Theorem \ref{intweylgp}, we have $w^{-1}\alpha\in \Phi_{[\lambda]}=\Phi_{[\mu]}$ for all $\alpha\in I$ since $w\in W_{[\lambda]}$. Since $\mu$ is regular, we get $\langle \mu+\rho,(w^{-1}\alpha)^\lor\rangle < 0$ for all $\alpha\in I$. 
	Then $w^{-1}\alpha\in \Phi_{[\lambda]}^+$ for all $\alpha\in I$. 
	Therefore, $w^{-1}\left(\Phi_{[\lambda]}\right)_I^+\subseteq \Phi_{[\lambda]}^+$.
	%Finally, by (1)$\iff$(4), we get ${}^IW_{[\lambda]}=\{w\in W_{[\lambda]} : w^{-1}\Phi_I^+\subseteq \Phi_{[\lambda]}^+\}$. Since $\lambda\in\Lambda_I^+$, we have $I\subseteq\Delta_{[\lambda]}$ and hence $W_I=\left(W_{[\lambda]}\right)_I$. We have ${}^IW_{[\lambda]}$ is the set of minimal length right coset representatives of $W_I$ in $ W_{[\lambda]}$ by Lemma \ref{JW}.
\end{proof}

For regular $\lambda\in\Lambda_I^+$, we can relate relative Kazhdan-Lusztig-Vogan polynomials associated to $\lambda$ and parabolic Kazhdan-Lusztig polynomials of ${}^IW_{[\lambda]}$ of type $-1$.

\begin{theorem}\label{KLpoly}
	For all $x,w\in {}^I W_{[\lambda]}$, it holds that 
	\[
	{}^IP_{x,w}^{\mu}(q)=P_{w_Ix,w_Iw}^{[\lambda]}(q)=P^{[\lambda],I,-1}_{x,w}(q).
	\]
	%where the polynomial on the RHS is the Kazhdan-Lusztig polynomial of $w_Ix,w_Iw\in W_{[\lambda]}$.
\end{theorem}

\begin{proof}
	
	%By Jantzen-Zuckerman's translation principle, $\cO^\p_{-2\rho}$ is equivalent to $\cO^\p_{\mu}$ for all regular integral, antidominant weight $\mu$.
	Since $\lambda$ is regular, we have $\Sigma_\mu=\emptyset$ and hence $\left(W_{[\lambda]}\right)_{\Sigma_{\mu}}=\{e\}$. Then for all $x,w\in {}^IW_{[\lambda]}$, we get 
	\begin{align*}
	{}^IP_{x,w}^{\mu}(q)
	&=(-1)^{\ell_{[\lambda]}(e)}P_{w_Ixe,w_Iw}^{[\lambda]}(q) & (\text{by Theorem \ref{KLVparaKL}})\\
	&=P_{w_Ix,w_Iw}^{[\lambda]}(q)& (\text{as $\ell_{[\lambda]}(e)=0$ and $w_Ixe=w_Ix$}) \\   
	&=P^{[\lambda],I,-1}_{x,w}(q).& (\text{by Proposition \ref{paraKL-KL}})
	\end{align*}
\end{proof}

\begin{remark} 
	The equality ${}^IP_{x,w}^{\mu}(q)=P^{[\lambda],I,-1}_{x,w}(q)$ is well-known when $\lambda$ is integral; see \cite[page 822]{HX} and \cite[page 147]{EHP}.
	%, Section 3.4 in \cite{BBMH}, which use Kazhdan-Lusztig conjectures (which has been proven in \cite{CLCD} for category $\cO^\p$) to get the result.
	Note that we have $P_{x,w}^{\mu}(q)=P^{[\lambda]}_{x,w}$.
\end{remark}

%After showing a Kazhdan-Lusztig-Vogan polynomial associated to $\lambda$ is related to a parabolic Kazhdan-Lusztig polynomial  of ${}^IW_{[\lambda]}$ of type $-1$ when $\lambda$ is regular. We use this as an application to determine when 
%${}^IP_{x,w}(1)$ is nonzero.

As an application of Theorem \ref{KLpoly}, we can determine when 
${}^IP_{x,w}^\mu(1)$ is nonzero.

\begin{corollary}\label{keyequiv}
	For all $x,w\in {}^IW_{[\lambda]}$, we have ${}^IP_{x,w}^{\mu}(1)\neq 0$ if and only if $w_Ix\le_{[\lambda]} w_Iw$.
\end{corollary}

\begin{proof}
	For all $x,w\in {}^IW_{[\lambda]}$, we have the following equivalent statements:
	\begin{align*}
	{}^IP_{x,w}^{\mu}(q)=0
	&\iff 
	\dim \mathrm{Ext}^i_{\cO^\p}(M_I(w_Ix\cdot \mu),L(w_Iw\cdot \mu))=0,\  \forall i\ge 0 \\
	&\iff \sum_{i\ge 0}\dim \mathrm{Ext}^i_{\cO^\p}(M_I(w_Ix\cdot \mu),L(w_Iw\cdot \mu))=0 \\
	&\iff {}^IP_{x,w}^{\mu}(1)=0.
	\end{align*}
	By property (a) in Section~\ref{kl-sect} and Theorem \ref{KLpoly}, it holds that
	$w_Ix\not\le_{[\lambda]} w_Iw\implies {}^IP_{x,w}^{\mu}(q)=P_{w_Ix,w_Iw}^{[\lambda]}(q)=0\implies {}^IP_{x,w}^{\mu}(1)=0$. Taking contrapositives, ${}^IP_{x,w}^{\mu}(1)\neq 0\implies w_Ix\le_{[\lambda]} w_Iw$.
	Conversely, by property (e) in Section~\ref{kl-sect} and Theorem \ref{KLpoly}, $w_Ix\le_{[\lambda]} w_Iw\implies {}^IP_{x,w}^{\mu}(q)=P_{w_Ix,w_Iw}^{[\lambda]}(q)$ has constant term $1\implies {}^IP_{x,w}^{\mu}(q)\neq 0\implies {}^IP_{x,w}^{\mu}(1)\neq 0$.
\end{proof}

\section{Parameterizations for Dirac cohomology of $L(\lambda)\in\cO^\p$}
\label{S:4}

\subsection{The general case}

Any simple module $V\in\cO^\p$ is isomorphic to $L(\lambda)$ for some $\lambda\in\Lambda_I^+$ and this implies $H_D(V)\cong H_D(L(\lambda))$ as an $\mathfrak{l}$-module (cf. Theorem \ref{determine'}). We can therefore concentrate on $H_D(L(\lambda))$. In this section we will show that 
$\cW_I(\lambda)$ is a parameterization of $H_D(L(\lambda))$.

%To justify $\cW_I(\lambda)$ is the parameterization of $H_D(L(\lambda))$. 
We will need the following lemma:

\begin{lemma}\label{multilemma}
	For all $x,w\in {}^IW_{[\lambda]}^{\Sigma_\mu}$, it holds that:
	\begin{enumerate}
		\item ${}^IP_{x,w}^{\Sigma_\mu}(1)\neq 0\implies w_Ix\le_{[\lambda]} w_Iw$.
		\item ${}^IP^{\Sigma_\mu}_{w,w}(q)=1$.
		%\item $x\le_{[\lambda]} w\implies x\cdot\mu\uparrow_{[\lambda]} w\cdot\mu\implies x\cdot\mu\le w\cdot\mu$. 
		\item $\{w\cdot\mu\in \h^*: w\in {}^IW_{[\lambda]}\}\subseteq -C_{\mathfrak{l}}-\rho$, where $C_\mathfrak{l}:=\{\nu\in\h^*:\langle\nu,\alpha^\lor\rangle\ge 0, \ \forall \alpha\in I\}$.
	\end{enumerate}
	Recall that $\mu$ is the unique antidominant weight in $W_{[\lambda]}\cdot\lambda$.
\end{lemma}

\begin{proof}
	We prove each part in turn.
	\begin{enumerate}
		\item 
		For all $x,w\in {}^IW_{[\lambda]}^{\Sigma_\mu}$, we have the following equivalent statements:
		\begin{align*}
		{}^IP_{x,w}^{{\Sigma_\mu}}(q)=0
		&\iff 
		\dim \mathrm{Ext}^i_{\cO^\p}(M_I(w_Ix\cdot \mu),L(w_Iw\cdot \mu))=0,\  \forall i\ge 0 \\
		&\iff \sum_{i\ge 0}\dim \mathrm{Ext}^i_{\cO^\p}(M_I(w_Ix\cdot \mu),L(w_Iw\cdot \mu))=0 \\
		&\iff {}^IP_{x,w}^{{\Sigma_\mu}}(1)=0.
		\end{align*}
		
		Therefore by Theorem \ref{KLVparaKL} and Theorem \ref{paraKL},
		\begin{align*}
		w_Ix\not\le_{[\lambda]} w_Iw &\iff (w_Ix)^{-1}\not\le_{[\lambda]} (w_Iw)^{-1}\\ 
		&\implies {}^IP_{x,w}^{\Sigma_\mu}(q)=P^{[\lambda],{\Sigma_\mu},q}_{(w_Ix)^{-1},(w_Iw)^{-1}}(q)=0\\
		&\iff {}^IP_{x,w}^{\Sigma_\mu}(1)=0.    
		\end{align*}
		
		Taking contrapositives, we get that ${}^IP_{x,w}^{\Sigma_\mu}(1)\neq 0\implies w_Ix\le_{[\lambda]} w_Iw$.
		
		\item By Theorem \ref{KLVparaKL}, we have ${}^IP_{x,w}^{\Sigma_\mu}(q)=P^{[\lambda],{\Sigma_\mu},q}_{(w_Ix)^{-1},(w_Iw)^{-1}}(q)$, so by Theorem \ref{paraKL}, we get   \[
		{}^IP_{w,w}^{\Sigma_\mu}(q)=P^{[\lambda],{\Sigma_\mu},q}_{(w_Iw)^{-1},(w_Iw)^{-1}}(q)=1.
		\] 
		\item 
		Since $\mu$ is antidominant, we have $\langle \mu+\rho,\alpha^\lor\rangle\le 0$ for all $\alpha\in\Delta_{[\mu]}=\Delta_{[\lambda]}$ by Proposition \ref{antidominant}.
		Let $w\in {}^IW_{[\lambda]}$, so that 
		$w\cdot \mu\in\mathfrak{h}^*$. By the remark following Lemma~\ref{bruhatorders} and Lemma~\ref{JW}, we get
		$w^{-1}\left(\Phi_{[\lambda]}\right)_I^+\subseteq \Phi_{[\lambda]}^+$. 
		Thus for all $\alpha\in I$, 
		$\langle -w(\mu+\rho),\alpha^\lor\rangle
		=-\langle \mu+\rho,(w^{-1}\alpha)^\lor\rangle\ge 0$.
		Then $-w(\mu+\rho)\in C_\mathfrak{l}$ so $w\cdot\mu\in -C_\mathfrak{l}-\rho$. Hence $\{w\cdot\mu\in \h^*: w\in {}^IW_{[\lambda]}\}\subseteq -C_{\mathfrak{l}}-\rho$.
	\end{enumerate}
\end{proof}

\begin{definition} [{See \cite[\S 1.11]{HJ}}]
	Let $[M : L(\eta)]$ denote the multiplicity of $L(\eta)$ in a Jordan-H\"older series of $M$.
\end{definition}

\begin{definition} [{See \cite[\S 4.2]{HJ}}]
	We write $M(\nu)\hookrightarrow M(\eta)$ to indicate that there is
	an embedding from Verma module $M(\nu)$ into Verma module $M(\eta)$.\end{definition}

\begin{theorem}[{\cite[Theorem 5.1]{HJ}}]
	\label{BGGVerma} 
	Let $\eta, \nu \in \h^*$.
	\begin{enumerate}
		\item (Verma) If $\nu$ is strongly linked to $\eta$, then $M(\nu)\hookrightarrow M(\eta)$; in
		particular, $[M(\eta) : L(\nu)] \neq 0$.
		\item (BGG) If $[M(\eta) : L(\nu)] \neq 0$, then $\nu$ is strongly linked to $\eta$.
	\end{enumerate}
\end{theorem}

\begin{remark} Let $\eta, \nu \in \h^*$. Then $M(\nu)\hookrightarrow M(\eta)$ if and only if $[M(\eta) : L(\nu)] \neq 0$.
\end{remark}

Now we are able to prove Theorem \ref{generalcase'}.

\begin{theorem}[Theorem \ref{generalcase'}]
	\label{generalcase}
	Let $\lambda\in\Lambda_I^+$, $\mathcal{S}_{[\lambda]}(\lambda):=\{\nu\in\h^*: \nu\uparrow_{[\lambda]} \lambda\}$, $C_\mathfrak{l}:=\{\nu\in\h^*:\langle\nu,\alpha^\lor\rangle\ge 0, \ \forall \alpha\in I\}$ and $\mathcal{L}_{\eta}
	:=\{\nu\in\h^*: \nu\le \eta\}$. Then 
	\begin{align*}
	\cW_I(\lambda)
	&\subseteq \left(\mathcal{S}_{[\lambda]}(\lambda)+\rho\right)\cap C_\mathfrak{l}\\
	&\subseteq \{\nu\in\Lambda_I^+-\rho: M(\nu)\hookrightarrow M(\lambda)\}+\rho\\
	&=\{\nu\in\Lambda_I^+-\rho: [M(\lambda),L(\nu)]\neq 0\}+\rho\\
	&\subseteq W_{[\lambda]}(\lambda+\rho)
	\cap \mathcal{L}_{\lambda+\rho}\cap C_\mathfrak{l},
	\end{align*}
\end{theorem}

\begin{proof}
	Let $\ow \in  {}^IW_{[\lambda]}^{\Sigma_\mu}$ be the unique element such that 
	$\lambda=w_I\ow\cdot \mu$, which exists by the remark following Lemma \ref{simplemod}.
	We prove the first inclusion in three steps.
	\begin{itemize}
		\item First we show $w_I\cdot\left(\{x\in {}^IW_{[\lambda]}: w_Ix\le_{[\lambda]} w_I\ow\}\cdot\mu\right)\subseteq w_I\cdot \left({}^IW_{[\lambda]}\cdot\mu\right)\cap \mathcal{S}_{[\lambda]}(\lambda)$.
		
		Clearly, $w_I\cdot\left(\{x\in {}^IW_{[\lambda]}: w_Ix\le_{[\lambda]} w_I\ow\}\cdot\mu\right)\subseteq w_I\cdot \left({}^IW_{[\lambda]}\cdot\mu\right)$.
		
		Suppose $\eta\in w_I\cdot\left(\{x\in {}^IW_{[\lambda]}: w_Ix\le_{[\lambda]} w_I\ow\}\cdot\mu\right)$. Then $\eta=w_Ix\cdot \mu$ with $w_Ix\le_{[\lambda]} w_I\ow$.
		
		By Lemma \ref{HJtypo}, we get $w_Ix\le_{[\lambda]} w_I\ow\implies \eta=w_Ix\cdot\mu\uparrow_{[\lambda]} w_I\ow\cdot\mu=\lambda$.
		Then $\eta\in \mathcal{S}_{[\lambda]}(\lambda)$ and hence 
		$w_I\cdot\left(\{x\in {}^IW_{[\lambda]}: w_Ix\le_{[\lambda]} w_I\ow\}\cdot\mu\right)\subseteq w_I\cdot \left({}^IW_{[\lambda]}\cdot\mu\right)\cap \mathcal{S}_{[\lambda]}(\lambda)$.
		
		\item Next, we show ${}^IW_{[\lambda]}\cdot\mu\subseteq W_{[\lambda]}\cdot \mu\cap(-C_\mathfrak{l}-\rho)$.
		This holds since, by Lemma \ref{multilemma}, 
		\begin{align*}
		{}^IW_{[\lambda]}\cdot\mu
		&=\{w\cdot\mu\in W_{[\lambda]}\cdot\mu: w\in {}^IW_{[\lambda]}\}\\
		&\subseteq\{w\cdot\mu\in W_{[\lambda]}\cdot\mu: w\cdot \mu\in-C_\mathfrak{l}-\rho\}\\
		&=W_{[\lambda]}\cdot \mu\cap(-C_\mathfrak{l}-\rho).    
		\end{align*}
		
		\item Finally, we show $\cW_I(\lambda)
		\subseteq \left(\mathcal{S}_{[\lambda]}(\lambda)+\rho\right)\cap C_\mathfrak{l}$.
		
		Let $X,Y$ be sets. 
		%We have $w_I\cdot (X\cap Y)=w_I\cdot X\cap w_I\cdot Y$ and $X\cap Y+\rho=(X+\rho)\cap (Y+\rho)$.
		Note that $W_I\subseteq W_{[\lambda]}$, $\mu\in W_{[\lambda]}\cdot \lambda$, $-w_IC_\mathfrak{l}=C_\mathfrak{l}$ and $\mathcal{S}_{[\lambda]}(\lambda)\subseteq W_{[\lambda]}\cdot \lambda$. Therefore,
		\allowdisplaybreaks
		\begin{align*}
			\cW_I(\lambda)
			&=\{w_Ix\cdot\mu+\rho(\mathfrak{u}): x\in {}^IW_{[\lambda]}^{\Sigma_\mu}, {}^IP_{x,\ow}^{\Sigma_\mu}(1)\neq 0\}+\rho_\mathfrak{l}& \text{(by Definition \ref{WIlambda})}\\
			&=w_I\cdot\left(\{x\in {}^IW_{[\lambda]}^{\Sigma_\mu}: {}^IP_{x,\ow}^{\Sigma_\mu}(1)\neq 0\}\cdot\mu\right)+\rho&\\
			&\subseteq w_I\cdot\left(\{x\in {}^IW_{[\lambda]}: w_Ix\le_{[\lambda]} w_I\ow\}\cdot\mu\right)+\rho & \text{(by Lemma \ref{multilemma})}\\
			&\subseteq w_I\cdot \left({}^IW_{[\lambda]}\cdot\mu\right)\cap \mathcal{S}_{[\lambda]}(\lambda)+\rho& \text{(by step 1)}\\
			&\subseteq w_I\cdot\left(W_{[\lambda]}\cdot \mu\cap(-C_\mathfrak{l}-\rho)\right)\cap \mathcal{S}_{[\lambda]}(\lambda)+\rho& \text{(by step 2)}\\
			&=w_I\cdot \left(W_{[\lambda]}\cdot \mu\right)\cap w_I\cdot(-C_\mathfrak{l}-\rho)
			\cap \mathcal{S}_{[\lambda]}(\lambda)+\rho& \text{($w_I\cdot (X\cap Y)=w_I\cdot X\cap w_I\cdot Y$)}\\
			&=\left(w_IW_{[\lambda]}\right)\cdot \mu\cap w_I\cdot(-C_\mathfrak{l}-\rho)
			\cap \mathcal{S}_{[\lambda]}(\lambda)+\rho& \text{($w_I\cdot \left(W_{[\lambda]}\cdot \mu\right)=\left(w_IW_{[\lambda]}\right)\cdot \mu$)}\\
			&=W_{[\lambda]}\cdot \mu\cap w_I\cdot(-C_\mathfrak{l}-\rho)
			\cap \mathcal{S}_{[\lambda]}(\lambda)+\rho& \text{($w_IW_{[\lambda]}=W_{[\lambda]}$)}\\
			&=W_{[\lambda]}\cdot \lambda\cap w_I\cdot(-C_\mathfrak{l}-\rho)
			\cap \mathcal{S}_{[\lambda]}(\lambda)+\rho& \text{($W_{[\lambda]}\cdot\mu=W_{[\lambda]}\cdot\lambda$)}\\
			&=W_{[\lambda]}\cdot \lambda\cap\left(-w_IC_\mathfrak{l}-\rho\right)
			\cap \mathcal{S}_{[\lambda]}(\lambda)+\rho& \text{($w_I\cdot(-C_\mathfrak{l}-\rho)=-w_IC_\mathfrak{l}-\rho$)}\\
			&=W_{[\lambda]}\cdot \lambda\cap\left( 
			C_\mathfrak{l}-\rho\right)
			\cap \mathcal{S}_{[\lambda]}(\lambda)+\rho& \text{($-w_IC_\mathfrak{l}=C_\mathfrak{l}$)}\\
			&=\left(W_{[\lambda]}\cdot \lambda+\rho\right)\cap\left( 
			C_\mathfrak{l}-\rho+\rho\right)
			\cap (\mathcal{S}_{[\lambda]}(\lambda)+\rho)& \text{($X\cap Y+\rho=(X+\rho)\cap (Y+\rho)$)}\\
			&=\left(W_{[\lambda]}\cdot \lambda+\rho\right)
			\cap (\mathcal{S}_{[\lambda]}(\lambda)+\rho)\cap C_\mathfrak{l}& \text{($C_\mathfrak{l}-\rho+\rho=C_\mathfrak{l}$)}\\
			&=\left(W_{[\lambda]}\cdot \lambda\cap \mathcal{S}_{[\lambda]}(\lambda) +\rho\right)\cap C_\mathfrak{l}& \text{($X\cap Y+\rho=(X+\rho)\cap (Y+\rho)$)}\\
			&=\left(\mathcal{S}_{[\lambda]}(\lambda) +\rho\right)\cap C_\mathfrak{l}.& \text{($\mathcal{S}_{[\lambda]}(\lambda)\subseteq W_{[\lambda]}\cdot \lambda$)}
		\end{align*}
	\end{itemize}

	Now we prove the second inclusion.
	Suppose $\eta\in \mathcal{S}_{[\lambda]}(\lambda)$ is such that $\eta+\rho\in (\mathcal{S}_{[\lambda]}(\lambda)+\rho)\cap C_\mathfrak{l}$.
	Then, by definition, it holds that $\eta\uparrow_{[\lambda]}\lambda$, which implies $\eta\uparrow \lambda$. By Theorem \ref{BGGVerma}, we get $M(\eta)\hookrightarrow M(\lambda)$. It remains to check that $\eta+\rho \in \Lambda_I^+$.  Because we have $\eta\in \mathcal{S}_{[\lambda]}(\lambda)\subseteq W_{[\lambda]}\cdot \lambda$, it holds that $\eta=w\cdot\lambda$ for some $w\in W_{[\lambda]}$. Then by the definition of $W_{[\lambda]}$, we have $\eta-\lambda\in\Lambda_r$. Recall that $\lambda\in \Lambda_I^+$. Then for all $\alpha\in I$, we get
	\[
	\langle \eta+\rho,\alpha^\lor\rangle
	=\langle \eta-\lambda,\alpha^\lor\rangle+\langle\lambda,\alpha^\lor\rangle+\langle \rho,\alpha^\lor\rangle\in \mathbb{Z}.
	\]
	Since $\eta+\rho\in C_\mathfrak{l}$, we also get $\langle \eta+\rho,\alpha^\lor\rangle\ge 0$ for all $\alpha\in I$. Thus $\eta+\rho\in\Lambda_I^+$ as desired. This proves the second inclusion, and the following equality is 
	clear from the remark after Theorem \ref{BGGVerma}.
	%	Hence, we get 
	%	$W_{[\lambda]}(\lambda+\rho)\cap \left(\mathcal{S}_{[\lambda]}(\lambda)+\rho\right)\cap C_\mathfrak{l}\subseteq \{\nu\in\Lambda_I^+-\rho: M(\nu)\hookrightarrow M(\lambda)\}+\rho
	%	=\{\nu\in\Lambda_I^+-\rho: [M(\lambda),L(\nu)]\neq 0\}+\rho$.
	%	Note that we get $[M(\lambda),L(\nu)]\neq 0\iff M(\nu)\hookrightarrow M(\lambda)$ by the remark following Theorem \ref{BGGVerma}.

	The last thing to show is that 
	\[\{\nu\in\Lambda_I^+-\rho: [M(\lambda),L(\nu)]\neq 0\}+\rho\\
	\subseteq W_{[\lambda]}(\lambda+\rho)
	\cap \mathcal{L}_{\lambda+\rho}\cap C_\mathfrak{l}.
	\]
	Suppose $\eta+\rho$ belongs to left hand set, in which case $[M(\lambda),L(\eta)]\neq 0$. By Theorem \ref{BGGVerma}, we get $\eta\uparrow \lambda$, and in particular it holds that $\eta\le \lambda$ and $\eta=w\cdot\lambda$ for some $w\in W_{[\lambda]}$. Also $\eta+\rho\in \Lambda_I^+\subseteq C_\mathfrak{l}$. Therefore $\eta+\rho\in W_{[\lambda]}(\lambda+\rho)\cap\mathcal{L}_{\lambda+\rho}\cap C_\mathfrak{l}$, which proves the last inclusion.       
\end{proof}

%The following lemma should be well-known to the experts, but it is included for the sake of completeness.

We will need the following well-known lemma in the proof of the next theorem.
\begin{lemma}\label{Schur} The following properties hold.
	\begin{enumerate}
		\item (Schur's Lemma.)
		If $V, W$ are simple $\mathfrak{l}$-modules, then as a vector space, 
		\[
		\Hom_\mathfrak{l}(V,W)\cong \begin{cases}
		\mathbb{C},& \text{if $V\cong W$ as an $\mathfrak{l}$-module}\\
		0,&  \text{if $V\not\cong W$ as an $\mathfrak{l}$-module} \\
		\end{cases}.
		\]
		\item If $M,N,P$ are $\mathfrak{l}$-modules and $M$ is isomorphic to a quotient of $N$ as an $\mathfrak{l}$-module, then there is an injection from $\Hom_\mathfrak{l}(M,P)$ to $\Hom_\mathfrak{l}(N,P)$.
		\item If $M_1,M_2,\cdots,M_k$ and $M$ are $\mathfrak{l}$-modules, then \[
		\Hom_\mathfrak{l}\left(\bigoplus_{i=1}^k M_i,M\right)\cong  \bigoplus_{i=1}^k\Hom_\mathfrak{l}\left(M_i,M\right)
		\] as a vector space.
		\item If $\bigoplus_{i=1}^k F(\eta_i)\cong \bigoplus_{i=1}^l F(\nu_i)$ as an $\mathfrak{l}$-module, then $\{\eta_i\in \Lambda_I^+ : 1\le i\le k\}=\{\nu_i\in\Lambda_I^+ : 1\le i\le l\}$.
	\end{enumerate}
	
\end{lemma}

\begin{proof} For the proof of (1), see \cite[Proposition 5.1 and Corollary 5.2]{AK}.
	
	%Property (2) follows from the fact that the functor $\mathrm{Hom}_\mathfrak{l}$ is contravariant in the first variable.
	Now we prove (2). Suppose $M\cong M'=N/Q$ as an $\mathfrak{l}$-module for some $\mathfrak{l}$-module $M'$ and $\mathfrak{l}$-submodule $Q$ of $N$. Then there is an $\mathfrak{l}$-module isomorphism $g:M\to M'$. 	
	This induces an isomorphism of vector spaces:	
	\begin{align*}
	\psi:\Hom_\mathfrak{l}(M,P)&\rightarrow \Hom_\mathfrak{l}(M',P)\\
	f&\mapsto f\circ g^{-1}.
	\end{align*}
	There is a quotient map $\pi:N\to N/Q=M'$, which is an $\mathfrak{l}$-module homomorphism.
	This induces an injection 
	\begin{align*}
	\varphi:\Hom_\mathfrak{l}(M',P)&\rightarrow \Hom_\mathfrak{l}(N,P)\\
	f&\mapsto f\circ \pi.
	\end{align*}
	The composition $\varphi\circ\psi$ is an injection from $\Hom_\mathfrak{l}(M,P)$ to $\Hom_\mathfrak{l}(N,P)$.
	
	Next, we prove (3). We have the following two linear maps:
	\begin{align*}
	\Hom_\mathfrak{l}\left(\bigoplus_{i=1}^k M_i,M\right)
	&\rightarrow  \bigoplus_{i=1}^k\Hom_\mathfrak{l}\left(M_i,M\right)\\
	f&\mapsto (f_1,f_2,\cdots,f_k),
	\end{align*}
	where $f_i(x):=f(0,\cdots,0,x,0,\cdots,0)$ with $x$ is in the $i$ th entry and
	\begin{align*}
	\bigoplus_{i=1}^k\Hom_\mathfrak{l}\left(M_i,M\right)
	&\rightarrow  \Hom_\mathfrak{l}\left(\bigoplus_{i=1}^k M_i,M\right)\\
	(g_1,g_2,\cdots,g_k)&\mapsto g,
	\end{align*}
	where $g(x_1,\cdots,x_k):=\sum_{i=1}^k g_{i}(x_i)$.
	Two maps are inverse to each other.
	This gives a vector space isomorphism between $\Hom_\mathfrak{l}\left(\bigoplus_{i=1}^k M_i,M\right)$ and $\bigoplus_{i=1}^k\Hom_\mathfrak{l}\left(M_i,M\right)$.
	
	Finally, we prove (4). 
	Suppose $\bigoplus_{i=1}^k F(\eta_i)\cong \bigoplus_{i=1}^l F(\nu_i)$ as an $\mathfrak{l}$-module. Then for each $j\in\{1,\cdots,k\}$, we have 
	\[
	\Hom_\mathfrak{l}\left(\bigoplus_{i=1}^k F(\eta_i),F(\eta_j)\right)
	\cong \Hom_\mathfrak{l}\left(\bigoplus_{i=1}^l F(\nu_i),F(\eta_j)\right)	
	\]
	as a vector space. This implies that
	\[
	\bigoplus_{i=1}^k \Hom_\mathfrak{l}\left(F(\eta_i),F(\eta_j)\right)
	\cong \bigoplus_{i=1}^l\Hom_\mathfrak{l}\left( F(\nu_i),F(\eta_j)\right)
	\]
	as a vector space by part (3) of Lemma \ref{Schur}. Since $F(\eta_i)\not\cong F(\eta_j)$ as an $\mathfrak{l}$-module for $i\neq j$, we have $\mathbb{C}\cong \bigoplus_{i=1}^l\Hom_\mathfrak{l}\left( F(\nu_i),F(\eta_j)\right)$ as a vector space by Schur's Lemma.
	This implies that $\Hom_\mathfrak{l}\left(F(\nu_p),F(\eta_j)\right)\cong \mathbb{C}$ as a vector space for some $1\le p\le l$.
	By Schur's Lemma again, we get $F(\eta_j)\cong F(\nu_p)$ as an $\mathfrak{l}$-module for some $1\le p\le l$. This implies that $\eta_j=\nu_p$ for some $1\le p\le l$. Since $F(\eta_i)$ is the finite dimensional simple $\mathfrak{l}$-module with highest weight $\eta_i$, we have $\eta_i\in \Lambda_I^+$ by a result in \cite[\S 9.2]{HJ}. Similarly, $\nu_i\in \Lambda_I^+$.
	This implies that $\eta_j=\nu_p\in \{\nu_i\in\Lambda_I^+ : 1\le i\le l\}$. Hence $\{\eta_i\in \Lambda_I^+ : 1\le i\le k\}\subseteq \{\nu_i\in\Lambda_I^+ : 1\le i\le l\}$. Similarly, we have $\{\nu_i\in\Lambda_I^+ : 1\le i\le l\}\subseteq \{\eta_i\in \Lambda_I^+ : 1\le i\le k\}$. Therefore, the claim follows.
\end{proof}

We are now able to prove Theorem
\ref{general1to1'}.

\begin{theorem} [Theorem \ref{general1to1'}]
	\label{general1to1}
	Let $\lambda,\eta\in \Lambda_I^+$. The following statements are then equivalent:
	\begin{enumerate}
		\item $\lambda=\eta$.
		%\item $H_D(L(\lambda))=H_D(L(\eta))$.
		\item $H_D(L(\lambda))\cong H_D(L(\eta))$ as an $\mathfrak{l}$-module.
		\item $\cW_I(\lambda)=\cW_I(\eta)$. 
	\end{enumerate}
\end{theorem}

\begin{proof}
	The implication (1)$\implies$(2) is trivial. 
	For (2)$\implies$(3), suppose $H_D(L(\lambda))\cong H_D(L(\eta))$ as an $\mathfrak{l}$-module. By Theorem \ref{maindecom} and Lemma \ref{Schur}, the set of highest weights appearing in the $\mathfrak{l}$-module decomposition of $H_D(L(\lambda))$ is equal to that of $H_D(L(\eta))$. Shifting by $\rho_\mathfrak{l}$ on both sides, we get $\cW_I(\lambda)=\cW_I(\eta)$.
	
	To show (3)$\implies$(1),
	suppose $\cW_I(\lambda)=\cW_I(\eta)$. Let $\ow \in  {}^IW_{[\lambda]}^{\Sigma_\mu}$ be the unique element such that $\lambda=w_I\ow\cdot \mu$,
	as in the remark following Lemma \ref{simplemod}.
	By Lemma \ref{multilemma}, we have ${}^IP^{\Sigma_\mu}_{\ow,\ow}(q)=1$ and hence $\lambda+\rho\in \cW_I(\lambda)$. By Theorem \ref{generalcase}, we get $\lambda+\rho\in \cW_I(\eta)\subseteq \mathcal{L}_{\eta+\rho}$. This implies $\lambda\le \eta$. Similarly, $\eta\le \lambda$. Hence we get $\lambda=\eta$.
\end{proof}

%\textit{Example.} Consider singular block, and see the $Z(\mathfrak{l})$-infl chara is a proper subset. 

We are now able to prove Theorem \ref{determine}.

\begin{theorem} [Theorem \ref{determine}] 
	\label{determine'}
	Suppose $V$ and $W$ are simple modules in the category $\cO^\p$. Then $V\cong W$ as an $\mathfrak{g}$-module
	if and only if $H_D(V)\cong H_D(W)$ as an $\mathfrak{l}$-module.
\end{theorem}

\begin{proof}%[Proof of Theorem \ref{determine}]
	Suppose $V\cong W$ as an $\mathfrak{g}$-module. Then there exists an $\mathfrak{g}$-module isomorphism $f:V\rightarrow W$. 	
	Let $f\otimes i_d$ be the tensor product of $\mathbb{C}$-linear maps $f$ and $i_d$, where $i_d:\mathcal{S}\rightarrow \mathcal{S}$ is the identity map on $\mathcal{S}$, i.e., $\left(f\otimes i_d\right)(\tilde{v}\otimes s)=f(\tilde{v})\otimes s$ and $f\otimes i_d$ is a $\mathbb{C}$-linear map.
	Let $D_V$ and $D_W$ be the actions of $D$ on $V\otimes\mathcal{S}$ and $W\otimes \mathcal{S}$, respectively. We prove $H_D(V)\cong H_D(W)$ as an $\mathfrak{l}$-module in three steps.
	\begin{itemize}
		\item First we show $f\otimes i_d$ is an $\mathfrak{l}$-module isomorphism.
		
		By \cite[Remark 3.6]{HX}, $V\otimes \mathcal{S}$ and $W\otimes \mathcal{S}$ are $\mathfrak{l}$-modules. Since $f$ is an $\mathfrak{g}$-module isomorphism, then 
		for all $r\in \mathfrak{l}$, $\tilde{v}\in V$ and $s\in \mathcal{S}$, we have 
		\begin{align*}
		r\cdot\left((f\otimes i_d)(\tilde{v}\otimes s)\right)
		&=r\cdot\left(f(\tilde{v})\otimes s\right)\\
		&=\left(r\cdot f(\tilde{v})\right)\otimes s+f(\tilde{v})\otimes \left(r\cdot s\right)\\
		&=f(r\cdot\tilde{v})\otimes s+f(\tilde{v})\otimes \left(r\cdot s\right)\\
		&=\left(f\otimes i_d\right)\left(\left(r\cdot\tilde{v}\right)\otimes s\right)+\left(f\otimes i_d\right)\left(\tilde{v}\otimes \left(r\cdot s\right)\right)\\
		&=\left(f\otimes i_d\right)\left(\left(r\cdot\tilde{v}\right)\otimes s+\tilde{v}\otimes \left(r\cdot s\right)\right)\\
		&=\left(f\otimes i_d\right)\left(r\cdot\left(\tilde{v}\otimes s\right)\right).
		\end{align*}
		Since the action of $\mathfrak{l}$ and $f\otimes i_d$ are $\mathbb{C}$-linear, $f\otimes i_d$ is an $\mathfrak{l}$-module homomorphism. Clearly, $f\otimes i_d$ is bijective. The claim follows.
		\item Next, we show $(f\otimes i_d)\circ D_V=D_W\circ(f\otimes i_d)$.
		
		Since $f$ is an $\mathfrak{g}$-module isomorphism, then for all $\tilde{v}\in V$ and $s\in \mathcal{S}$, we have
		\allowdisplaybreaks
		\begin{align*}
		(f\otimes i_d)\circ D_V(\tilde{v}\otimes s)
		&=(f\otimes i_d)\left(\left(\sum_{1\le i\le m} Z_i \otimes Z_i + 1 \otimes v \right)\cdot (\tilde{v}\otimes s)\right)\\
		&=(f\otimes i_d)\left(\sum_{1\le i\le m} \left(Z_i \cdot\tilde{v}\right) \otimes \left(Z_i \cdot s\right)+ \tilde{v} \otimes \left(v \cdot s\right) \right)\\
		&=\sum_{1\le i\le m} (f\otimes i_d) \left(\left(Z_i \cdot\tilde{v}\right) \otimes \left(Z_i \cdot s\right)\right)+ (f\otimes i_d) \left(\tilde{v} \otimes \left(v \cdot s\right) \right)\\
		&=\sum_{1\le i\le m} f\left(Z_i \cdot\tilde{v}\right) \otimes \left(Z_i \cdot s\right)+  f\left(\tilde{v} \right) \otimes \left(v \cdot s\right)\\
		&=\sum_{1\le i\le m} \left(Z_i \cdot f\left(\tilde{v}\right)\right) \otimes \left(Z_i \cdot s\right)+  f\left(\tilde{v} \right) \otimes \left(v \cdot s\right)\\
		&=\left(\sum_{1\le i\le m} Z_i \otimes Z_i + 1 \otimes v\right)\cdot\left(f(\tilde{v})\otimes s\right)\\
		&=D_W\circ (f\otimes i_d)(\tilde{v}\otimes s).
		\end{align*}
		Since $D_V$, $D_W$ and $f\otimes i_d$ are $\mathbb{C}$-linear, the claim follows.
		\item Finally, we show $H_D(V)\cong H_D(W)$ as an $\mathfrak{l}$-module.
		
		It is easy to check that $\ker \left(D_V\right)=\left(f\otimes i_d\right)^{-1}\left(\ker \left(D_W\right)\right)$ and $\mathrm{Im} \left(D_V\right)=\left(f\otimes i_d\right)^{-1}\left(\mathrm{Im} \left(D_W\right)\right)$.
		Since $f\otimes i_d:V\otimes \mathcal{S}\rightarrow W\otimes \mathcal{S}$ is an $\mathfrak{l}$-module isomorphism, its restriction \[
		\left(f\otimes i_d\right)\big|_{\left(f\otimes i_d\right)^{-1}\left(\ker \left(D_W\right)\right)}: \left(f\otimes i_d\right)^{-1}\left(\ker \left(D_W\right)\right)
		\rightarrow \ker \left(D_W\right)
		\]
		is also an $\mathfrak{l}$-module isomorphism. Let $\pi_W:\ker \left(D_W\right)\rightarrow \dfrac{\ker \left(D_W\right)}{\left(\ker \left(D_W\right)\cap\mathrm{Im} \left(D_W\right)\right)}$ be the quotient map. Note that $\pi_W$ is a surjective $\mathfrak{l}$-module homomorphism. Then \[
		\pi_W\circ \left(f\otimes i_d\right)\big|_{\left(f\otimes i_d\right)^{-1}\left(\ker \left(D_W\right)\right)}: \left(f\otimes i_d\right)^{-1}\left(\ker \left(D_W\right)\right)\rightarrow \dfrac{\ker \left(D_W\right)}{\left(\ker \left(D_W\right)\cap\mathrm{Im} \left(D_W\right)\right)}
		\]
		is a surjective $\mathfrak{l}$-module homomorphism with kernel $\left(f\otimes i_d\right)^{-1}\left(\ker \left(D_W\right)\cap\mathrm{Im} \left(D_W\right)\right)$. Then by the First Isomorphism Theorem, we have
		\begin{align*}
		H_D(V)
		&=\dfrac{\ker \left(D_V\right)}{\ker \left(D_V\right)\cap \mathrm{Im} \left(D_V\right)}\\
		&=\dfrac{\left(f\otimes i_d\right)^{-1}\left(\ker \left(D_W\right)\right)}{\left(f\otimes i_d\right)^{-1}\left(\ker \left(D_W\right)\right)\cap \left(f\otimes i_d\right)^{-1}\left(\mathrm{Im} \left(D_W\right)\right)}\\
		&=\dfrac{\left(f\otimes i_d\right)^{-1}\left(\ker \left(D_W\right)\right)}{\left(f\otimes i_d\right)^{-1}\left(\ker \left(D_W\right)\cap\mathrm{Im} \left(D_W\right)\right)}\\
		&\cong
		\dfrac{\ker \left(D_W\right)}{\ker \left(D_W\right)\cap\mathrm{Im} \left(D_W\right)}\\
		&=H_D(W)
		\end{align*}
		as an $\mathfrak{l}$-module.
	\end{itemize}
	Conversely, suppose $H_D(V)\cong H_D(W)$ as an $\mathfrak{l}$-module. Since $V\cong L(\lambda)$ and $W\cong L(\lambda')$ as $\mathfrak{g}$-modules for some $\lambda,\lambda'\in\Lambda_I^+$, we get $H_D(L(\lambda))\cong H_D(V)\cong H_D(W)\cong H_D(L(\lambda'))$ as an $\mathfrak{l}$-module. 
	Theorem \ref{general1to1} then implies that $\lambda=\lambda'$, so $V\cong L(\lambda)=L(\lambda')\cong W$ as an $\mathfrak{g}$-module.
\end{proof}

\subsection{The case for regular infinitesimal character}

Under the assumption that $\lambda\in \Lambda_I^+$ is regular, we can view $\cW_I(\lambda)$ in four different ways. As a result, we get four parameterizations of $H_D(L(\lambda))$ when $\lambda\in \Lambda_I^+$ is regular.

%In this section, we identify $H_D(L(\lambda))$ with $W_{[\lambda]}(\lambda+\rho)\cap\mathcal{L}_{\lambda+\rho}\cap C_\mathfrak{l}$, where $\mathcal{L}_{\eta}=\{\nu\in\h^*: \nu\le \eta\}$, for regular $\lambda\in\Lambda_I^+$.

%Before giving the four parameterizations of $H_D(L(\lambda))$, we need the following two lemmas:

Now we are able to describe two geometric parameterizations of $H_D(L(\lambda))$.

\begin{theorem} [Theorem \ref{geompara2'}] 
	\label{geompara2}
	Let $\mathcal{R}$ be the set of regular weights in $\h^*$. Let $\lambda\in\Lambda_I^+\cap \mathcal{R}$.
	Then \begin{align*}
	\cW_I(\lambda)
	=W_{[\lambda]}(\lambda+\rho)\cap\mathcal{L}_{\lambda+\rho}\cap C_\mathfrak{l}
	=\left(\mathcal{S}_{[\lambda]}(\lambda)+\rho\right)\cap C_\mathfrak{l}.
	\end{align*}
\end{theorem}

\begin{proof}
	Let $\ow \in  {}^IW_{[\lambda]}^{\Sigma_\mu}$ be the unique element such that 
	$\lambda=w_I\ow\cdot \mu$, which exists by the remark following Lemma \ref{simplemod}.
	Since $\lambda$ is regular, it holds that ${}^IP_{x,\ow}^{\Sigma_\mu} (q)={}^IP_{x,\ow}^{\mu}(q)$ and ${}^IW_{[\lambda]}^{\Sigma_\mu}={}^IW_{[\lambda]}$. 
	%By Theorem \ref{maindecom}, we have 
	%$H_D(L(\lambda))\cong\bigoplus_{x\in {}^IW_{[\lambda]}}{}^I P_{x,\ow}(1)F(w_Ix\cdot\mu+\rho(\mathfrak{u}))$ as $\mathfrak{l}$-modules.
	%Then
	%\begin{align*}
	%\cW_I(\lambda)
	%&=w_I\cdot\left(\{x\in {}^IW_{[\lambda]}: w_Ix\le_{[\lambda]} w_Iw\}\cdot\mu\right)+\rho.    
	%\end{align*}
	We prove the first equality in the theorem statement in three steps.
	%Then we give the proof for the first geometric parameterization of $H_D(L(\lambda))$ in three steps.
	
	\begin{itemize}
		\item First we show $w_I\cdot\left(\{x\in {}^IW_{[\lambda]}: w_Ix\le_{[\lambda]} w_I\ow\}\cdot\mu\right)=w_I\cdot \left({}^IW_{[\lambda]}\cdot\mu\right)\cap \mathcal{L}_\lambda$.
		
		Clearly, $w_I\cdot\left(\{x\in {}^IW_{[\lambda]}: w_Ix\le_{[\lambda]} w_I\ow\}\cdot\mu\right)\subseteq w_I\cdot \left({}^IW_{[\lambda]}\cdot\mu\right)$.
		
		Suppose $\eta\in w_I\cdot\left(\{x\in {}^IW_{[\lambda]}: w_Ix\le_{[\lambda]} w_I\ow\}\cdot\mu\right)$. Then $\eta=w_Ix\cdot \mu$ with $w_Ix\le_{[\lambda]} w_I\ow$.
		
		By Lemma \ref{HJtypo}, we get $w_Ix\le_{[\lambda]} w_I\ow\iff \eta=w_Ix\cdot\mu\le w_I\ow\cdot\mu=\lambda$.
		Then $\eta\in \mathcal{L}_\lambda$ and hence 
		$w_I\cdot\left(\{x\in {}^IW_{[\lambda]}: w_Ix\le_{[\lambda]} w_I\ow\}\cdot\mu\right)\subseteq w_I\cdot \left({}^IW_{[\lambda]}\cdot\mu\right)\cap \mathcal{L}_\lambda$.
		
		Conversely, suppose $\eta\in w_I\cdot \left({}^IW_{[\lambda]}\cdot\mu\right)\cap \mathcal{L}_\lambda$. Then $\eta=w_Ix\cdot\mu$ with $x\in {}^IW_{[\lambda]}$ and $\eta\le \lambda$.
		Then we get $w_Ix\cdot\mu\le w_I\ow\cdot\mu$, which is equivalent to $w_Ix\le_{[\lambda]} w_I\ow$ by Lemma \ref{HJtypo}.
		Then $\eta\in w_I\cdot\left(\{x\in {}^IW_{[\lambda]}: w_Ix\le_{[\lambda]} w_I\ow\}\cdot\mu\right)$ and hence 
		$w_I\cdot \left({}^IW_{[\lambda]}\cdot\mu\right)\cap \mathcal{L}_\lambda \subseteq w_I\cdot\left(\{x\in {}^IW_{[\lambda]}: w_Ix\le_{[\lambda]} w_I\ow\}\cdot\mu\right)$. Therefore, $w_I\cdot\left(\{x\in {}^IW_{[\lambda]}: w_Ix\le_{[\lambda]} w_I\ow\}\cdot\mu\right)=w_I\cdot \left({}^IW_{[\lambda]}\cdot\mu\right)\cap \mathcal{L}_\lambda$.
		
		\item Next, we observe that, by Lemma \ref{IW},  
		\begin{align*}
		{}^IW_{[\lambda]}\cdot\mu
		&=\{w\cdot\mu\in W_{[\lambda]}\cdot\mu: w\in {}^IW_{[\lambda]}\}\\
		&=\{w\cdot\mu\in W_{[\lambda]}\cdot\mu: w\cdot \mu\in-C_\mathfrak{l}-\rho\}\\
		&=W_{[\lambda]}\cdot \mu\cap(-C_\mathfrak{l}-\rho).    
		\end{align*}
		
		\item Finally, we show $\cW_I(\lambda)=W_{[\lambda]}(\lambda+\rho)\cap\mathcal{L}_{\lambda+\rho}\cap C_\mathfrak{l}$.
		
		Let $X,Y$ be sets.
		Note that $W_I\subseteq W_{[\lambda]}$, $\mu\in W_{[\lambda]}\cdot \lambda$ and $-w_IC_\mathfrak{l}=C_\mathfrak{l}$. 
		Therefore,
		\allowdisplaybreaks
		\begin{align*}
			\cW_I(\lambda)
			&= \{w_Ix\cdot\mu+\rho(\mathfrak{u}): x\in {}^IW_{[\lambda]}, {}^IP_{x,\ow}^{\mu}(1)\neq 0\}+\rho_\mathfrak{l}&\text{(by Definition \ref{WIlambda})}\\
			&=\{w_Ix\cdot\mu: x\in {}^IW_{[\lambda]}, w_Ix\le_{[\lambda]} w_I\ow\}+\rho &\text{(by Corollary \ref{keyequiv})}\\
			&=w_I\cdot\left(\{x\in {}^IW_{[\lambda]}: w_Ix\le_{[\lambda]} w_I\ow\}\cdot\mu\right)+\rho&\\
			&=w_I\cdot \left({}^IW_{[\lambda]}\cdot\mu\right)\cap \mathcal{L}_\lambda+\rho& \text{(by step 1)}\\
			&=w_I\cdot\left(W_{[\lambda]}\cdot \mu\cap(-C_\mathfrak{l}-\rho)\right)\cap \mathcal{L}_\lambda+\rho& \text{(by step 2)}\\
			&=w_I\cdot \left(W_{[\lambda]}\cdot \mu\right)\cap w_I\cdot(-C_\mathfrak{l}-\rho)
			\cap \mathcal{L}_\lambda+\rho&\text{($w_I\cdot (X\cap Y)=w_I\cdot X\cap w_I\cdot Y$)}\\
			&=\left(w_IW_{[\lambda]}\right)\cdot \mu\cap w_I\cdot(-C_\mathfrak{l}-\rho)
			\cap \mathcal{L}_\lambda+\rho&\text{($w_I\cdot \left(W_{[\lambda]}\cdot \mu\right)=\left(w_IW_{[\lambda]}\right)\cdot \mu$)}\\
			&=W_{[\lambda]}\cdot \mu\cap w_I\cdot(-C_\mathfrak{l}-\rho)
			\cap \mathcal{L}_\lambda+\rho&
			\text{($w_IW_{[\lambda]}=W_{[\lambda]}$)}\\
			&=W_{[\lambda]}\cdot \lambda\cap w_I\cdot(-C_\mathfrak{l}-\rho)
			\cap \mathcal{L}_\lambda+\rho&
			\text{($W_{[\lambda]}\cdot\mu=W_{[\lambda]}\cdot\lambda$)}\\
			&=W_{[\lambda]}\cdot \lambda\cap\left(-w_IC_\mathfrak{l}-\rho\right)
			\cap \mathcal{L}_\lambda+\rho&\text{($w_I\cdot(-C_\mathfrak{l}-\rho)=-w_IC_\mathfrak{l}-\rho$)}\\
			&=W_{[\lambda]}\cdot\lambda\cap\left( 
			C_\mathfrak{l}-\rho\right)
			\cap \mathcal{L}_\lambda+\rho&\text{($-w_IC_\mathfrak{l}=C_\mathfrak{l}$)}\\
			&=\left(W_{[\lambda]}\cdot\lambda+\rho\right)\cap\left( 
			C_\mathfrak{l}-\rho+\rho\right)
			\cap (\mathcal{L}_\lambda+\rho)&\text{($X\cap Y+\rho=(X+\rho)\cap (Y+\rho)$)}\\
			&=W_{[\lambda]}(\lambda+\rho)\cap\left( 
			C_\mathfrak{l}-\rho+\rho\right)
			\cap (\mathcal{L}_\lambda+\rho)&\text{($W_{[\lambda]}\cdot\lambda+\rho=W_{[\lambda]}(\lambda+\rho)$)}\\
			&=W_{[\lambda]}(\lambda+\rho)
			\cap (\mathcal{L}_\lambda+\rho)\cap C_\mathfrak{l}&\text{($C_\mathfrak{l}-\rho+\rho=C_\mathfrak{l}$)}\\
			&=W_{[\lambda]}(\lambda+\rho)
			\cap \mathcal{L}_{\lambda+\rho}\cap C_\mathfrak{l}.&\text{($\mathcal{L}_{\lambda}+\rho=\mathcal{L}_{\lambda+\rho}$)}
		\end{align*}
	\end{itemize}
	This proves that $\cW_I(\lambda)
	=W_{[\lambda]}(\lambda+\rho)\cap\mathcal{L}_{\lambda+\rho}\cap C_\mathfrak{l}$. The second equality holds by
	Theorem \ref{generalcase}.
	%      it follows that  $\cW_I(\lambda)=W_{[\lambda]}(\lambda+\rho)\cap\left(\mathcal{S}_{[\lambda]}(\lambda)+\rho\right)\cap C_\mathfrak{l}$.
	% 
	%By the similar argument above, one can show $w_I\cdot\left(\{x\in {}^IW_{[\lambda]}: w_Ix\le_{[\lambda]} w_I\ow\}\cdot\mu\right)=w_I\cdot \left({}^IW_{[\lambda]}\cdot\mu\right)\cap \mathcal{S}_{[\lambda]}(\lambda)$ by Lemma \ref{HJtypo} and hence $\cW_I(\lambda)=W_{[\lambda]}(\lambda+\rho)\cap\left(\mathcal{S}_{[\lambda]}(\lambda)+\rho\right)\cap C_\mathfrak{l}$. 
\end{proof}

\begin{remark}
	Theorems \ref{generalcase} and \ref{geompara2} together imply that for $\lambda\in\Lambda_I^+\cap\mathcal{R}$, it holds that
	\[
	\cW_I(\lambda)
	=\{\nu\in\Lambda_I^+-\rho: [M(\lambda),L(\nu)]\neq 0\}+\rho
	=\{\nu\in\Lambda_I^+-\rho: M(\nu)\hookrightarrow M(\lambda)\}+\rho.
	\]
\end{remark}

We can use Theorem \ref{geompara2} to derive two algebraic parameterizations of $H_D(L(\lambda))$ in terms of the multiplicities of the composition factors of a Verma module and the embeddings between Verma modules, respectively. Before showing that, we need the following proposition. 
\begin{proposition}\label{req'd subset}
	Let $\lambda\in\Lambda_I^+\cap\mathcal{R}$. Then
	\[
	\cW_I(\lambda)
	=W_{[\lambda]}\cdot\lambda\cap\mathcal{L}_\lambda\cap \Lambda_I^++\rho
	=\mathcal{S}_{[\lambda]}(\lambda)\cap \Lambda_I^++\rho.
	\]
	%, where $\mathcal{L}_{\eta}=\{\nu\in\h^*: \nu\le \eta\}$ and $\mathcal{S}_{[\lambda]}(\lambda)=\{\nu\in\h^*: \nu\uparrow_{[\lambda]} \lambda\}$.
\end{proposition}

\begin{proof} 
	Let $C_\mathfrak{l}^\circ:=\{\nu\in\h^*:\langle\nu,\alpha^\lor\rangle> 0, \ \forall \alpha\in I\}$. 
	Since $\lambda$ is regular, we get $\eta$ is regular for all $\eta\in W_{[\lambda]}\cdot \lambda$. This implies that
	\begin{align*}
	\cW_I(\lambda)
	=W_{[\lambda]}(\lambda+\rho)\cap\mathcal{L}_{\lambda+\rho}\cap C_\mathfrak{l}
	=W_{[\lambda]}(\lambda+\rho)\cap\mathcal{L}_{\lambda+\rho}\cap C_\mathfrak{l}^\circ.
	\end{align*}
	Since $W_{[\lambda]}\cdot\lambda=W_{[\lambda]}(\lambda+\rho)-\rho$ and $\mathcal{L}_\lambda=\mathcal{L}_{\lambda+\rho}-\rho$, we have
	\begin{align*}
	\cW_I(\lambda)-\rho
	&=W_{[\lambda]}(\lambda+\rho)\cap\mathcal{L}_{\lambda+\rho}\cap C_\mathfrak{l}^\circ-\rho\\
	&=\left(W_{[\lambda]}(\lambda+\rho)-\rho\right)\cap\left(\mathcal{L}_{\lambda+\rho}-\rho\right)\cap \left(C_\mathfrak{l}^\circ-\rho\right)\\
	&=W_{[\lambda]}\cdot\lambda\cap\mathcal{L}_\lambda\cap \left(C_\mathfrak{l}^\circ-\rho\right).
	\end{align*}

	Suppose $\eta\in W_{[\lambda]}\cdot\lambda\cap\mathcal{L}_\lambda\cap \left(C_\mathfrak{l}^\circ-\rho\right)$. It holds that $\eta=w\cdot\lambda$ for some $w\in W_{[\lambda]}$. Then by the definition of $W_{[\lambda]}$, we have $\eta-\lambda\in \Lambda_r$. Since $\lambda\in\Lambda_I^+$, we get $\langle\eta+\rho,\alpha^\lor\rangle=\langle\eta-\lambda,\alpha^\lor\rangle+\langle\lambda,\alpha^\lor\rangle+\langle\rho,\alpha^\lor\rangle\in \mathbb{Z}$ for all $\alpha\in I$.
	Since $\eta\in C_\mathfrak{l}^\circ-\rho$, we get $\eta+\rho\in C_\mathfrak{l}^\circ$ and then $\langle\eta+\rho,\alpha^\lor\rangle>0$ for all $\alpha\in I$. This implies that $\langle\eta+\rho,\alpha^\lor\rangle\in \mathbb{Z}^{>0}$ for all $\alpha\in I$, or equivalently that
	$\langle\eta,\alpha^\lor\rangle\in \mathbb{Z}^{\ge 0}$ for all $\alpha\in I$. This implies $\eta\in \Lambda_I^+$.
	Hence $W_{[\lambda]}\cdot\lambda\cap\mathcal{L}_\lambda\cap \left(C_\mathfrak{l}^\circ-\rho\right)\subseteq W_{[\lambda]}\cdot\lambda\cap\mathcal{L}_\lambda\cap \Lambda_I^+$.

	Conversely, suppose $\eta\in W_{[\lambda]}\cdot\lambda\cap\mathcal{L}_\lambda\cap \Lambda_I^+$. Then $\langle\eta,\alpha^\lor \rangle\in \mathbb{Z}^{\ge 0}$ for all $\alpha\in I$ and then $\langle\eta+\rho,\alpha^\lor \rangle\in \mathbb{Z}^{>0}$ for all $\alpha\in I$. This implies $\eta\in C_{\mathfrak{l}}^\circ-\rho$.
	Hence $W_{[\lambda]}\cdot\lambda\cap\mathcal{L}_\lambda\cap \Lambda_I^+\subseteq W_{[\lambda]}\cdot\lambda\cap\mathcal{L}_\lambda\cap \left(C_\mathfrak{l}^\circ-\rho\right)$.
	Therefore, $\cW_I(\lambda)-\rho
	=W_{[\lambda]}\cdot\lambda\cap\mathcal{L}_\lambda\cap \left(C_\mathfrak{l}^\circ-\rho\right)
	=W_{[\lambda]}\cdot\lambda\cap\mathcal{L}_\lambda\cap \Lambda_I^+$.
	By a similar argument, we get $\cW_I(\lambda)-\rho=\mathcal{S}_{[\lambda]}(\lambda)\cap \Lambda_I^+$.
\end{proof}

Now we are able to describe two algebraic parameterizations of $H_D(L(\lambda))$.

\begin{theorem} [Theorem \ref{parameterization}]
	\label{parameterization'}
	Let $\lambda\in\Lambda_I^+\cap\mathcal{R}$. Then
	\begin{align*}
	\cW_I(\lambda)
	=\{\nu\in\Lambda_I^+: [M(\lambda),L(\nu)]\neq 0\}+\rho
	=\{\nu\in\Lambda_I^+: M(\nu)\hookrightarrow M(\lambda)\}+\rho.
	\end{align*}
\end{theorem}

\begin{proof}
	By Proposition \ref{req'd subset}, if $\lambda\in\Lambda_I^+\cap\mathcal{R}$
	then
	$\cW_I(\lambda)-\rho
	=W_{[\lambda]}\cdot\lambda\cap\mathcal{L}_\lambda\cap \Lambda_I^+
	=\mathcal{S}_{[\lambda]}(\lambda)\cap \Lambda_I^+$.
	Suppose $\eta\in \cW_I(\lambda)-\rho$. Then $\eta\in \mathcal{S}_{[\lambda]}(\lambda)$ or, equivalently, $\eta\uparrow_{[\lambda]} \lambda$. This implies $\eta\uparrow \lambda$. By Theorem \ref{BGGVerma}, we get $[M(\lambda),L(\eta)]\neq 0$. Since $\eta\in\Lambda_I^+$, we get $\eta\in \{\nu\in\Lambda_I^+: [M(\lambda),L(\nu)]\neq 0\}$.

	Conversely, suppose $\eta\in \{\nu\in\Lambda_I^+: [M(\lambda),L(\nu)]\neq 0\}$. Then $[M(\lambda),L(\eta)]\neq 0$, so by Theorem \ref{BGGVerma}, we get $\eta\uparrow \lambda$. In particular, by the definitions of strong linkage and  $W_{[\lambda]}$, we have $\eta\le\lambda$ and $\eta=w\cdot\lambda$ for some $w\in W_{[\lambda]}$. Since $\eta\in \Lambda_I^+$, we get $\eta\in W_{[\lambda]}\cdot\lambda\cap\mathcal{L}_\lambda\cap \Lambda_I^+=\cW_I(\lambda)-\rho$. 
	By the remark following Theorem \ref{BGGVerma}, we obtain the second algebraic parameterization.
\end{proof}

%Now we justify that $\cW_I(\lambda)$ is the parameterization of $H_D(L(\lambda))$ by using Theorem \ref{geompara2}. As a result, we have four parameterizations for Dirac cohomology of $L(\lambda)$ with regular infinitesimal character.
%\begin{proposition}\label{1to1}There is a one-to-one correspondence between\begin{align*}\{H_D(L(\lambda)): L(\lambda)\in \cO^\p, \lambda\in \mathcal{R} \}&\longleftrightarrow \{\cW_I(\lambda): \lambda\in \Lambda_I^+\cap \mathcal{R}\}\\    H_D(L(\lambda))&\mapsto \cW_I(\lambda)\end{align*}\end{proposition}
%\begin{proof}The forward map is defined by mapping $H_D(L(\lambda))$ to $\cW_I(\lambda)$, which is a well-defined map since $L(\lambda)\in\cO^\p$ iff $\lambda\in \Lambda_I^+$ and by the decomposition in Theorem \ref{maindecom}.
%By Theorem \ref{geompara2}, $\cW_I(\lambda)=W_{[\lambda]}(\lambda+\rho)        \cap \mathcal{L}_{\lambda+\rho}\cap C_\mathfrak{l}$ when $\lambda\in \Lambda_I^+\cap \mathcal{R}$.Since the set is finite, one can pick the highest weight in this set, which is $\lambda+\rho$ by definition of $\mathcal{L}_{\lambda+\rho}$. Then the composition $\cW_I(\lambda)\mapsto\lambda\mapsto H_D(L(\lambda))$ gives the inverse map.\end{proof}

One application of the parameterizations of $H_D(L(\lambda))$ is to obtain an extended version of the Verma-BGG Theorem for Verma modules with regular infinitesimal character. We can show that the condition in terms of strong linkage in the Verma-BGG Theorem is equivalent to some seemingly weaker or stronger conditions.

\begin{theorem}\label{5equiv}
	Let $\lambda\in \Lambda_I^+\cap\mathcal{R}$ and $\eta\in\Lambda_I^+$. The following statements are then equivalent:
	\begin{enumerate}
		\item $[M(\lambda),L(\eta)]\neq 0$.
		\item $M(\eta)\hookrightarrow M(\lambda)$.
		\item $\eta$ is strongly linked to $\lambda$.
		\item $\eta$ is $[\lambda]$-strongly linked to $\lambda$.
		\item $\eta\le \lambda$ and $\eta=w\cdot\lambda$ for some $w\in W_{[\lambda]}$.
		\item $\cW_I(\eta)
		\subseteq\cW_I(\lambda)$.
		\item $\cW(\eta)
		\subseteq\cW(\lambda)$.
	\end{enumerate}
	%Note that $\mu$ is the antidominant weight in $W_{[\lambda]}\cdot\lambda$.
\end{theorem}

\begin{proof}
	By Theorem \ref{BGGVerma} and the remark following it, we get (1)$\iff$(2)$\iff$(3).
	Suppose $\eta$ satisfies (4); then by the definitions of $[\lambda]$-strong linkage and $W_{[\lambda]}$, we deduce that $\eta$ satisfies (5),
	so we get that (4)$\implies$(5).

	By Proposition \ref{req'd subset} and Theorem \ref{parameterization'}, we have \[
	W_{[\lambda]}\cdot \lambda\cap\mathcal{L}_\lambda\cap \Lambda_I^+=\{\nu\in\Lambda_I^+: M(\nu)\hookrightarrow M(\lambda)\}=\mathcal{S}_{[\lambda]}(\lambda)\cap \Lambda_I^+.
	\]
	If $\eta$ satisfies (5), then it follows that $\eta$ satisfies (2), which implies (4)
	because of the preceding identity. This means (4)$\implies $(5)$\implies $(2)$\implies $(4). 
	
	To show that (5)$\iff$(6), suppose $\eta\le \lambda$ and $\eta=w\cdot\lambda$ for some $w\in W_{[\lambda]}$. Then $\eta-\lambda\in\Lambda_r$ and hence $W_{[\eta]}=W_{[\lambda]}$. Since $\eta=w\cdot\lambda\in W_{[\lambda]}\cdot\lambda$, we get $W_{[\eta]}(\eta+\rho)=W_{[\lambda]}(\lambda+\rho)$.
	Since $\eta \le \lambda$, we get $\mathcal{L}_{\eta+\rho}\subseteq\mathcal{L}_{\lambda+\rho}$. Therefore 
	\[\cW_I(\eta)\subseteq W_{[\eta]}(\eta+\rho)\cap\mathcal{L}_{\eta+\rho}\cap C_\mathfrak{l}\subseteq W_{[\lambda]}(\lambda+\rho)\cap\mathcal{L}_{\lambda+\rho}\cap C_\mathfrak{l}=\cW_I(\lambda)\]  by Theorems \ref{generalcase} and  \ref{geompara2}. 
	Conversely, suppose $\cW_I(\eta)
	\subseteq\cW_I(\lambda)$. Then
	$\eta+\rho\in \cW_I(\eta) \subseteq \cW_I(\lambda)=W_{[\lambda]}(\lambda+\rho)\cap\mathcal{L}_{\lambda+\rho}\cap C_\mathfrak{l}$ by the argument in the proof of Theorem \ref{general1to1}, and Theorem \ref{geompara2}. This implies $\eta+\rho\in \mathcal{L}_{\lambda+\rho}$ and $\eta+\rho\in W_{[\lambda]}(\lambda+\rho)$.
	Therefore $\eta\le \lambda$ and $\eta=w\cdot\lambda$ for some $w\in W_{[\lambda]}$.
	We conclude that (5)$\iff$(6) as needed.
	
	Finally, we show (5)$\iff$(7). By the equivalence (5)$\iff$(6) with $I=\emptyset$, for all $\lambda\in\mathcal{R}$ and $\eta\in\h^*$, it holds that $\eta\le \lambda$ and $\eta=w\cdot\lambda$ for some $w\in W_{[\lambda]}$ is equivalent to $\cW(\eta)
	\subseteq\cW(\lambda)$. Since $\Lambda_I^+\cap\mathcal{R}\subseteq \mathcal{R}$ and $\Lambda_I^+\subseteq\h^*$, we conclude that (5)$\iff$(7) as needed.   
\end{proof}

\section{Dirac cohomology of Kostant modules}
\label{S:5}

\subsection{Dirac cohomology of Kostant modules}

Following \cite{BBMH}, a finite poset is called an \emph{interval} if it has a unique minimum and a unique maximum. A finite poset is called \emph{graded} if it is an interval and if all maximal chains between any two elements
have the same length. In this case the poset has a well-defined \emph{rank
	function} whose value at a vertex $x$ is the length of any maximal chain from the unique minimum to $x$.

Continue to let $\lambda\in\Lambda_I^+$. As noted in \cite[\S 3.2]{BBMH}, it holds that the posets of the form $\left({}^IW_{[\lambda]}, \le_{[\lambda]}\right)$ are graded and that the rank function on ${}^IW_{[\lambda]}$ is the restriction of the length function $\ell_{[\lambda]}$ on $W_{[\lambda]}$ (see \cite[Corollary 3.8]{VD2}).

%However, the posets of the form ${}^IW_{[\lambda]}^{\Sigma_\mu}$ are not graded in general.

%Let $\cO^\p_\mu$ be a regular or singular subcategory of category $\cO^\p$.
\begin{definition} [{See \cite[\S 3.3]{BBMH}}]
For $w \in
{}^IW_{[\lambda]}^{\Sigma_\mu}$, we say that $L(w_Iw\cdot\mu)$ is a Kostant module in $\cO^\p_\mu$ if there exists a graded interval $[v, w]$ of $\left({}^IW_{[\lambda]}^{\Sigma_\mu}, \le_{[\lambda]}\right)$ such that as an $\mathfrak{l}$-module,
\[
H^i(\mathfrak{u}, L(w_Iw\cdot\mu)) \cong
\bigoplus_{\substack{ x\in [v,w] \\ r(x) = r(w)-i}}F(w_Ix\cdot\mu) \quad \text{for all $i\ge 0$},
\]
where $r$ is the rank function on $[v, w]$.
\end{definition}

%\begin{remark} [{\cite[\S 3.3]{BBMH}}] When $\lambda$ is regular, the relevant interval is $[v, w] = [e, w] = {}^IW_{[\lambda]}$, and the rank function is the length function.
%
%\end{remark}

\begin{lemma} [{See \cite[\S 3.4]{BBMH} and \cite[Theorem 5.13]{EHP}}] \label{Kostantmod}
	Let $\lambda\in\Lambda_I^+\cap\mathcal{R}$ and $w\in {}^IW_{[\lambda]}$. Then $L(w_Iw\cdot\mu)$ is a Kostant module if and only if ${}^IP_{x,w}^{\mu}(q)=1$ for all $x\le_{[\lambda]} w$.
\end{lemma}

\begin{remark}
In \cite{BBMH}, it is assumed that $\lambda$ is integral, but the
proof in \cite[\S 3.4]{BBMH} also works for nonintegral weights (which is our context).
\end{remark}

%\textit{Example.}
%Let $\g=\mathfrak{so}(2n,\mathbb{R})$, $\mathfrak{l}=\mathfrak{k}$, then ${}^IW$ is a simple chain and ${}^IP_{x,w}(q)=\begin{cases}
%1&\ \text{if } x\le w\\
%0&\ \text{if otherwise}\\
%\end{cases}$.

%Let $\mu$ is regular, integral, antidominant (e.g. $\mu=-2\rho$). Then every simple module in $\cO_\mu^\p$ is a Kostant module.

\begin{proposition}\label{neccessary}
Let $\lambda,\lambda' \in \Lambda_I^+$. Suppose $H_D(L(\lambda))$ is isomorphic to a quotient of $H_D(L(\lambda'))$ as an $\mathfrak{l}$-module. Then $\lambda\uparrow_{[\lambda']} \lambda'$.
\end{proposition}

\begin{proof}
As in the remark following Lemma \ref{simplemod},
let $\ow \in  {}^IW_{[\lambda]}^{\Sigma_\mu}$ and $\ow' \in  {}^IW_{[\lambda']}^{\Sigma'_{\mu'}}$ be the unique elements such that $\lambda=w_I\ow\cdot \mu$ and
$\lambda'=w_I\ow'\cdot \mu'$, respectively, where $\mu'$ is the antidominant weight in $W_{[\lambda']}\cdot \lambda'$ and ${\Sigma'_{\mu'}}$ is the set of singular simple roots associated to $\mu'$ in $\Delta_{[\lambda']}$.
By Lemma \ref{multilemma}, it holds that ${}^IP^{\Sigma_\mu}_{\ow,\ow}(q)=1$.
Suppose $H_D(L(\lambda))$ is isomorphic to a quotient of $H_D(L(\lambda'))$ as an $\mathfrak{l}$-module. By Lemma \ref{Schur}, there is an injection from $\Hom_\mathfrak{l}\left(H_D(L(\lambda)),F(\lambda+\rho(\mathfrak{u}))\right)$ to $\Hom_\mathfrak{l}\left(H_D(L(\lambda')),F(\lambda+\rho(\mathfrak{u}))\right)$.
%Then we have $\Hom_\mathfrak{l}\left(H_D(L(\lambda)),F(\lambda+\rho(\mathfrak{u}))\right)\subseteq \Hom_\mathfrak{l}\left(H_D(L(\lambda')),F(\lambda+\rho(\mathfrak{u}))\right)$ by Lemma \ref{Schur}.
%
It holds that 
\begin{align*}
&\Hom_\mathfrak{l}\left(H_D(L(\lambda)),F(\lambda+\rho(\mathfrak{u}))\right)&\\
&\cong 
\Hom_\mathfrak{l}\left(\bigoplus_{x\in {}^I W_{[\lambda]}^{{\Sigma_{\mu}}}}{}^IP_{x,\ow}^{{\Sigma_{\mu}}}(1)F(w_Ix\cdot\mu+\rho(\mathfrak{u})),F(\lambda+\rho(\mathfrak{u}))\right)&\text{(by Theorem \ref{maindecom})}\\
&\cong 
\bigoplus_{x\in {}^I W_{[\lambda]}^{{\Sigma_{\mu}}}}{}^IP_{x,\ow}^{{\Sigma_{\mu}}}(1)\Hom_\mathfrak{l}\left(F(w_Ix\cdot\mu+\rho(\mathfrak{u})),F(\lambda+\rho(\mathfrak{u}))\right)&\text{(by Lemma \ref{Schur})}\\
&\cong \Hom_\mathfrak{l}\left(F(\lambda+\rho(\mathfrak{u})),F(\lambda+\rho(\mathfrak{u}))\right)&\text{(by Schur's Lemma, ${}^IP^{\Sigma_\mu}_{\ow,\ow}(q)=1$)}\\
&\cong 
\mathbb{C}&\text{(by Schur's Lemma)}
\end{align*}
as a vector space.
And we have
\begin{align*}
&\Hom_\mathfrak{l}\left(H_D(L(\lambda')),F(\lambda+\rho(\mathfrak{u}))\right)\\
&\cong 
\Hom_\mathfrak{l}\left(\bigoplus_{x\in {}^I W_{[\lambda']}^{{\Sigma'_{\mu'}}}}{}^IP_{x,\ow'}^{{\Sigma'_{\mu'}}}(1)F(w_Ix\cdot\mu'+\rho(\mathfrak{u})),F(\lambda+\rho(\mathfrak{u}))\right)&\text{(by Theorem \ref{maindecom})}\\
&\cong 
\bigoplus_{x\in {}^I W_{[\lambda']}^{{\Sigma'_{\mu'}}}}{}^IP_{x,\ow'}^{{\Sigma'_{\mu'}}}(1)\Hom_\mathfrak{l}\left(F(w_Ix\cdot\mu'+\rho(\mathfrak{u})),F(\lambda+\rho(\mathfrak{u}))\right)&\text{(by Lemma \ref{Schur})}
\end{align*}
as a vector space.
We therefore have an injection from $\mathbb{C}$ to
\[
\bigoplus_{x\in {}^I W_{[\lambda']}^{{\Sigma'_{\mu'}}}}{}^IP_{x,\ow'}^{{\Sigma'_{\mu'}}}(1)\Hom_\mathfrak{l}\left(F(w_Ix\cdot\mu'+\rho(\mathfrak{u})),F(\lambda+\rho(\mathfrak{u}))\right).
\]
Then there is $x\in {}^I W_{[\lambda']}^{{\Sigma'_{\mu'}}}$ such that ${}^IP_{x,\ow'}^{{\Sigma'_{\mu'}}}(1)\neq 0$ and $
\Hom_\mathfrak{l}\left(F(w_Ix\cdot\mu'+\rho(\mathfrak{u})),F(\lambda+\rho(\mathfrak{u}))\right)\cong \mathbb{C}$.
By Schur's Lemma, there is $x\in {}^I W_{[\lambda']}^{{\Sigma'_{\mu'}}}$ such that ${}^IP_{x,\ow'}^{{\Sigma'_{\mu'}}}(1)\neq 0$ and $\lambda+\rho(\mathfrak{u})=w_Ix\cdot\mu'+\rho(\mathfrak{u})$, (i.e., $\lambda=w_Ix\cdot\mu'$).
Since ${}^IP_{x,\ow'}^{{\Sigma'_{\mu'}}}(1)\neq 0$, we get $w_Ix\le_{[\lambda']} w_I\ow'$ by Lemma \ref{multilemma}.
Then by Lemma \ref{HJtypo}, we have $\lambda=w_Ix\cdot \mu'\uparrow_{[\lambda']} w_I\ow'\cdot \mu'=\lambda'$.
\end{proof}

\begin{theorem}\label{kostant-prop}
Let $\lambda,\lambda' \in \Lambda_I^+$. Suppose $\lambda$ or $\lambda'$ is regular and $L(\lambda)$ is a Kostant module.
Then
the following statements are equivalent:
\begin{enumerate}
    \item $H_D(L(\lambda))$ is isomorphic to a quotient of $H_D(L(\lambda'))$ as an $\mathfrak{l}$-module.
    \item $\lambda$ is $[\lambda']$-strongly linked to $\lambda'$.
    \item $\lambda$ is strongly linked to $\lambda'$.
    \item $M(\lambda)\hookrightarrow M(\lambda')$.
    \item $[M(\lambda'), L(\lambda)]\neq 0$.
    \item $\lambda\le \lambda'$ and $\lambda=w\cdot\lambda'$ for some $w\in W_{[\lambda']}$.
\end{enumerate}
%Note that $\mu'$ is the antidominant weight in $W_{[\lambda']}\cdot\lambda'$.
\end{theorem}

\begin{proof}
We continue to follow the notations used in Proposition \ref{neccessary}. By Proposition \ref{neccessary}, we get that (1)$\implies$(2).
By the definitions of $[\lambda']$-strong linkage and strong linkage, we get that (2)$\implies$(3). By Theorem \ref{BGGVerma} and the remark below it, we get that (3)$\iff$(4)$\iff$(5). By the definitions of strong linkage and $W_{[\lambda']}$, we get (3)$\implies$(6).

To show that (6)$\implies$(1), suppose $\lambda\le\lambda'$ and $\lambda=w\cdot\lambda'$ for some $w\in W_{[\lambda']}$.
This implies $\lambda-\lambda'\in\Lambda_r$ and then $W_{[\lambda]}\cdot \lambda=W_{[\lambda']}\cdot \lambda'$. Then we get $\mu=\mu'$, where $\mu$ and $\mu'$ are the antidominant weights in $W_{[\lambda]}\cdot \lambda$ and $W_{[\lambda']}\cdot \lambda'$, respectively.
Since $\lambda$ or $\lambda'$ is regular, we get that $\mu=\mu'$ is regular. By Lemma \ref{HJtypo}, 
we have 
\[
\lambda\le \lambda' \iff w_I\ow\cdot \mu\le w_I\ow'\cdot \mu \iff w_I\ow\le_{[\lambda]} w_I\ow'.
\]
Note that we have $[\lambda]=[\mu]=[\mu']=[\lambda']$. By Corollary \ref{keyequiv}, we get
\[
{}^IP_{x,\ow}^{\mu}(1)\neq 0\iff w_Ix\le_{[\lambda]} w_I\ow \implies w_Ix\le_{[\lambda]} w_I\ow' \iff w_Ix\le_{[\lambda']} w_I\ow'\iff {}^IP_{x,\ow'}^{\mu'}(1)\neq 0.
\]
Since $L(\lambda)$ is a Kostant module, we get ${}^IP_{x,\ow}^{\mu}(1)\le 1$ by Theorem \ref{KLpoly}, Theorem \ref{paraKL} and Lemma \ref{Kostantmod}. This implies that ${}^IP_{x,\ow}^{\mu}(1)\le {}^IP_{x,\ow'}^{\mu'}(1)$.
Therefore,
\[
\bigoplus_{x\in {}^I W_{[\lambda]}}{}^IP_{x,\ow}^{\mu}(1)F(w_Ix\cdot\mu+\rho(\mathfrak{u})) \ \text{
is isomorphic to a quotient of } \bigoplus_{x\in {}^I W_{[\lambda]}}{}^IP_{x,\ow'}^{\mu'}(1)F(w_Ix\cdot\mu+\rho(\mathfrak{u}))
\]
as an $\mathfrak{l}$-module.
Since 
\[
H_D(L(\lambda))\cong \bigoplus_{x\in {}^I W_{[\lambda]}}{}^IP_{x,\ow}^{\mu}(1)F(w_Ix\cdot\mu+\rho(\mathfrak{u}))
\]
and
\[
H_D(L(\lambda'))
\cong \bigoplus_{x\in {}^I W_{[\lambda']}}{}^IP_{x,\ow'}^{\mu'}(1)F(w_Ix\cdot\mu'+\rho(\mathfrak{u}))
=\bigoplus_{x\in {}^I W_{[\lambda]}}{}^IP_{x,\ow'}^{\mu'}(1)F(w_Ix\cdot\mu+\rho(\mathfrak{u}))
\]
as $\mathfrak{l}$-modules,
we get
$H_D(L(\lambda))$ is isomorphic to a quotient of $H_D(L(\lambda'))$ as an $\mathfrak{l}$-module.
\end{proof}

\begin{remark}
Suppose $M,M',N,N'$ are $\mathfrak{l}$-modules and $P'$ is an $\mathfrak{l}$-submodule of $N'$. If $M\cong M'$, $N\cong N'$ as $\mathfrak{l}$-modules and $M'\cong N'/P'$, then $M\cong N/P$ as an $\mathfrak{l}$-module for some $\mathfrak{l}$-submodule $P$ of $N$. The reason is as follows: we have a surjective $\mathfrak{l}$-module homomorphism $\psi: N\rightarrow N'\rightarrow N'/P'\cong M'\rightarrow M$. By the First Isomorphism Theorem, we get $M\cong N/P$ as an $\mathfrak{l}$-module, where $P=\ker \psi$.
\end{remark}

\section{Simplicity criterion for parabolic Verma modules}
\label{S:6}

The Verma-BGG Theorem is the key ingredient in proving the simplicity criterion for Verma modules (see \cite[Theorem 4.8]{HJ}) and the simplicity criterion for parabolic Verma modules with regular infinitesimal character (see \cite[Theorem 9.12]{HJ}). The algebraic parameterizations of $H_D(L(\lambda))$ suggest that we can use Dirac cohomology to give a new proof of the simplicity criterion for Verma modules and derive a new simplicity criterion for parabolic Verma modules with regular infinitesimal character.

\subsection{Simplicity criterion for Verma modules}

\begin{proposition} [{\cite[Proposition 4.11]{HX}}]
	\label{HDVerma}
Suppose that $M_I(\lambda)$ is a parabolic Verma module with highest weight $\lambda\in \Lambda_I^+$. Then, there is an $\mathfrak{l}$-module isomorphism
$
H_D(M_I(\lambda))\cong F(\lambda)\otimes \mathbb{C}_{\rho(\mathfrak{u})}.
$
\end{proposition}

Before giving a new proof of the simplicity criterion, we need two lemmas.

\begin{lemma} [{See \cite[Proposition 3.5]{JL}}] \label{initial}
Suppose $\lambda\in\mathfrak{h}^*$, $J\subseteq \Delta_{[\lambda]}$ and $w\in{}^JW_{[\lambda]}$ has a reduced expression $s_{i_1}\cdots s_{i_r}$.
% with $i_m\in\{1,\cdots,l\}$. 
Then the ``initial segment" $s_{i_1}\cdots s_{i_j}$ belongs to $ {}^JW_{[\lambda]}$ for each $j=1,\cdots,r$. 
\end{lemma}

\begin{proof}
	Since $J\subseteq \Delta_{[\lambda]}$, we can replace $W$ and $I$ in \cite[Proposition 3.5]{JL} by $W_{[\lambda]}$ and $J$, respectively. 
\end{proof}

\begin{lemma}\label{ParaKL=1}
Let $\lambda\in\mathfrak{h}^*$ and $w\in W_{[\lambda]}^{\Sigma_\mu}$.
If $w=s_{i_1}\cdots s_{i_r}$ is a reduced expression, then
$P_{s_{i_1}w,w}^{{\Sigma_\mu}}(q)=1$. 
\end{lemma}

\begin{proof}
Since $w\in W_{[\lambda]}^{\Sigma_\mu}$, 
we have $w^{-1}\in {}^{\Sigma_\mu} W_{[\lambda]}$. Since $\Sigma_\mu\subseteq \Delta_{[\lambda]}$, we get $w^{-1}s_{i_1}\in {}^{\Sigma_\mu} W_{[\lambda]}$ by Lemma \ref{initial}.
Then by Theorem \ref{KLVparaKL}, we get $P_{s_{i_1}w,w}^{{\Sigma_\mu}}(q)
=P_{w^{-1}s_{i_1},w^{-1}}^{[\lambda],{\Sigma_\mu},q}(q)$.
Now consider Proposition \ref{recursion} with $u=w^{-1}s_{i_1}$, $v=w^{-1}$, $s=s_{i_1}\in D(w^{-1})=D(v)$ and $J=\Sigma_\mu$.
Since $w^{-1}s_{i_1}<_{[\lambda]}w^{-1}\in {}^{\Sigma_\mu} W_{[\lambda]}$, we get $u<_{[\lambda]}us\in {}^{\Sigma_\mu} W_{[\lambda]}$.
Thus by Theorem \ref{paraKL}, we deduce that the first term in the recursion is 
\[
\tilde{P}
=qP^{[\lambda],{\Sigma_\mu},q}_{us,vs}(q)+P^{[\lambda],{\Sigma_\mu},q}_{u,vs}(q)
=qP^{[\lambda],{\Sigma_\mu},q}_{w^{-1},w^{-1}s_{i_1}}(q)+P^{[\lambda],{\Sigma_\mu},q}_{w^{-1}s_{i_1},w^{-1}s_{i_1}}(q)
=P^{[\lambda],{\Sigma_\mu},q}_{w^{-1}s_{i_1},w^{-1}s_{i_1}}(q)
=1.
\]
The second term in the recursion is a sum over
\begin{align*}
    {\{u\le_{[\lambda]} x\le_{[\lambda]} vs: xs<_{[\lambda]}x\}}
    ={\{w^{-1}s_{i_1}\le_{[\lambda]} x\le_{[\lambda]} w^{-1}s_{i_1}: xs_{i_1}<_{[\lambda]}x\}}=\emptyset
\end{align*}
since $w^{-1}s_{i_1}<_{[\lambda]}w^{-1}$.
Therefore this term must be zero.
Hence
$P_{s_{i_1}w,w}^{{\Sigma_\mu}}(q)
=1+0=1$.
\end{proof}

We are now able to prove Theorem \ref{Verma}.
This result is  \cite[Theorem 4.8]{HJ}, but our method of proof using Dirac cohomology is new.

\begin{theorem} [Theorem \ref{Verma}] 
	\label{Verma'}
Let $\lambda\in\h^*$. Then 
$M(\lambda)\cong L(\lambda)$ as an $\mathfrak{g}$-module if and only if $\lambda $ is an antidominant weight.
\end{theorem}

\begin{proof}
	Suppose $M(\lambda)\cong L(\lambda)$ as an $\mathfrak{g}$-module and
	recall that $\mu$ is the antidominant weight in $W_{[\lambda]}\cdot \lambda$. Then by the definition of $W_{[\lambda]}$, we have $\mu-\lambda\in\Lambda_r$ and hence $L(\lambda)\in \cO_\mu$. By Lemma \ref{simplemod} with $I=\emptyset$, we have $\lambda=\ow\cdot\mu$ for some $\ow\in W_{[\lambda]}^{\Sigma_\mu}$. 
	To show that $\lambda$ is antidominant, it suffices to show that $\ow=e$.
	
	Assume $\ow\neq e$ on contrary. Then $\ow$ has a reduced expression $s_{i_1}\cdots s_{i_r}$ with $r\ge 1$, and Lemma \ref{multilemma} with $I=\emptyset$ and Lemma \ref{ParaKL=1}
	imply that $P_{\ow,\ow}^{\Sigma_\mu}(1)=1$ and $P_{s_{i_1}\ow,\ow}^{\Sigma_\mu}(1)=1$. Note that $\ow, s_{i_1}\ow\in W_{[\lambda]}^{\Sigma_\mu}$. Let $\mathfrak{n}=\bigoplus_{\alpha>0}\mathfrak{g}_\alpha$. Then by Theorems \ref{maindecom} and \ref{determine'} with $I=\emptyset$, and Proposition \ref{HDVerma} with $I=\emptyset$, it holds that as an $\mathfrak{h}$-module,
	\begin{align*}
	F(\ow\cdot\mu+\rho(\mathfrak{n}))
	\oplus F(s_{i_1}\ow\cdot\mu+\rho(\mathfrak{n}))
	&\subseteq H_D(L(\lambda))\\
	&\cong H_D(M(\lambda))\\
	&\cong 
	F(\ow\cdot\mu)\otimes\mathbb{C}_{\rho(\mathfrak{n})}\\
	&\cong 
	F(\ow\cdot\mu+\rho(\mathfrak{n})).
	\end{align*}
	This implies that $\dim F(s_{i_1}\ow\cdot\mu+\rho(\mathfrak{n}))=0$,
	which is a contradiction. Thus $\ow=e$, which implies that $\lambda=\mu$ for some antidominant weight $\mu$.
	
	To show the converse, 
	suppose $\lambda$ is an antidominant weight. Then $M(\lambda)\in \cO_\lambda$. All composition factors of $M(\lambda)$ are of the form $L(\eta)$ with $\eta\le \lambda$, and in particular with $L(\eta)\in \cO_\lambda$. Then by Lemma \ref{simplemod} with $I=\emptyset$,  all composition factors of $M(\lambda)$ are of the form $L(x\cdot\lambda)$ with $x\cdot\lambda\le \lambda$ and $x\in W_{[\lambda]}^{\Sigma_\lambda}$.
	Since $e\le_{[\lambda]} x$ for all $x\in W_{[\lambda]}^{\Sigma_\lambda}$, we get $\lambda \le x\cdot \lambda$ for all $x\in W_{[\lambda]}^{\Sigma_\lambda}$ by Lemma \ref{HJtypo}. Thus only $L(\lambda)$ can occur as a composition factor.
	%	But it occurs just once, so 
	By \cite[Theorem 1.2]{HJ}, we have $\dim M(\lambda)_{\lambda}=1$. This implies that $[M(\lambda),L(\lambda)]=1$ and hence $M(\lambda) \cong L(\lambda)$ as an $\mathfrak{g}$-module.
\end{proof}
%
%\begin{remark}
%The above result gives a new proof of \cite[Theorem 4.8]{HJ}.
%\end{remark}

\subsection{Simplicity criterion for parabolic Verma modules with regular infinitesimal character}

%The following is well-known:
%
%\begin{lemma} [{See \cite[Proposition 5.13]{BK2}}] \label{uniqdecom}
%Every element $w\in W_{[\lambda]}$ can be uniquely written in the form $w=uv$, where $u\in W_I$ and $v\in {}^IW_{[\lambda]}$. Furthermore, if $w=uv$ is such decomposition, then $\ell_{[\lambda]}(w)=\ell_{[\lambda]}(u)+\ell_{[\lambda]}(v)$.
%\end{lemma}
%
%\begin{proof}
%Since $I\subseteq\Delta_{[\lambda]}$, we have $W_I=\left(W_{[\lambda]}\right)_I$.
%The remark following Lemma \ref{bruhatorders} implies that ${}^IW_{[\lambda]}$ is the set of minimal length right coset representatives of $W_I$ in $ W_{[\lambda]}$.
%We can then replace $W$ in \cite[Proposition 5.13]{BK2} by $W_{[\lambda]}$. 
%\end{proof}

\begin{lemma}\label{bruhat}
Let $\lambda\in\Lambda_I^+$. If $x,w\in {}^IW_{[\lambda]}$ then
$x\le_{[\lambda]} w\iff w_Ix\le_{[\lambda]} w_Iw$.
\end{lemma}

\begin{proof}
%	Let $P_{u,v}^{[\lambda]}(q)$ and $P_{u,v}^{[\lambda],I,-1}(q)$ be the Kazhdan-Lusztig polynomial of $u,v\in W_{[\lambda]}$ and the parabolic Kazhdan-Lusztig polynomial of $u, v\in{}^IW_{[\lambda]}$ of type $-1$, respectively.
%	
	The proof of Corollary \ref{keyequiv} shows that for all $x,w\in {}^IW_{[\lambda]}$, we have $P_{w_Ix,w_Iw}^{[\lambda]}(q)\neq 0\iff w_Ix\le_{[\lambda]}w_Iw$.
	Note that we have Theorem \ref{paraKL} and the remark following Proposition \ref{paraKL-KL}. Then by a similar argument, for all $x,w\in {}^IW_{[\lambda]}$, we have $P_{x,w}^{[\lambda],I,-1}(q)\neq 0\iff x\le_{[\lambda]}w$.
	By Proposition \ref{paraKL-KL}, we have $P_{x,w}^{[\lambda],I,-1}(q)=P_{w_Ix,w_Iw}^{[\lambda]}(q)$. The claim follows.
\end{proof}

	%This follows as a straightforward exercise from the subword characterization of the Bruhat order; we omit the details.
	
%\begin{proof}
%Suppose $x\le_{[\lambda]} w$, let $w=s_{j_1}\cdots s_{j_r}$ be a reduced expression. Then $x$ has a reduced expression $x=s_{i_1}\cdots s_{i_l}$ with $l\le r$ and $\{i_1,\cdots, i_l\}\subseteq \{j_1,\cdots, j_r\}$. Let $w_I=t_1\cdots t_k$ be a reduced expression.
%    %
%    By Lemma \ref{uniqdecom}, $t_1\cdots t_ks_{i_1}\cdots s_{i_l}$ and $t_1\cdots t_ks_{j_1}\cdots s_{j_r}$ are reduced expressions for $w_Ix$ and $w_Iw$, respectively, since $\ell_{[\lambda]}(w_Ix)=\ell_{[\lambda]}(w_I)+\ell_{[\lambda]}(x)$ and $\ell_{[\lambda]}(w_Iw)=\ell_{[\lambda]}(w_I)+\ell_{[\lambda]}(w)$.
%    Therefore $w_Ix\le_{[\lambda]} w_Iw$.
%    
%    Conversely, suppose $w_Ix\le_{[\lambda]} w_Iw$. Let $w_I=t_1\cdots t_k$, $x=s_{1}'\cdots s_{p}'$, $w=s_{i_1}\cdots s_{i_q}$ be reduced expressions.
%By Lemma \ref{uniqdecom}, $w_Ix=t_1\cdots t_ks_{1}'\cdots s_{p}'$ and $w_Iw=t_1\cdots t_ks_{i_1}\cdots s_{i_q}$ are reduced expressions,
%since $\ell_{[\lambda]}(w_Ix)=\ell_{[\lambda]}(w_I)+\ell_{[\lambda]}(x)$ and $\ell_{[\lambda]}(w_Iw)=\ell_{[\lambda]}(w_I)+\ell_{[\lambda]}(w)$.  
%    Now $t_1\cdots t_ks_{1}'\cdots s_{p}'\le_{[\lambda]} t_1\cdots t_ks_{i_1}\cdots s_{i_q}$ implies that $t_1\cdots t_ks_{1}'\cdots s_{p}'$ is a subword of $t_1\cdots t_ks_{i_1}\cdots s_{i_q}$, so $s_{1}'\cdots s_{p}'$ is a subword of $s_{i_1}\cdots s_{i_q}$ and hence $x\le_{[\lambda]} w$.
%\end{proof}

Finally, we prove Theorem \ref{ParaVerma}.

\begin{theorem} [Theorem \ref{ParaVerma}]
	\label{ParaVerma'}
	Let $\lambda\in\Lambda_I^+\cap \mathcal{R}$.
	Then
	$M_I(\lambda)\cong L(\lambda)$ as an $\mathfrak{g}$-module if and only if $\lambda=w_I\cdot\nu$ for some antidominant weight $\nu$.
\end{theorem}

This result partially generalizes Theorem \ref{Verma'}.

\begin{proof}
Suppose $M_I(\lambda)\cong L(\lambda)$ as an $\mathfrak{g}$-module and
recall that $\mu$ is the antidominant weight in $W_{[\lambda]}\cdot \lambda$. Then by the definition of $W_{[\lambda]}$, we have $\mu-\lambda\in\Lambda_r$ and hence $L(\lambda)\in \cO^\p_\mu$. By Lemma \ref{simplemod}, $\lambda=w_I\ow\cdot\mu$ for some $\ow\in {}^IW_{[\lambda]}$. 
It suffices to show that $\ow=e$.

Assume $\ow\neq e$ on contrary. Then by Lemma \ref{multilemma}, we get ${}^IP_{\ow,\ow}^{\mu}(1)=1$, and by Lemma \ref{bruhat} and Corollary \ref{keyequiv}, we get ${}^IP_{e,\ow}^{\mu}(1)\neq 0$ since $e\le_{[\lambda]} \ow$. Note that $\ow,e\in {}^IW_{[\lambda]}$. Then by Theorems \ref{maindecom} and \ref{determine'}, and Proposition \ref{HDVerma}, it holds that as an $\mathfrak{l}$-module,
\begin{align*}
F(w_I\ow\cdot\mu+\rho(\mathfrak{u}))
\oplus {}^IP_{e,\ow}^{\mu}(1)F(w_I\cdot\mu+\rho(\mathfrak{u}))
&\subseteq H_D(L(\lambda))\\
&\cong H_D(M_I(\lambda))\\
&\cong 
F(w_I\ow\cdot\mu)\otimes\mathbb{C}_{\rho(\mathfrak{u})}\\
&\cong 
F(w_I\ow\cdot\mu+\rho(\mathfrak{u})).
\end{align*}
This implies that $\dim F(w_I\cdot\mu+\rho(\mathfrak{u}))=0$, which is again a contradiction. Thus $\ow=e$, which implies that $\lambda=w_I\cdot\mu$ for some antidominant weight $\mu$.

Conversely, suppose $\lambda=w_I\cdot\nu$ for some antidominant weight $\nu$. 
Since $\lambda\in\Lambda_I^+$, we get $w_I\in W_I\subseteq W_{[\lambda]}$ by the remark following Corollary \ref{category}. This implies that $\nu$ is the antidominant weight in $W_{[\lambda]}\cdot\lambda$, i.e., $\nu=\mu$. Then by the definition of $W_{[\lambda]}$, we have $\mu-\lambda\in\Lambda_r$ and hence $M_I(\lambda)\in \cO^\p_\mu$. All composition factors of $M_I(\lambda)$ are of the form $L(\eta)$ with $\eta\le \lambda$. By Proposition \ref{9.3}, we have $L(\eta)\in \cO^\p$. This implies that $L(\eta)\in \cO^\p_\mu$ since $[\eta]=[\lambda]=[\mu]$. Then by Lemma \ref{simplemod}, all composition factors of $M_I(\lambda)$ are of the form $L(w_Ix\cdot\mu)$ with $w_Ix\cdot\mu\le \lambda$ and $x\in {}^IW_{[\lambda]}$.
Since $e\le_{[\lambda]} x$ for all $x\in {}^IW_{[\lambda]}$,  Lemmas \ref{bruhat} and  \ref{HJtypo} imply that $\lambda=w_I\cdot \nu=w_I\cdot \mu\le w_Ix\cdot \mu$ for all $x\in {}^IW_{[\lambda]}$. Thus only $L(\lambda)$ can occur as a composition factor.
%But it again occurs just once, so 
By \cite[Theorem 1.2]{HJ}, we have $\dim M_I(\lambda)_{\lambda}=1$. This implies that $[M_I(\lambda),L(\lambda)]=1$ and hence $M_I(\lambda)\cong L(\lambda)$ as an $\mathfrak{g}$-module.
\end{proof}

%\begin{remark}
%The above result partially generalize Theorem \ref{Verma'}.
%\end{remark}

Let $\Psi := \Phi \backslash \Phi_I$ and 
$\Psi^+_\lambda := \{\beta \in \Psi^+ : \langle \lambda + \rho, \beta^\lor\rangle \in \mathbb{Z}^{>0}\}$.
It is interesting to compare the above result with Jantzen's simplicity criterion:

\begin{theorem} [{See \cite[Theorem 9.12 and Corollary 9.13]{HJ}}] 
Let $\lambda \in \Lambda_I^+$ and $\lambda$ be regular. 
Then $M_I(\lambda)$ is simple if and only if $\Psi^+_\lambda = \emptyset$.
\end{theorem}

This leads to Corollary \ref{nontrivial}. 
\begin{corollary} [Corollary \ref{nontrivial}]
	\label{nontrivial'}
	Let $\lambda\in\Lambda_I^+\cap \mathcal{R}$. Then 
	$\Psi^+_\lambda = \emptyset$ if and only if $\lambda=w_I\cdot\nu$ for some antidominant weight $\nu$.
	%, where $\Psi^+_\lambda := \{\beta \in \Psi^+ : \langle \lambda + \rho, \beta^\lor\rangle \in \mathbb{Z}^{>0}\}$ and $\Psi := \Phi \backslash \Phi_I$.
\end{corollary}



%% References

\bibliographystyle{abbrv}
\bibliography{references}

\begin{thebibliography}{10}

\bibitem{BGG}
I.~Bernstein, I.~M. Gelfand, and S.~I. Gelfand.
\newblock Structure of representations generated by vectors of highest weight.
\newblock {\em Functional Analysis and its Applications}, 5(1):1--8, 1971.

\bibitem{BGG1}
I.~Bernstein, I.~M. Gelfand, and S.~I. Gelfand.
\newblock Differential operators on the base affine space and a study of
  $\mathfrak{g}$-modules.
\newblock {\em Lie groups and their representations (Proc. Summer School,
  Bolyai J{\'a}nos Math. Soc., Budapest, 1971)}, pages 21--64, 1975.

\bibitem{BGG2}
I.~Bernstein, I.~M. Gelfand, and S.~I. Gelfand.
\newblock Category of $\mathfrak{g}$-modules.
\newblock {\em Functional Analysis and its Applications}, 10(2):87--92, 1976.

\bibitem{BB}
J.~Bernstein and A.~Beilinson.
\newblock Localisation de $\mathfrak{g}$-modules.
\newblock {\em CR Acad. Sci. Paris}, 292:15--18, 1981.

\bibitem{AF}
A.~Bjorner and F.~Brenti.
\newblock {\em Combinatorics of Coxeter groups}, volume 231.
\newblock Springer Science \& Business Media, 2006.

\bibitem{BBMH}
B.~D. Boe and M.~Hunziker.
\newblock Kostant modules in blocks of category {$\mathcal{O}_S$}.
\newblock {\em Communications in Algebra}, 37(1):323--356, 2009.

\bibitem{BBDN}
B.~D. Boe and D.~K. Nakano.
\newblock Representation type of the blocks of category {$\mathcal{O}_S$}.
\newblock {\em Advances in Mathematics}, 196(1):193--256, 2005.

\bibitem{Fran0}
F.~Brenti.
\newblock {Kazhdan-Lusztig} and {R-polynomials}, {Young's lattice}, and {Dyck
  partitions}.
\newblock {\em Pacific Journal of Mathematics}, 207(2):257--286, 2002.

\bibitem{BJLKM}
J.-L. Brylinski and M.~Kashiwara.
\newblock {Kazhdan-Lusztig} conjecture and holonomic systems.
\newblock {\em Inventiones Mathematicae}, 64(3):387--410, 1981.

\bibitem{CLCD}
L.~G. Casian and D.~H. Collingwood.
\newblock The {Kazhdan-Lusztig} conjecture for generalized {Verma} modules.
\newblock {\em Mathematische Zeitschrift}, 195(4):581--600, 1987.

\bibitem{VD2}
V.~V. Deodhar.
\newblock Some characterizations of {Bruhat} ordering on a {Coxeter} group and
  determination of the relative {M{\"o}bius function}.
\newblock {\em Inventiones Mathematicae}, 39(2):187--198, 1977.

\bibitem{VD}
V.~V. Deodhar.
\newblock On some geometric aspects of {Bruhat} orderings {II}. {The} parabolic
  analogue of {Kazhdan-Lusztig} polynomials.
\newblock {\em Journal of Algebra}, 111(2):483--506, 1987.

\bibitem{EHP}
T.~J. Enright, M.~Hunziker, and W.~A. Pruett.
\newblock Diagrams of {Hermitian} type, highest weight modules, and syzygies of
  determinantal varieties.
\newblock In {\em Symmetry: Representation Theory and Its Applications}, pages
  121--184. Springer, 2014.

\bibitem{TEBS}
T.~J. Enright and B.~Shelton.
\newblock {\em Categories of highest weight modules: applications to classical
  {Hermitian} symmetric pairs: applications to classical {Hermitian} symmetric
  pairs}, volume 367.
\newblock American Mathematical Soc., 1987.

\bibitem{HP1}
J.-S. Huang and P.~Pand{\v{z}}i{\'c}.
\newblock Dirac cohomology, unitary representations and a proof of a conjecture
  of {Vogan}.
\newblock {\em Journal of the American Mathematical Society}, 15(1):185--202,
  2002.

\bibitem{HPV}
J.-S. Huang, P.~Pand{\v{z}}i{\'c}, and D.~Vogan.
\newblock On classifying unitary modules by their {Dirac} cohomology.
\newblock {\em Science China Mathematics}, 60(11):1937--1962, 2017.

\bibitem{HX}
J.-S. Huang and W.~Xiao.
\newblock Dirac cohomology of highest weight modules.
\newblock {\em Selecta Mathematica}, 18(4):803--824, 2012.

\bibitem{HJ}
J.~E. Humphreys.
\newblock {\em Representations of semisimple Lie algebras in the BGG category
  $\mathcal{O}$}, volume~94.
\newblock American Mathematical Soc., 2008.

\bibitem{RSI}
R.~S. Irving.
\newblock Singular blocks of the category $\mathcal{O}$.
\newblock {\em Mathematische Zeitschrift}, 204(1):209--224, 1990.

\bibitem{KL}
D.~Kazhdan and G.~Lusztig.
\newblock Representations of {Coxeter} groups and {Hecke} algebras.
\newblock {\em Inventiones Mathematicae}, 53(2):165--184, 1979.

\bibitem{AK}
A.~W. Knapp.
\newblock {\em Lie groups beyond an introduction}, volume 140.
\newblock Springer Science \& Business Media, 2013.

\bibitem{BK}
B.~Kostant.
\newblock A cubic {Dirac} operator and the emergence of {Euler} number
  multiplets of representations for equal rank subgroups.
\newblock {\em Duke Mathematical Journal}, 100(3):447--501, 1999.

\bibitem{BK1}
B.~Kostant.
\newblock Dirac cohomology for the cubic {Dirac} operator.
\newblock In {\em Studies in Memory of Issai Schur}, pages 69--93. Springer,
  2003.

\bibitem{JL}
J.~Lepowsky.
\newblock A generalization of the {Bernstein-Gelfand-Gelfand} resolution.
\newblock {\em Journal of Algebra}, 49(2):496--511, 1977.

\bibitem{WS1}
W.~Soergel.
\newblock $\mathfrak{n}$-cohomology of simple highest weight modules on walls
  and purity.
\newblock {\em Inventiones Mathematicae}, 98(3):565--580, 1989.

\bibitem{WS}
W.~Soergel.
\newblock Kategorie $\mathcal{O}$, perverse {Garben} und {Moduln} {\"u}ber den
  {Koinvarianten} zur {Weylgruppe}.
\newblock {\em Journal of the American Mathematical Society}, 3(2):421--445,
  1990.

\bibitem{DV}
D.~A. Vogan~Jr.
\newblock Irreducible characters of semisimple {Lie} groups {II.} {The
  Kazhdan-Lusztig} conjectures.
\newblock {\em Duke Mathematical Journal}, 46(4):805--859, 1979.

\end{thebibliography}

\end{document}